\documentclass[a4paper,11pt]{amsart}
\usepackage{amsmath,amssymb,amsthm,graphicx,multirow}
\usepackage[all]{xy}
\usepackage[top=30truemm,bottom=30truemm,left=25truemm,right=25truemm]{geometry}

\allowdisplaybreaks[1]

\newtheorem{Th}{Theorem}[section]
\newtheorem{Cor}[Th]{Corollary}
\newtheorem{Prop}[Th]{Proposition}
\newtheorem{Lem}[Th]{Lemma}
\newtheorem*{Th*}{Theorem}

\theoremstyle{definition}
\newtheorem{Def}[Th]{Definition}
\newtheorem{Cond}[Th]{Condition}
\newtheorem{Ex}[Th]{Example}

\theoremstyle{remark}
\newtheorem{Rem}[Th]{Remark}

\newcommand{\Z}{\mathbb{Z}}
\newcommand{\C}{\mathbb{C}}

\newcommand{\rb}{\mathrm{b}}
\newcommand{\rT}{\mathrm{T}}
\newcommand{\np}{\mathrm{np}}
\newcommand{\pr}{\mathrm{pr}}
\newcommand{\ppr}{\mathrm{pp}}
\newcommand{\reg}{\mathrm{reg}}
\newcommand{\pin}{\mathrm{pi}}

\newcommand{\cC}{\mathcal{C}}
\newcommand{\cD}{\mathcal{D}}

\newcommand{\cF}{\mathcal{F}}
\newcommand{\cH}{\mathcal{H}}
\newcommand{\cT}{\mathcal{T}}
\newcommand{\cS}{\mathcal{S}}
\newcommand{\cW}{\mathcal{W}}
\newcommand{\cX}{\mathcal{X}}
\newcommand{\cY}{\mathcal{Y}}

\newcommand{\sD}{\mathsf{D}}
\newcommand{\sF}{\mathsf{F}}
\newcommand{\sK}{\mathsf{K}}
\newcommand{\sL}{\mathsf{L}}
\newcommand{\sR}{\mathsf{R}}
\newcommand{\sT}{\mathsf{T}}
\newcommand{\sW}{\mathsf{W}}

\newcommand{\rone}{\textbf{1}}
\newcommand{\rtwo}{\textbf{2}}
\newcommand{\rthr}{\textbf{3}}

\newcommand{\bc}{\boldsymbol{c}}
\newcommand{\be}{\boldsymbol{e}}
\newcommand{\bg}{\boldsymbol{g}}
\newcommand{\bh}{\boldsymbol{h}}
\newcommand{\bp}{\boldsymbol{p}}
\newcommand{\bC}{\boldsymbol{C}}
\newcommand{\bD}{\boldsymbol{D}}
\newcommand{\bG}{\boldsymbol{G}}

\DeclareMathOperator{\module}{\mathsf{mod}}
\DeclareMathOperator{\proj}{\mathsf{proj}}
\DeclareMathOperator{\inj}{\mathsf{inj}}
\DeclareMathOperator{\add}{\mathsf{add}}

\DeclareMathOperator{\Hom}{\mathsf{Hom}}
\DeclareMathOperator{\rad}{\mathsf{rad}}
\DeclareMathOperator{\End}{\mathsf{End}}
\DeclareMathOperator{\Ext}{\mathsf{Ext}}
\DeclareMathOperator{\Filt}{\mathsf{Filt}}
\DeclareMathOperator{\Fac}{\mathsf{Fac}}
\DeclareMathOperator{\Sub}{\mathsf{Sub}}
\DeclareMathOperator{\ind}{\mathsf{ind}}
\DeclareMathOperator{\ann}{\mathsf{ann}}
\DeclareMathOperator{\soc}{\mathsf{soc}}
\DeclareMathOperator{\dimension}{\mathsf{dim}}
\DeclareMathOperator{\Cone}{\mathsf{Cone}}
\DeclareMathOperator{\supp}{\mathsf{supp}}

\DeclareMathOperator{\twosilt}{\mathsf{2-silt}}
\DeclareMathOperator{\twocosilt}{\mathsf{2-cosilt}}
\DeclareMathOperator{\inttstr}{\mathsf{int-t-str}}

\DeclareMathOperator{\twosmc}{\mathsf{2-smc}}
\DeclareMathOperator{\sttilt}{\mathsf{s\tau-tilt}}
\DeclareMathOperator{\stitilt}{\mathsf{s}\tau^{--1}\mathsf{-tilt}}
\DeclareMathOperator{\trigid}{\mathsf{\tau-rigid}}
\DeclareMathOperator{\tirigid}{\tau^{--1}\mathsf{-rigid}}
\DeclareMathOperator{\brick}{\mathsf{brick}}
\DeclareMathOperator{\sbrick}{\mathsf{sbrick}}
\DeclareMathOperator{\sbr}{\mathsf{sb}}
\DeclareMathOperator{\fLsbrick}{\mathsf{f_L-sbrick}}
\DeclareMathOperator{\fRsbrick}{\mathsf{f_R-sbrick}}
\DeclareMathOperator{\tors}{\mathsf{tors}}
\DeclareMathOperator{\torf}{\mathsf{torf}}
\DeclareMathOperator{\wide}{\mathsf{wide}}
\DeclareMathOperator{\fwide}{\mathsf{f-wide}}
\DeclareMathOperator{\ftors}{\mathsf{f-tors}}
\DeclareMathOperator{\fLwide}{\mathsf{f_L-wide}}
\DeclareMathOperator{\ftorf}{\mathsf{f-torf}}
\DeclareMathOperator{\fRwide}{\mathsf{f_R-wide}}

\DeclareMathOperator{\Coker}{\mathsf{Coker}}
\DeclareMathOperator{\Ker}{\mathsf{Ker}}
\DeclareMathOperator{\Image}{\mathsf{Im}}
\DeclareMathOperator{\mini}{\mathsf{min}}
\DeclareMathOperator{\maks}{\mathsf{max}}

\renewcommand{\Im}{\Image}
\renewcommand{\mod}{\module}
\renewcommand{\dim}{\dimension}
\renewcommand{\min}{\mini}
\renewcommand{\max}{\maks}

\renewcommand{\Phi}{\varPhi}
\renewcommand{\phi}{\varphi}

\title{Semibricks}
\date{\today}
\author{Sota Asai}
\address{Graduate School of Mathematics, Nagoya University, Furo-cho,
Chikusa-ku, Nagoya-shi, Aichi-ken, 464-8602, Japan}
\email{m14001v@math.nagoya-u.ac.jp}

\begin{document}

\begin{abstract}
In representation theory of finite-dimensional algebras,
(semi)bricks are a generalization of (semi)simple modules,
and they have long been studied.
The aim of this paper is to study semibricks
from the point of view of $\tau$-tilting theory.
We construct canonical bijections between
the set of support $\tau$-tilting modules,
the set of semibricks satisfying a certain finiteness condition,
and the set of 2-term simple-minded collections.
In particular, we unify Koenig--Yang bijections and
Ingalls--Thomas bijections generalized by Marks--\v{S}t\!'ov\'{i}\v{c}ek,
which involve several important notions 
in the derived categories and the module categories.
We also investigate connections between our results and 
two kinds of reduction theorems of $\tau$-rigid modules 
by Jasso and Eisele--Janssens--Raedschelders.
Moreover, we study semibricks over Nakayama algebras and tilted algebras in detail
as examples.
\end{abstract}

\maketitle

\tableofcontents

\setcounter{section}{-1}

\section{Introduction}

In representation theory of a finite-dimensional algebra $A$ over a field $K$, 
the notion of (semi)simple modules is fundamental.
By Schur's Lemma, they satisfy the following properties:
\begin{itemize}
\item the endomorphism ring of a simple module is a division algebra,
\item there exists no nonzero homomorphism between two nonisomorphic simple modules.
\end{itemize} 

A module $M$ in $\mod A$ is called a \textit{brick} 
if its endomorphism ring is a division algebra.
This notion is a generalization of simple modules,
and it has long been studied in representation theory \cite{Ringel, Gabriel1, Gabriel2}.
Typical examples of bricks are given as 
preprojective modules and preinjective modules over
a finite-dimensional hereditary algebra.
Sometimes, it is useful to consider 
sets of isoclasses of pairwise Hom-orthogonal bricks.
We simply call them \textit{semibricks},
and define $\sbrick A$ as the set of semibricks.
It follows from a classical result by Ringel \cite{Ringel} 
that the semibricks $\cS$ correspond bijectively to the wide subcategories $\cW$
of $\mod A$,
that is, the subcategories 
which are closed under taking kernels, cokernels, and extensions.
Under this correspondence, $\cS$ is the set of  
the simple objects in $\cW$,
and $\cW$ consists of all the $A$-modules filtered by bricks in $\cS$.
Moreover, bricks and wide subcategories have close relationship with
ring epimorphisms and universal localizations
\cite{Stenstrom, Schofield, GL}.

In this paper, we assign a condition for semibricks.
We say that a semibrick $\cS$ is \textit{left finite}
if the smallest torsion class $\sT(\cS) \subset \mod A$
containing $\cS$ is functorially finite.
We write $\fLsbrick A$ for the set of left finite semibricks.
Right finite semibricks and the set $\fRsbrick A$ are defined dually
by using torsion-free classes.
We define left-finiteness and right-finiteness of 
wide subcategories in the same way as semibricks.
We write $\fLwide A$ (resp.\ $\fRwide A$) for the set of 
left (resp.\ right) finite wide subcategories 
of $\mod A$.
Clearly, Ringel's bijection is restricted to bijections 
$\fLsbrick A \to \fLwide A$ and $\fRsbrick A \to \fRwide A$.

Recently, Adachi--Iyama--Reiten \cite{AIR} obtained a bijection
from the set $\sttilt A$ of basic \textit{support $\tau$-tilting modules} in $\mod A$
to the set $\ftors A$ of functorially finite torsion classes in $\mod A$,
where $M$ is sent to $\Fac M$.
Support $\tau$-tilting modules are a generalization of tilting modules
and are defined by using the Auslander--Reiten translation $\tau$.
Adachi--Iyama--Reiten also proved that 
there are operations called \textit{mutations} of support $\tau$-tilting modules,
that is, constructing a new support $\tau$-tilting module by 
changing one indecomposable direct summand of a given one.
They showed that such mutations are nothing but the adjacency relations 
with respect to the inclusions of the corresponding torsion classes.

The following first main result shows 
that the support $\tau$-tilting modules also
correspond bijectively to the left finite semibricks,
where $\ind N$ denotes the set of isoclasses of indecomposable direct summands of $N$.
This is an extension of a result by Demonet--Iyama--Jasso \cite{DIJ}.

\begin{Th}[Theorem \ref{sttilt_fsbrick}]\label{intro_1}
We have the following maps,
where $B:=\End_A(M)$ in each case.
\begin{itemize}
\item[(1)]
There exists a surjection $\trigid A \to \fLsbrick A$ defined 
as $M \mapsto \ind(M/{\rad_B M})$. 
Under this map, the left finite semibricks for two $\tau$-rigid modules $M,M' \in \trigid A$ 
coincide if and only if $\Fac M=\Fac M'$.
\item[(2)]
There exists a bijection $\sttilt A \to \fLsbrick A$ defined 
as $M \mapsto \ind(M/{\rad_B M})$.
\end{itemize}
\end{Th}

For example, this map sends the progenerator $A$ to
the set of simple $A$-modules.
In the proof of Theorem \ref{intro_1}, 
we obtain a canonical bijection $\sT \colon \fLsbrick A \to \ftors A$,
see Proposition \ref{comm_3}.
Consequently, we recover Marks--\v{S}t\!'ov\'{i}\v{c}ek bijection
between $\fLwide A$ and $\ftors A$ \cite{MS}.
As an application of Theorem \ref{intro_1}, 
we have the following result.

\begin{Cor}[Corollary \ref{card_bound}]
If a semibrick $\cS$ is either left finite or right finite,
then $\# \cS \le n_A$ holds,
where $n_A$ is the number of isoclasses of simple $A$-modules.
\end{Cor}

This corollary does not generally hold for arbitrary semibricks.
For example, if $A$ is the Kronecker algebra,
then there exists a semibrick consisting of infinitely many bricks.

In Subsection \ref{label_mod},
we use the map in Theorem \ref{intro_1} 
to label the exchange quiver of $\sttilt A$ with bricks,
that is, to define a map from the set of arrows in the exchange quiver
to the set of bricks, see Definition \ref{label_mod_def}.

Subsection \ref{wide_mod} is devoted to further study of wide subcategories.
Our aim in this subsection is to describe 
each left finite wide subcategory $\cW$ of $\mod A$
as the module category $\mod A'$ of some finite-dimensional algebra $A'$.
We can take $M \in \sttilt A$ corresponding to $\cW$,
then we have an equivalnce $\Hom_A(M,?) \colon \Fac M \to \Sub_B DM$ 
by a Brenner--Butler type theorem, where $B:=\End_A(M)$.
We restrict this equivalence to $\cW \subset \Fac M$.
The following result shows that
we can regard $\cW$ naturally as $\mod B/\langle e \rangle$
given by a certain explicit idempotent $e \in B$.

\begin{Th}[Theorem \ref{wide_B}]\label{wide_B_intro}
The equivalence $\Hom_A(M,?) \colon \Fac M \to \Sub_B DM$ is restricted to 
an equivalence $\Hom_A(M,?) \colon \cW \to \mod B/\langle e \rangle$.
\end{Th}

In Subsection \ref{on_EJR},
we consider semibricks for a factor algebra $A/I$
with $I \subset A$ an ideal.
We clearly have an inclusion $\sbrick A/I \subset \sbrick A$, 
and we give the following sufficient condition
so that the equality $\sbrick A/I=\sbrick A$ holds:
$I$ is generated by some elements
belonging to the intersection of the center $Z(A)$ and the radical $\rad A$.
This condition on the ideal $I$ was originally considered by
Eisele--Janssens--Raedschelders \cite{EJR}.
In this situation, we also prove that $\fLsbrick A/I=\fLsbrick A$ 
and recover their canonical bijection $\sttilt A \to \sttilt A/I$.

\medskip

Furthermore, we apply semibricks to study the derived category $\sD^\rb(\mod A)$ 
in Section \ref{Section_smc}.
A \textit{simple-minded collection} in $\sD^\rb(\mod A)$ is 
a set of isoclasses of objects in $\sD^\rb(\mod A)$ 
satisfying the conditions in Schur's Lemma and some additional conditions,
see Definition \ref{smc_def} for details.
For our purpose, it is useful to consider simple-minded collections
which are \textit{2-term}, that is,
the $i$th cohomology $H^i(X)$ of every object $X$ in a simple-minded collection vanishes if $i \ne -1,0$.
A simple-minded collection in $\sD^\rb(\mod A)$ always consists of 
$n_A$ objects, 
and if it is 2-term,
each of the objects belongs to either $\mod A$ or $(\mod A)[1]$, see \cite{KY,BY}.
We write $\twosmc A$ for the set of 2-term simple-minded collections in $\sD^\rb(\mod A)$.
Our next main theorem gives bijections 
between the sets $\twosmc A$, $\fLsbrick A$, and $\fRsbrick A$.

\begin{Th}[Theorem \ref{smc_fsbrick} (1)]
There exist bijections 
\begin{align*}
? \cap \mod A \colon \twosmc A \to \fLsbrick A \quad \mathrm{and} \quad
?[-1] \cap \mod A \colon \twosmc A \to \fRsbrick A
\end{align*}
given by $\cX \mapsto \cX \cap \mod A$ and 
$\cX \mapsto \cX[-1] \cap \mod A$.
\end{Th}

Consequently, a 2-term simple-minded collection in $\sD^\rb(\mod A)$
is a union of a right finite semibrick shifted by $[1]$ and a left finite semibrick.

In recent years, thanks to many authors such as \cite{AIR,BY,IT,KY,MS}, 
canonical bijections between the sets of important objects containing 
the following ones have been discovered. 
Here, $\proj A$ (resp.\ $\inj A$) 
denotes the full subcategory of $\mod A$ 
consisting of all the projective (resp.\ injective) modules.

\begin{itemize}
\item[(a)]
The set $\fLsbrick A$ of left finite semibricks in $\mod A$.
\item[(a$'$)]
The set $\fRsbrick A$ of right finite semibricks in $\mod A$.
\item[(b)]
The set $\sttilt A$ of 
isoclasses of basic support $\tau$-tilting $A$-modules in $\mod A$.
\item[(b$'$)]
The set $\stitilt A$ of 
isoclasses of basic support $\tau^{-1}$-tilting $A$-modules in $\mod A$.
\item[(c)]
The set $\ftors A$ of functorially finite torsion classes in $\mod A$.
\item[(c$'$)]
The set $\ftorf A$ of functorially finite torsion-free classes in $\mod A$.
\item[(d)]
The set $\fLwide A$ of left finite wide subcategories of $\mod A$.
\item[(d$'$)]
The set $\fRwide A$ of right finite wide subcategories of $\mod A$.
\item[(e)]
The set $\twosmc A$ of 2-term simple-minded collections in $\sD^\rb(\mod A)$.
\item[(f)]
The set $\inttstr A$ of intermediate t-structures with length heart in $\sD^\rb(\mod A)$.
\item[(g)]
The set $\twosilt A$ of 2-term silting objects in $\sK^\rb(\proj A)$.
\item[(g$'$)]
The set $\twocosilt A$ of 2-term cosilting objects in $\sK^\rb(\inj A)$.
\end{itemize}

If $A$ is the path algebra of a Dynkin quiver $Q$,
Ingalls--Thomas \cite{IT} showed that 
there are also one-to-one correspondences between the above sets  
and the set of clusters in the cluster algebra for the quiver $Q$, which was
introduced by Fomin--Zelevinsky \cite{FZ}. 

Figure \ref{big_comm_intro} shows 
some of the known bijections
and our new bijections 
(arrows with labels in rectangles).
We prove the following result.

\begin{Th}[Theorem \ref{smc_fsbrick} (2)]
The diagram in Figure \ref{big_comm_intro} below is commutative 
and all the maps are bijective.
In this diagram,
$\cT \in \ftors A$ corresponds to $\cF \in \ftorf A$ 
if and only if $(\cT,\cF)$ is a torsion pair in $\mod A$.
\begin{figure}[h]
\begin{align*}
\begin{xy}
( 0, 36) *+{\stitilt A}   ="06",
(55, 36) *+{\ftorf A}     ="16",
(99, 36) *+{\fRsbrick A}  ="36",
( 0, 24) *+{\twocosilt A} ="04",
( 0, 12) *+{\twosilt A}   ="02",
(55, 18) *+{\inttstr A}   ="13",
(99, 18) *+{\twosmc A}    ="33",
( 0,  0) *+{\sttilt A}    ="00",
(55,  0) *+{\ftors A}     ="10",
(99,  0) *+{\fLsbrick A}  ="30",
\ar ^{\Sub}                            "06";"16"
\ar _{\sF}                             "36";"16"
\ar ^{H^{-1}}                          "04";"06"
\ar _{\textup{(heart)}[-1]\cap \mod A}"13";"16"
\ar _{\boxed{?[-1] \cap \mod A}}           "33";"36"
\ar ^{\nu}                             "02";"04"
\ar ^{I \mapsto ({^\perp}I[{<}0],{^\perp}I[{>}0])} "04";"13"
\ar _{P \mapsto (P[{<}0]^\perp,P[{>}0]^\perp)} "02";"13"
\ar ^{\textup{simples of heart}}       "13";"33"
\ar ^{\Fac}                            "00";"10"
\ar _{\sT}                             "30";"10"
\ar _{H^0}                             "02";"00"
\ar ^{\textup{(heart)}\cap \mod A}    "13";"10"
\ar ^{\boxed{? \cap \mod A}}               "33";"30"
\ar@{-} "00"; ( 0,-8)
\ar@{-}_{\boxed{M \mapsto \ind(M/{\rad_B M})}} ( 0,-8);(99,-8)
\ar (99,-8); "30"
\ar@{-} "06"; ( 0,44)
\ar@{-}^{\boxed{M \mapsto \ind(\soc_B M)}} ( 0,44);(99,44)
\ar (99,44); "36"
\end{xy}
\end{align*}
\caption{The commutative diagram}\label{big_comm_intro}
\end{figure}
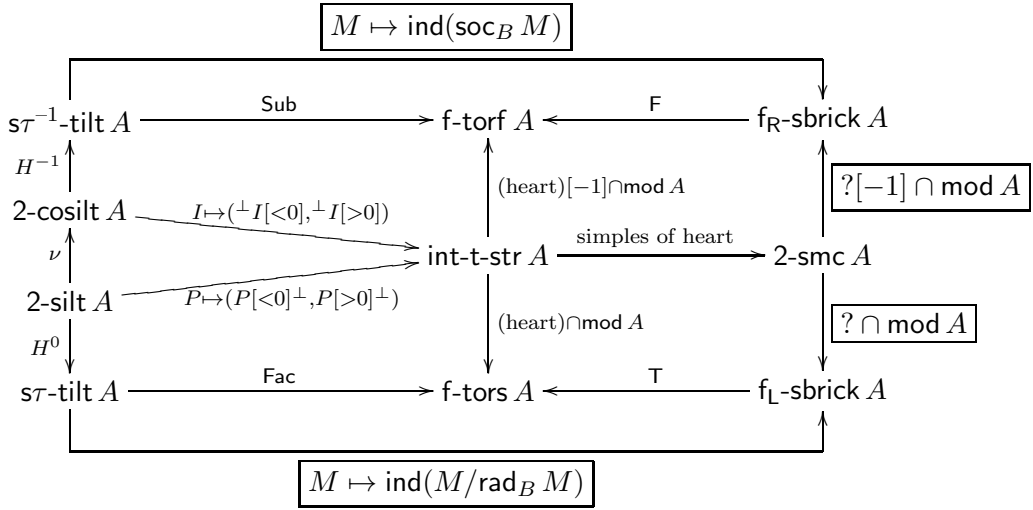
\end{Th}

\medskip

In Subsection \ref{wide_der},
we study left finite wide subcategories $\cW$ 
in a similar way to Subsection \ref{wide_mod},
but we use $P \in \twosilt A$ and $C:=\End_{\sD^\rb(\mod A)}(P)$
corresponding to $\cW$, instead of 
$M \in \sttilt A$ and $B=\End_A(M)$.
There is an equivalence 
$\Hom_{\sD^\rb(\mod A)}(P,?) \colon \cH \to \mod C$ \cite{IY},
where $\cH$ is the heart of $(P[{<}0]^\perp,P[{>}0]^\perp) \in \inttstr A$.
For an explicitly defined idempotent $f \in C$,
we have the following equivalence, 
which is parallel to Theorem \ref{wide_B_intro}. 

\begin{Th}[Theorem \ref{wide_C}]\label{wide_C_intro}
The equivalence $\Hom_{\sD^\rb(\mod A)}(P,?) \colon \cH \to \mod C$ is restricted to 
an equivalence $\Hom_{\sD^\rb(\mod A)}(P,?) \colon \cW \to \mod C/\langle f \rangle$.
\end{Th}

Moreover, there is a canonical surjection $\phi \colon C \to B$ of $K$-algebras, 
and so we have a fully faithful functor $\mod B \to \mod C$.
The surjection $\phi$ induces 
a surjection $\phi \colon C/\langle f \rangle \to B/\langle e \rangle$
and a fully faithful functor $\mod B/\langle e \rangle \to \mod C/\langle f \rangle$.
This embedding is actually an equivalence; 
hence, $\phi \colon C/\langle f \rangle \to B/\langle e \rangle$ 
is an isomorphism of algebras, see Theorem \ref{wide_C_B}.

In Subsection \ref{subsec_Gro},
we investigate 2-term silting objects in $\sK^\rb(\proj A)$
and 2-term simple-minded collections in $\sD^\rb(\mod A)$
from the point of view of the Grothendieck groups.

Let $P \in \twosilt A$.
Then, $P$ has exactly $n_A$ nonisomorphic indecomposable direct summands \cite{KY}.
The set $\{ P_1,P_2,\ldots,P_{n_A} \}$ of indecomposable direct summands of $P$ 
gives a $\Z$-basis of 
the Grothendieck group $K_0(\sK^\rb(\proj A))$ \cite{AI}.
On the other hand, 
the corresponding 2-term simple-minded collection $\cX \in \twosmc A$ 
has $n_A$ distinct elements, which are denoted by $X_1,X_2,\ldots,X_{n_A}$. 
They induce a $\Z$-basis of the Grothendieck group $K_0(\sD^\rb(\mod A))$ \cite{KY}.

There is a natural $\Z$-bilinear form
$\langle ?,? \rangle \colon 
K_0(\sK^\rb(\proj A)) \times K_0(\sD^\rb(\mod A)) \to \Z$
given as
\begin{align*}
\langle P, X \rangle := \sum_{i \in \Z} (-1)^j \dim_K \Hom_{\sD^\rb(\mod A)}(P,X[j]).
\end{align*}
We prove that 
the $\Z$-basis $\{ P_1,P_2,\ldots,P_{n_A} \}$ of $K_0(\sK^\rb(\proj A))$
and the $\Z$-basis $\{X_1,X_2,\ldots,X_{n_A}\}$ of $K_0(\sD^\rb(\mod A))$
satisfy the following ``duality'' with respect to the bilinear form 
$\langle ?,? \rangle$.

\begin{Th}[Theorem \ref{P_X_dual_Gro}]
There exists a permutation $\sigma \colon \{1,2,\ldots,n_A\} \to \{1,2,\ldots,n_A\}$
satisfying
\begin{align*}
\langle P_k,X_l \rangle=\begin{cases}
\dim_K \End_{\sD^\rb(\mod A)}(X_l) & (\sigma(k)=l) \\
0 & (\sigma(k) \ne l)
\end{cases}.
\end{align*}
\end{Th}

\medskip

We give examples of semibricks and wide subcategories in Section \ref{sec_ex}.

In Subsection \ref{Nakayama_number},
we calculate the numbers $a_{n,l}:=\#\sbrick A_{n,l}$ 
and $b_{n,l}:=\#\sbrick B_{n,l}$
for the following two families 
$(A_{n,l})_{n,l \ge 1}$ and $(B_{n,l})_{n,l \ge 1}$ of Nakayama algebras:
\begin{align*}
A_{n,l} \colon
\begin{xy}
(  0,  0)*+{1}="1",
( 10,  0)*+{2}="2",
( 20,  0)*+{\cdots}="cdots",
( 30,  0)*+{n}="n",
\ar "1";"2"
\ar "2";"cdots"
\ar "cdots";"n"
\end{xy},
\quad \text{all the paths of length $l$ are 0},\\
B_{n,l} \colon
\begin{xy}
(  0,  0)*+{1}="1",
( 10,  0)*+{2}="2",
( 20,  0)*+{\cdots}="cdots",
( 30,  0)*+{n}="n",
\ar "1";"2"
\ar "2";"cdots"
\ar "cdots";"n"
\ar@{-} "n";(30,-6)
\ar@{-} (30,-6);(0,-6)
\ar (0,-6);"1"
\end{xy}
, \quad \text{all the paths of length $l$ are 0}.
\end{align*}
Our result is expressed with the Catalan numbers $(c_n)_{n \ge 0}$.
In particular, a path algebra of type $\mathbb{A}_n$ has exactly $c_{n+1}$
semibricks.

\begin{Th}[Theorem \ref{Nakayama_main}]
Let $n,l \ge 1$ be integers.
The following equations hold:
\begin{align*}
a_{n,l}&=c_{n+1} \quad (n=1,2,\ldots,l), &
a_{n,l}&=2a_{n-1,l}+\displaystyle \sum_{i=2}^{l}c_{i-1}a_{n-i,l} \quad (n \ge l+1), \\
b_{n,l}&=(n+1)c_n \quad (n=1,2,\ldots,l), &
b_{n,l}&=2b_{n-1,l}+{\displaystyle \sum_{i=2}^{l}c_{i-1}b_{n-i,l}} \quad (n \ge l+1).
\end{align*}
\end{Th}

Since Nakayama algebras are representation-finite, 
every semibrick is left finite.
Thus, by Theorem \ref{intro_1}, 
there is a bijection between the sets $\sbrick A$ and $\sttilt A$.
The latter set was also investigated by Adachi \cite{Adachi}.

Subsection \ref{Section_tilt} deals with functorially finiteness of wide subcategories.
We write $\fwide A$ for the set of functorially finite wide subcategories of $\mod A$.
If $A$ is hereditary, then $\fwide A$ coincides with $\fLwide A$ \cite{IT}.
We show that there is an inclusion $\fLwide A \subset \fwide A$ for
an arbitrary algebra $A$,
and we give an example of a tilted algebra $A$ 
which does not satisfy the equality $\fLwide A = \fwide A$ (Example \ref{MS_ex}).
Moreover, we prove the following result in the case that
$K$ is an algebraically closed field.

\begin{Th}[Theorem \ref{tilt_th}]\label{tilt_th_intro}
Let $H$ be a hereditary algebra,
$T \in \mod H$ be a tilting module,
and $A:=\End_H(T)$.
Then the following assertions hold.
\begin{itemize}
\item[(1)]
If $\Sub_H \tau T$ has only finitely many indecomposable $H$-modules, 
then $\fwide A = \fLwide A$.
If $\Fac_H T$ has only finitely many indecomposable $H$-modules, 
then $\fwide A = \fRwide A$.
\item[(2)]
If $T$ is either preprojective or preinjective, 
then $\fwide A=\fLwide A=\fRwide A$.
\item[(3)]
Assume that $H$ is a hereditary algebra of extended Dynkin type.
We decompose $T$ as $T_{\ppr} \oplus T_{\reg} \oplus T_{\pin}$ with  
a preprojective module $T_{\ppr}$, 
a regular module $T_{\reg}$, 
and a preinjective module $T_{\pin}$.
If $T_{\reg} \ne 0$ and $T_{\pin} = 0$, 
then we have $\fwide A \supsetneq \fRwide A$, and
if $T_{\reg} \ne 0$ and $T_{\ppr} = 0$, 
then we have $\fwide A \supsetneq \fLwide A$.
\end{itemize}
\end{Th}

As a consequence of Theorem \ref{tilt_th_intro} and \cite[XVII.3]{SimS}, 
we obtain the following criterion
in the case that $H$ is a hereditary algebra of extended Dynkin type.

\begin{Cor}[Corollary \ref{ex_Dynkin_table}]
We use the setting of Theorem \ref{tilt_th} (3).
Then, the tilting $H$-module $T$ satisfies one of the conditions (1)--(5) 
in the following table, 
which shows whether $\fwide A = \fLwide A$ and $\fwide A = \fRwide A$
hold in each case.
\textup{
\begin{center}
\begin{tabular}{c|ccc|cc}
No. & $T_{\ppr}$ & $T_{\reg}$ & $T_{\pin}$ &
$\fwide A = \fLwide A$ & $\fwide A = \fRwide A$ \\
\hline
(1) & $\ne 0$ & $=   0$ & $=   0$ & Yes & Yes \\
(2) & $=   0$ & $=   0$ & $\ne 0$ & Yes & Yes \\
(3) & $\ne 0$ & $\ne 0$ & $=   0$ & Yes & No  \\
(4) & $=   0$ & $\ne 0$ & $\ne 0$ & No  & Yes \\
(5) & $\ne 0$ &         & $\ne 0$ & Yes & Yes 
\end{tabular}
\end{center}
}
\end{Cor}

\subsection{Notation}

We use the symbol ``$\subset$'' as the meaning of ``$\subseteq$''.
The composition of maps or functors $f \colon X \to Y$ and $g \colon Y \to Z$ is denoted by $gf$. 

Throughout of this paper, 
$K$ is a field and $A$ is a finite-dimensional $K$-algebra.
The category of 
finite-dimensional right $A$-modules is denoted by $\mod A$,
and its full subcategory consisting of all the
projective (resp.\ injective) $A$-modules is denoted by $\proj A$ (resp.\ $\inj A$).
Unless otherwise stated, algebras and modules are finite-dimensional,
and subcategories are full subcategories
which are closed under taking isomorphic objects.
The bounded derived category of $\mod A$ is denoted by $\sD^\rb(\mod A)$,
and the homotopy category of the bounded complex category over $\proj A$
(resp.\ $\inj A$) 
is denoted by $\sK^\rb(\proj A)$ (resp.\ $\sK^\rb(\inj A)$).

For each full subcategory $\cC \subset \mod A$,
we define full subcategories of $\mod A$ as follows:
\begin{itemize}
\item
$\add \cC$ is the additive closure of $\cC$,
\item 
$\Fac \cC$ consists of the factor modules of objects in $\add \cC$,
\item 
$\Sub \cC$ consists of the submodules of objects in $\add \cC$,
\item 
$\Filt \cC$ consists of the objects $M$ 
such that there exists a sequence $0 =M_0 \subset M_1 \subset \cdots \subset M_n=M$
with $M_i/M_{i-1} \in \add \cC$,
\item
$\cC^\perp$ and ${^\perp \cC}$ are defined as
$\cC^\perp:=\{ M \in \mod A \mid  \Hom_A(\cC,M)=0 \}$ and
${^\perp \cC}:=\{ M \in \mod A \mid \Hom_A(M,\cC)=0 \}$.
\end{itemize}
If the set of objects of $\cC$ is a finite set $\{ M_1,\ldots,M_n \}$,
then we write $\add (M_1,\ldots,M_n)$ for $\add \cC$,
and other full subcategories are similarly denoted.
In order to emphasize the algebra $A$, we sometimes use notations such as 
$\add_A \cC$.

For $M \in \mod A$, the notation $\ind M$ denotes
the set of isoclasses of indecomposable direct summands of $M$,
and we set $|M|:=\#(\ind M)$.
If $M \cong \bigoplus_{i=1}^m M_i^{n_i}$ with $M_i$ indecomposable,
$M_i \not \cong M_j$ ($i \ne j$), and $n_i \ge 1$,
then $\ind M=\{M_1,\ldots,M_m\}$ and $|M|=m$ hold.

For a set $X \subset A$, 
the notation $\langle X \rangle$ denotes 
the two-sided ideal of $A$ generated by $X$.
We define a functor $D$ as the $K$-dual $\Hom_K(?,K)$.

We often identify an object and its isoclass.

\section{Semibricks in Module Categories}\label{Section_sttilt}

In this section, we consider the sets of pairwise Hom-orthogonal bricks
called \textit{semibricks},
and give bijections between the semibricks and other concepts in module categories,
especially \textit{support $\tau$-tilting modules}.

\subsection{Bijections I}

We first recall the definition of bricks,
and give the definition of semibricks.

\begin{Def}
We define the following notions.
\begin{itemize}
\item[(1)] 
A module $S$ in $\mod A$ is called a \textit{brick} 
if $\End_A(S)$ is a division $K$-algebra.
We write $\brick A$ for the set of isoclasses of bricks in $\mod A$.   
\item[(2)] 
A subset $\cS \subset \brick A$ is called a \textit{semibrick} 
if $\Hom_A(S_1,S_2)=0$ holds for any $S_1 \ne S_2 \in \cS$.
We write $\sbrick A$ for the set of semibricks in $\mod A$.   
\end{itemize}
\end{Def}

By Schur's Lemma,
a simple $A$-module is a brick,
and a set of isoclasses of simple $A$-modules is a semibrick. 

Note that we allow a semibrick to be an infinite set.
This occurs, for example, in the case of the Kronecker quiver algebra.
Any subset of a semibrick is also a semibrick,
so if there is a semibrick consisting of infinitely many bricks in $\mod A$, 
then there are uncountably many semibricks in $\mod A$.
In this paper, we treat the semibricks satisfying some conditions on torsion pairs.
Each of such semibricks has $|A|$ bricks at most, see Corollary \ref{card_bound}.

We recall some fundamental properties of torsion pairs.

A full subcategory $\cT \subset \mod A$ is called a \textit{torsion class} 
if $\cT$ is closed under taking extensions and factor modules, and
a full subcategory $\cF \subset \mod A$ is called a \textit{torsion-free class} 
if $\cF$ is closed under taking extensions and submodules.
We define $\tors A$ as the set of torsion classes in $\mod A$
and $\torf A$ as the set of torsion-free classes in $\mod A$.
For a torsion class $\cT$, the corresponding torsion pair is $(\cT,\cT^\perp)$, and 
for a torsion-free class $\cF$, the corresponding torsion pair is $({^\perp \cF},\cF)$.
Note that ${^\perp \cC} \in \tors A$ and $\cC^\perp \in \torf A$ hold
for any full subcategory $\cC \subset \mod A$.

Let $\cC \subset \mod A$ be a full subcategory. 
We write $\sT(\cC)$ for the smallest torsion class containing $\cC$, 
and $\sF(\cC)$ for the smallest torsion-free class containing $\cC$.
It is well-known that $\sT(\cC)=\Filt(\Fac \cC)$ and $\sF(\cC)=\Filt(\Sub \cC)$ hold, 
see \cite[Lemma 3.1]{MS}.

We mainly consider functorially finite torsion classes and torsion-free classes
(see \cite[Section 2]{MS} for the definition of functorially finiteness) in this paper.
We write $\ftors A$ for the set of functorially finite torsion classes in $\mod A$ and
$\ftorf A$ for the set of functorially finite torsion-free classes in $\mod A$.
By \cite[Theorem]{Smalo},
for a torsion pair $(\cT,\cF)$ in $\mod A$,
$\cT \in \ftors A$ holds if and only if $\cF \in \ftorf A$.

With these preparations on torsion pairs,
we define left finiteness and right finiteness of semibricks as follows.

\begin{Def}
Let $\cS \in \sbrick A$.
\begin{itemize}
\item[(1)]
The semibrick $\cS$ is said to be \textit{left finite} 
if $\sT(\cS)$ is functorially finite.
We write $\fLsbrick A$ for the set of left finite semibricks in $\mod A$.
\item[(2)]
The semibrick $\cS$ is said to be \textit{right finite} 
if $\sF(\cS)$ is functorially finite.
We write $\fRsbrick A$ for the set of right finite semibricks in $\mod A$.
\end{itemize}
\end{Def}

One may wonder whether a subset of a left finite semibrick is left finite again.
It does not hold in general, see Example \ref{MS_ex}. 

We recall the notion of support $\tau$-tilting modules,
which was introduced by 
Adachi--Iyama--Reiten \cite{AIR}.
Let $M \in \mod A$ and $P \in \proj A$,
then the $A$-module $M$ is called a \textit{$\tau$-rigid module} if
$M$ satisfies $\Hom_A(M,\tau M)=0$ for the Auslander--Reiten translation $\tau$,
and the pair $(M,P)$ is called a \textit{$\tau$-rigid pair} if
$M$ is $\tau$-rigid and $\Hom_A(P,M)=0$.
We write $\trigid A$ 
for the set of isoclasses of basic $\tau$-rigid modules in $\mod A$.
Here, $M$ is said to be \textit{basic}
if $M$ can be decomposed into $\bigoplus_{i=1}^m M_i$
with $M_i$ indecomposable and $M_i \not \cong M_j$ for $i \ne j$.
We remark that any $\tau$-rigid module $M$ satisfies
$\Ext_A^1(M,\Fac M)=0$, see \cite[Proposition 5.8]{AS2}.

Assume that $(M,P)$ is a $\tau$-rigid pair. 
We say that $(M,P)$ is a \textit{support $\tau$-tilting pair}
if $|M|+|P|=|A|$, 
and that $(M,P)$ is an \textit{almost support $\tau$-tilting pair}
if $|M|+|P|=|A|-1$.
An $A$-module $M$ is called a \textit{support $\tau$-tilting module} if
there exists $P \in \proj A$ such that
$(M,P)$ is a support $\tau$-tilting pair.
We write $\sttilt A$ 
for the set of isoclasses of basic support $\tau$-tilting modules in $\mod A$,
and we use both notations $M \in \sttilt A$ and $(M,P) \in \sttilt A$.
We remark that if $(M,P)$ and $(M,Q)$ are support $\tau$-tilting pairs,
then $\add P=\add Q$ holds \cite[Proposition 2.3]{AIR}.

The notions of $\tau^{-1}$-rigid modules and
support $\tau^{-1}$-tilting modules are also defined dually,
see the paragraphs just before \cite[Theorem 2.15]{AIR}.
We write 
$\tirigid A$ for 
the set of isoclasses of basic $\tau^{-1}$-rigid modules in $\mod A$ and
$\stitilt A$ for 
the set of isoclasses of basic support $\tau^{-1}$-tilting modules in $\mod A$.

Now, we are ready to state our first main theorem,
which gives a large extension of a bijection given in \cite[Theorem 4.1]{DIJ}.
Our theorem gives a bijection between the support $\tau$-tilting modules and
the left finite semibricks, and its dual bijection.

\begin{Th}\label{sttilt_fsbrick}
We have the following maps,
where $B:=\End_A(M)$ in each case.
\begin{itemize}
\item[(1)]
There exists a surjection $\trigid A \to \fLsbrick A$ defined 
as $M \mapsto \ind(M/{\rad_B M})$. 
Under this map, 
the left finite semibricks for two $\tau$-rigid modules $M,M' \in \trigid A$ 
coincide if and only if $\Fac M=\Fac M'$.
\item[(2)]
There exists a bijection $\sttilt A \to \fLsbrick A$ defined 
as $M \mapsto \ind(M/{\rad_B M})$. 
\item[(3)]
There exists a surjection $\tirigid A \to \fRsbrick A$ defined 
as $M \mapsto \ind({\soc_B M})$. 
Under this map, the right finite semibricks for two $\tau^{-1}$-rigid modules 
$M,M' \in \tirigid A$ 
coincide if and only if $\Sub M=\Sub M'$.
\item[(4)]
There exists a bijection $\stitilt A \to \fRsbrick A$ defined 
as $M \mapsto \ind(\soc_B M)$. 
\end{itemize}
\end{Th}

We remark that the two conditions 
$\#\sttilt A<\infty$ and $\ftors A=\tors A$ 
are equivalent
\cite[Theorem 3.8, Proposition 3.9]{DIJ}. 
In this case, $A$ is said to be \textit{$\tau$-tilting finite},
and we have the equations $\fLsbrick A=\fRsbrick A=\sbrick A$ and 
the bijections $\sttilt A \to \sbrick A$ and $\stitilt A \to \sbrick A$.

In the rest of this subsection, we prove Theorem \ref{sttilt_fsbrick}.
To investigate the maps in this theorem in detail,
we sometimes use the following notation.

\begin{Def}\label{summand_symbol}
Let $M \in \trigid A$ and $B:=\End_A(M)$.
We decompose $M$ as
$M=\bigoplus_{i=1}^m M_i$ with $M_i$ indecomposable,
and then define 
\begin{align*}
L:=\rad_B M, \quad N:=M/L, \quad
L_i:=\sum_{f \in \rad_A(M,M_i)} \Im f \subset M_i, \quad N_i:=M_i/L_i
\end{align*}
for $i=1,2,\ldots,m$.
Moreover, we set $I:=\{i \in \{1,2,\ldots,m\} \mid N_i \ne 0\}$.
\end{Def}

Clearly, $L=\bigoplus_{i=1}^m L_i$ and $N=\bigoplus_{i=1}^m N_i$ hold.
Using this notation, we first show the well-definedness of the maps in 
Theorem \ref{sttilt_fsbrick}.

\begin{Lem}\label{well-def}
In the setting of Definition \ref{summand_symbol}, 
the following assertions hold.
\begin{itemize}
\item[(1)]
We have $\Ext_A^1(M,L)=0$.
\item[(2)]
The module $N_i$ is a brick or zero for each $i$.
\item[(3)]
If $i \ne j$, we have $\Hom_A(M_i,N_j)=0$ and $\Hom_A(N_i,N_j)=0$.
\item[(4)]
The $A$-module $N$ is basic and $\ind N=\{N_i \mid i \in I \} \in \sbrick A$. 
\item[(5)]
The torsion class $\sT(N)$ is equal to $\Fac M$.
\item[(6)]
We have $\ind N=\{N_i \mid i \in I \} \in \fLsbrick A$.
\end{itemize}
\end{Lem}

\begin{proof}
(1)
Because $M$ is $\tau$-rigid, it is enough to show that $L \in \Fac M$.
Take a $K$-basis $f_1,\ldots,f_s \colon M \to M$ of $\rad B \subset B = \End_A(M)$,
and set $f:=\begin{bmatrix} f_1 & \cdots & f_s \end{bmatrix} \colon M^{\oplus s} \to M$.
Then, $\Im f=\rad_B M$ holds, so $L \in \Fac M$.
Thus, we have the assertion.

(2)
It is sufficient to show that any nonzero endomorphism $f \colon N_i \to N_i$
is an isomorphism.
Consider the exact sequence 
$0 \to L_i \xrightarrow{\mu_i} M_i \xrightarrow{\pi_i} N_i \to 0$.
Since $\Ext_A^1(M_i,L_i)=0$ follows from (1), 
we obtain that $\Hom_A(M_i,\pi_i)$ is surjective. 
Thus, there exists $g \colon M_i \to M_i$ such that $\pi_i g = f \pi_i$.
Because $f \ne 0$, the map $g$ is an isomorphism 
(otherwise, we have $\Im g \subset L_i$, which yields $f=0$, a contradiction).
Therefore, $f$ is also an isomorphism.

(3)
Let $i \ne j$ and $f \colon M_i \to N_j$.
Consider the exact sequence 
$0 \to L_j \xrightarrow{\mu_j} M_j \xrightarrow{\pi_j} N_j \to 0$.
Since $\Ext_A^1(M_i,L_j)=0$ follows from (1), 
we get that $\Hom_A(M_i,\pi_j)$ is surjective. 
Thus, there exists $g \colon M_i \to M_j$ such that $\pi_j g = f$.
Because $M$ is basic, we get $M_i \not \cong M_j$.
Therefore, we have $\Im g \subset L_j$ and $f=0$.
We get that $\Hom_A(M_i,N_j)=0$ for all $i \ne j$.
Because $N_i \in \Fac M_i$, we also obtain that $\Hom_A(N_i,N_j)=0$. 

(4)
We obtain the assertions by (2) and (3).

(5)
Since $M$ is $\tau$-rigid, $\Fac M$ is a torsion class \cite[Theorem 5.10]{AS2}.
Thus, the inclusion $\sT(N) \subset \Fac M$ holds.
To prove that $\Fac M \subset \sT(N)$,
it suffices to show that $M \in \sT(N)$.
We define $f \colon M^{\oplus s} \to M$ as in (1), then
$f((\rad_B^t M)^{\oplus s})=\rad_B^{t+1}M$ holds for all $t \ge 0$.
Therefore, $\rad_B^{t+1}M/{\rad_B^{t+2}M} \in \Fac (\rad_B^t M/{\rad_B^{t+1}M})$.
By induction, we have $\rad_B^t M/{\rad_B^{t+1}M} \in \Fac N$ for all $t \ge 0$.
Moreover, there exists $t_0$ such that $\rad_B^{t_0} M=0$.
Then the sequence 
$M \supset \rad_B M \supset \cdots \supset \rad_B^{t_0} M=0$
shows that $M \in \Filt(\Fac N)=\sT(N)$.
Thus, the assertion $\sT(N)=\Fac M$ is proved.

(6)
By (4), it remains to show $\sT(N) \in \ftors A$.
We have obtained $\sT(N)=\Fac M \in \tors A$ in (5),
and $\Fac M$ is functorially finite in $\mod A$ by \cite[Proposition 4.6]{AS1}.
Thus, we have the assertion.
\end{proof}

We recall an important property of support $\tau$-tilting modules 
from \cite{AIR}, which associates the support $\tau$-tilting modules
and the functorially finite torsion classes bijectively.

\begin{Prop}\label{sttilt_ftors}
\cite[Theorem 2.7]{AIR}
There exists a bijection $\Fac \colon \sttilt A \to \ftors A$ 
defined as $M \mapsto \Fac M$.
The inverse $\ftors A \to \sttilt A$ is given by 
sending each $\cT \in \ftors A$ 
to the direct sum of indecomposable Ext-projective objects in $\cT$.
\end{Prop}


To show Theorem \ref{sttilt_fsbrick}, 
we study the operation $\sT$
taking the smallest torsion class in the following lemmas.

\begin{Lem}\label{Schur}
Let $\cS \in \sbrick A$, then the following hold. 
\begin{itemize}
\item[(1)] 
Let $S \in \cS$.
Every nonzero homomorphism $f \colon M \to S$ with $M \in \sT(\cS)$ is surjective. 
Moreover, we have $\Ker f \in \sT(\cS)$.
\item[(2)]
Let $L \in \sT(\cS)$ and assume $L \ne 0$. 
If every nonzero homomorphism $f \colon M \to L$ with $M \in \sT(\cS)$ is surjective 
and satisfies $\Ker f \in \sT(\cS)$, then we have $L \in \cS$.
\end{itemize}
\end{Lem}

\begin{proof}
(1)
Since $M \in \sT(\cS)=\Filt (\Fac \cS)$,
there exists a sequence 
$0=M_0 \subset M_1 \subset \cdots \subset M_{l-1} \subset M_l = M$
with $M_i/M_{i-1} \in \Fac \cS$.
Because $f \ne 0$,
we can take $k:=\min\{ i \mid f(M_i) \ne 0\} \ge 1$.
Then $f$ induces a nonzero map $\overline{f} \colon M/M_{k-1} \to S$
with $f(M_k/M_{k-1}) \ne 0$.
Because $M_k/M_{k-1} \in \Fac \cS$,
there exists a homomorphism $g \colon N \to M/M_{k-1}$ 
with $N \in \add \cS$ and $\Im g=M_k/M_{k-1}$,
and we have a nonzero map $\overline{f}g \ne 0 \colon N \to S$.
Then there exists an indecomposable direct summand $S_1$ of $N$ such that
$(\overline{f}g)|_{S_1} \ne 0 \colon S_1 \to S$.
It is clear that $S_1$ is a brick in $\cS$, 
and $(\overline{f}g)|_{S_1} \colon S_1 \to S$ is isomorphic
because $\cS$ is a semibrick.
Thus, $\overline{f}$ is surjective, and so is $f$.

The isomorphism $(\overline{f}g)|_{S_1} \colon S_1 \to M/M_{k-1} \to S$ implies 
that the epimorphism $\overline{f} \colon M/M_{k-1} \to S$ splits.
Thus, we have $S \oplus \Ker \overline{f} \cong M/M_{k-1} \in \Filt (\Fac \cS)=\sT(\cS)$.
Therefore, $\Ker \overline{f}$ belongs to $\sT(\cS)$.
Clearly, there exists an exact sequence
$0 \to M_{k-1} \to \Ker f \to \Ker \overline{f} \to 0$,
and $M_{k-1}$ also belongs to $\Filt (\Fac \cS)=\sT(\cS)$.
Therefore, we have $\Ker f \in \sT(\cS)$.

(2)
Since $L \in \sT(\cS)=\Filt (\Fac \cS)$ and $L \ne 0$,
there exists a nonzero map $f \colon S \to L$ with $S \in \cS$.
By assumption, $f$ is surjective and $\Ker f \in \sT(\cS)$.
An inclusion $\Ker f \to S$ is zero or surjective by (1),
and then $f \ne 0$ implies $\Ker f=0$.
Therefore, $f$ is isomorphic; hence, we have $L \in \cS$.
\end{proof}

\begin{Lem}\label{sbrick_tors}
The map $\sT \colon \sbrick A \to \tors A$ is injective.
Moreover, it is restricted to an injection $\sT \colon \fLsbrick A \to \ftors A$.
\end{Lem}

\begin{proof}
Let semibricks $\cS_1$ and $\cS_2$ satisfy $\sT(\cS_1)=\sT(\cS_2)$.

We claim that $\cS_1 \subset \cS_2$.
Let $S_1 \in \cS_1$.
By $S_1 \in \sT(\cS_1) = \sT(\cS_2)=\Filt(\Fac \cS_2)$,
there exists a nonzero map $f \colon S_2 \to S_1$ with $S_2 \in \cS_2$.
Because $S_2 \in \sT(\cS_2)=\sT(\cS_1)$, 
Lemma \ref{Schur} (1) implies that $f$ is surjective.
Similarly, we have a surjection 
$g \colon S_1' \to S_2$ with $S_1' \in \cS_1$.
Then $fg \colon S_1' \to S_1$ is surjective.
Because $\cS_1$ is a semibrick, $fg$ is isomorphic.
Since $g$ is surjective, $f$ is isomorphic.
Thus, we have $S_1 \cong S_2 \in \cS_2$.
Now, we have proved that $\cS_1 \subset \cS_2$.

By symmetry, we also obtain that $\cS_2 \subset \cS_1$; hence, $\cS_1=\cS_2$.
This implies the injectivity.

The restriction $\sT \colon \fLsbrick A \to \ftors A$ 
is well-defined by definition, and it is injective.
\end{proof}

Now, we can show Theorem \ref{sttilt_fsbrick}.

\begin{proof}[Proof of Theorem \ref{sttilt_fsbrick}]
We prove (1) and (2). 
By applying (1) and (2) to the opposite algebra $A^\mathrm{op}$ and taking the $K$-duals,
we obtain (3) and (4).

The maps in (1) and (2) are well-defined by Lemma \ref{well-def}.

Lemma \ref{well-def} also implies that the following diagram is commutative:
\begin{align*}
\begin{xy}
(-40, 12) *+{\sttilt A} = "2",
(-40,  0) *+{\trigid A} = "0",
(  0,  0) *+{\ftors A} = "1",
( 40,  0) *+{\fLsbrick A} = "3",
(-40, -8) = "4",
( 40, -8) = "5",
\ar_{\text{incl.}} "2"; "0"
\ar^{\Fac} "0"; "1"
\ar^{\Fac} "2"; "1"
\ar_{\sT} "3"; "1"
\ar@{-} "0"; "4"
\ar@{-}_{M \mapsto \ind(M/{\rad_B M})} "4"; "5" 
\ar "5"; "3"
\end{xy}.
\end{align*}
Here, $\Fac \colon \trigid A \to \ftors A$ is surjective by Proposition \ref{sttilt_ftors},
so $\sT \colon \fLsbrick A \to \ftors A$ is surjective.
By Lemma \ref{sbrick_tors},
$\sT \colon \fLsbrick A \to \ftors A$ is also injective,
so it is bijective.
Thus, the map in (1) is surjective,
and two $\tau$-rigid modules $M,M' \in \trigid A$ are 
sent to the same left finite semibrick
if and only if $\Fac M=\Fac M'$.
We can prove the bijectivity of the map in (2) in a similar way,
since $\Fac \colon \sttilt A \to \ftors A$ is bijective
by Proposition \ref{sttilt_ftors}.
\end{proof}

In the proof of Theorem \ref{sttilt_fsbrick} above,
we have already obtained the next result.

\begin{Prop}\label{comm_3}
The following conditions hold, where $B:=\End_A(M)$ in each case.
\begin{itemize}
\item[(1)] The map $\sT \colon \fLsbrick A \to \ftors A$ is a bijection,
and we have the following commutative diagram of bijections:
\begin{align*}
\begin{xy}
(-40,  0) *+{\sttilt A} = "0",
(  0,  0) *+{\ftors A} = "1",
( 40,  0) *+{\fLsbrick A} = "3",
(-40, -8) = "4",
( 40, -8) = "5",
\ar^{\Fac} "0"; "1"
\ar_{\sT} "3"; "1"
\ar@{-} "0"; "4"
\ar@{-}_{M \mapsto \ind(M/{\rad_B M})} "4"; "5" 
\ar "5"; "3"
\end{xy}.
\end{align*}
\item[(2)] The map $\sF \colon \fRsbrick A \to \ftorf A$ is a bijection,
and we have the following commutative diagram of bijections:
\begin{align*}
\begin{xy}
(-40,  0) *+{\stitilt A} = "0",
(  0,  0) *+{\ftorf A} = "1",
( 40,  0) *+{\fRsbrick A} = "3",
(-40, -8) = "4",
( 40, -8) = "5",
\ar^{\Sub} "0"; "1"
\ar_{\sF} "3"; "1"
\ar@{-} "0"; "4"
\ar@{-}_{M \mapsto \ind(\soc_B M)} "4"; "5" 
\ar "5"; "3"
\end{xy}.
\end{align*}
\end{itemize}
\end{Prop}

Therefore, the functorially finite torsion classes in $\mod A$ 
correspond bijectively 
not only to the support $\tau$-tilting modules, which are in the ``projective'' side
(Proposition \ref{sttilt_ftors}),
but also to the left finite semibricks, which are in the ``simple'' side
(Lemma \ref{Schur}).
These are connected by the map in Theorem \ref{sttilt_fsbrick}.

We have the next corollary on the cardinalities of semibricks.

\begin{Cor}\label{card_bound}
If $\cS \in \sbrick A$ is either left finite or right finite,
then $\# \cS \le |A|$ holds.
\end{Cor}

\begin{proof}
We argue the left finite case. The other case is similarly proved.

Let $\cS \in \fLsbrick A$.
Then, we can take $M \in \sttilt A$ such that 
$\ind(M/{\rad_B M})=\cS$ by Theorem \ref{sttilt_fsbrick}.
Thus, we have $\# \cS=|M/{\rad_B M}| \le |M|$ 
by Lemma \ref{well-def} (4).
By the definition of support $\tau$-tilting modules, 
$|M| \le |A|$ holds.
Therefore, we have $\#\cS \le |A|$.
\end{proof}

\subsection{Labeling the exchange quiver with bricks I}\label{label_mod}

The map in Theorem \ref{sttilt_fsbrick} gives us a way of labeling
the exchange quiver of $\sttilt A$ with bricks,
that is, to label each arrow with a brick.
Here, the exchange quiver is the quiver expressing the \textit{mutations}.

We briefly recall the definition of
mutations in $\sttilt A$ from \cite[Definition 2.19]{AIR}.
Let $(M,P) \ne (N,Q) \in \sttilt A$.
If there exists an almost support $\tau$-tilting pair 
which is a common direct summand of $(M,P)$ and $(N,Q)$,
then we say that $(N,Q)$ is a \textit{mutation} of $(M,P)$.
By \cite[Theorem 2.18]{AIR},
for each direct indecomposable summand $L$ of $M$ or $P$,
there uniquely exists a mutation of $(M,P)$ at $L$.
If $N$ is a mutation of $M$, then exactly one of
$\Fac M \supsetneq \Fac N$ or $\Fac M \subsetneq \Fac N$ holds.
If $\Fac M \supsetneq \Fac N$, 
then $N$ is called a \textit{left mutation} of $M$,
and otherwise $N$ is called a \textit{right mutation} of $M$.
We remark that if $N$ is the left mutation of $M$ at $M_1$ with $M=M_1 \oplus M_2$,
then we have $\Fac N=\Fac M_2$.

The exchange quiver of $\sttilt A$ is a quiver 
with its vertices the elements of $\sttilt A$,
and there is an arrow from $M$ to $N$ if and only if 
$N$ is a left mutation of $M$.
For any vertex $M$, the number of arrows which are either from or to $M$ is always $|A|$.

The following lemma is important, and will be used later.

\begin{Lem}\label{no_tors_between}\cite[Example 3.5]{DIJ}
If $N$ is a left mutation of $M$ in $\sttilt A$,
then there exists no torsion class $\cT \in \tors A$ satisfying 
$\Fac M \supsetneq \cT \supsetneq \Fac N$.
\end{Lem}

We analogously define mutations in $\stitilt A$.
Assume that $N$ is a mutation of $M$ in $\stitilt A$.
In this paper, we call $N$ a left mutation of $M$
if $\Sub M \subsetneq \Sub N$
and a right mutation of $M$ if $\Sub M \supsetneq \Sub N$.

There is a one-to-one correspondence between $\sttilt A$ and $\stitilt A$ given
by the following bijections:
\begin{align*}
\sttilt A \xrightarrow[\cong]{\Fac} \ftors A \xrightarrow[\cong]{?^\perp} \ftorf A
\xleftarrow[\cong]{\Sub} \stitilt A.
\end{align*}
We have the next easy observations on this map.

\begin{Prop}\label{sttilt_stitilt}
The above bijection $\sttilt A \to \stitilt A$ satisfies the following properties.
\begin{itemize}
\item[(1)]\cite[Proposition 2.16]{AIR}
Let $(M,P) \in \sttilt A$, and
decompose $M=M_{\np} \oplus M_{\pr}$ 
so that $M_{\np}$ has no nonzero projective direct summand 
and that $M_{\pr}$ is projective.
Then $(M,P) \in \sttilt A$ maps to 
$(\tau M_{\np} \oplus \nu P, \nu M_{\pr}) \in \stitilt A$.
\item[(2)]
Assume that $M,N \in \sttilt A$ map to $M',N' \in \stitilt A$ respectively.
Then $N$ is a left mutation of $M$ in $\sttilt A$ if and only if
$N'$ is a left mutation of $M'$ in $\stitilt A$.
\item[(3)]
The exchange quivers of $\sttilt A$ and $\stitilt A$ are isomorphic
by this bijection.
\end{itemize}
\end{Prop}

\begin{proof}
The part (2) is easy, because the above bijection preserves direct summands by (1).

The part (3) is immediately deduced from (2).
\end{proof}

We also need the following detailed description of Theorem \ref{sttilt_fsbrick}.

\begin{Prop}\label{brick_eq}
In the setting of Definition \ref{summand_symbol},
let $M \in \sttilt A$.
Then the following conditions are equivalent for $i=1,2,\ldots,m$.
\begin{itemize}
\item[(a)] The module $N_i$ is a brick.
\item[(b)] The module $N_i$ is nonzero.
\item[(c)] The module $M_i$ is not in $\Fac \bigoplus_{j \ne i} M_j$.
\item[(d)] There exists a left mutation of $M$ at $M_i$ in $\sttilt A$.
\end{itemize}
In particular, the number of left mutations of $M$ is equal to $|M/{\rad_B M}|$.
\end{Prop}

\begin{proof}
The conditions (a) and (b) are equivalent by Lemma \ref{well-def} (2).

Next, we show the equivalence of (b) and (c).
For each $i$, set $B_i:=\End_A(M_i)$ and 
\begin{align*}
L'_i:=\sum_{f \in \rad_A(M_j,M_i), \ j \ne i} \Im f, \quad
L''_i:=\sum_{f \in \rad_A(M_i,M_i)} \Im f.
\end{align*}
They are $B_i$-$A$-subbimodules of $M_i$ and satisfy $L_i=L'_i+L''_i$.
We can see that 
(b) holds if and only if $L'_i+L''_i \ne M_i$, and that
(c) holds if and only if $L'_i \ne M_i$.
Therefore, it is sufficient to prove that $L'_i+L''_i=M_i$ holds if and only if $L'_i=M_i$.
Clearly, $L'_i=M_i$ implies $L'_i+L''_i=M_i$.
On the other hand, assume $L'_i+L''_i=M_i$.
Because $L''_i=\rad_{B_i} M_i$, we have $L'_i=M_i$ by applying Nakayama's Lemma
as left $B_i$-modules.
Thus, (b) and (c) are equivalent.

The equivalence of (c) and (d) is proved in \cite[Definition-Proposition 2.28]{AIR}.
\end{proof}

Now, we are able to define labels of the exchange quivers.

\begin{Def}\label{label_mod_def}
We label the exchange quivers of $\sttilt A$
and $\stitilt A$ with bricks as follows.
\begin{itemize}
\item[(1)]
Let $M \in \sttilt A$ and decompose $M$ as
$M=\bigoplus_{i=1}^m M_i$ with $M_i$ indecomposable.
Assume that $M \to N$ is an arrow in the exchange quiver of $\sttilt A$,
and that $N$ is the left mutation of $M$ at $M_i$.
Then we label this arrow with a brick $M_i/{\sum_{f \in \rad_A(M,M_i)} \Im f}$.
\item[(2)]
Let $N' \in \stitilt A$ and decompose $N'$ as
$N'=\bigoplus_{i=1}^m N'_i$ with $N'_i$ indecomposable.
Assume that $M' \to N'$ is an arrow in the exchange quiver of $\stitilt A$,
and that $M'$ is the right mutation of $N'$ at $N'_i$.
Then we label this arrow with a brick $\bigcap_{f \in \rad_A(N'_i,N')} \Ker f$.
\end{itemize}
\end{Def}

We came up with this labeling from mutations of 2-term simple-minded collections
in the derived category $\sD^\rb(\mod A)$, which are discussed 
in Subsection \ref{der_label}. 
This labeling is extended to the Hasse quiver of $\tors A$ in
\cite[Section 3.2]{DIRRT}.

Now we recall that there is a bijection between $\sttilt A$ and $\stitilt A$ 
(see Proposition \ref{sttilt_stitilt}).
Actually, the two definitions of labeling coincide under this bijection.

\begin{Th}\label{brick_label_coincide}
The bijection between $\sttilt A$ and $\stitilt A$ 
preserves the labels of the exchange quivers of $\sttilt A$
and $\stitilt A$.
\end{Th}

The above theorem is proved by
the following lemma and proposition 
characterizing the labels of the exchange quivers.

\begin{Lem}\label{brick_label_lem}
Assume that $M \to N$ is an arrow in the exchange quiver of $\sttilt A$ 
labeled with a brick $S$,
and that $M,N \in \sttilt A$ correspond to $M',N' \in \stitilt A$ respectively.
We set $(\cT_1,\cF_1):=(\Fac M, \Sub M')$, $(\cT_2,\cF_2):=(\Fac N, \Sub N')$, 
and $\cC:=\cT_1 \cap \cF_2$.
Then we have the following assertions.
\begin{itemize}
\item[(1)]
The brick $S$ belongs to $\cC$.
\item[(2)]
If $L \in \cC$ and $L \ne 0$, then $\sT(\cT_2 \cup \{ L \})=\cT_1$.
In particular, we have $\sT(\cT_2 \cup \{ S \})=\cT_1$.
\item[(3)]
If $L_1,L_2 \in \cC$ and $L_1,L_2 \ne 0$, then we have $\Hom_A(L_1,L_2) \ne 0$.
\end{itemize}
\end{Lem}

\begin{proof}
There exists an indecomposable direct summand $M_1$ of $M$
such that $N$ is the left mutation of $M$ at $M_1$.
We decompose $M$ as $M_1 \oplus M_2$,
then we obtain $\Fac M_2=\cT_2$.

(1)
It is clear that $S \in \cT_1$.
To prove that $S \in \cF_2$, it is sufficient to show that $\Hom_A(M_2,S)=0$,
which follows from Lemma \ref{well-def} (3).
Thus, $S$ belongs to $\cF_2$; hence, we have $S \in \cC$.

(2)
Because $L \in \cF_2$ and $L \ne 0$, we get $L \notin \cT_2$.
Thus, we have $\sT(\cT_2 \cup \{ L \}) \supsetneq \cT_2$.
By the assumption $L \in \cT_1$, we also have 
$\cT_1 \supset \sT(\cT_2 \cup \{ L \}) \supsetneq \cT_2$.
By Lemma \ref{no_tors_between}, $\sT(\cT_2 \cup \{ L \})$ must coincide with $\cT_1$.
In particular, we obtain that $\sT(\cT_2 \cup \{ S \})=\cT_1$ by (1).

(3)
Because $L_2 \in \cC \subset \cF_2$, we have $\Hom_A(\cT_2,L_2)=0$.
We assume that $\Hom_A(L_1,L_2)=0$, and deduce a contradiction. 
In this case, (2) implies that $\Hom_A(\cT_1,L_2)=0$, which is false,
because $L_2 \in \cT_1$ and $L_2 \ne 0$.
Thus, we have $\Hom_A(L_1,L_2) \ne 0$.
\end{proof}

\begin{Prop}\label{brick_label_prop}
In the setting above, the following assertions hold.
\begin{itemize}
\item[(1)]
There exists a unique brick in $\cC$, and it is $S$.
\item[(2)]
The arrow $M' \to N'$ in the exchange quiver of $\stitilt A$
is also labeled with $S$.
\item[(3)]
The subcategory $\cC$ coincides with $\Filt S$.
\end{itemize}
\end{Prop}

\begin{proof}
(1)
Let $L$ be a brick in $\cC$.
By Lemma \ref{brick_label_lem} (1), $S$ is a brick belonging to $\cC$,
so it suffices to show that $L \cong S$.
By Lemma \ref{brick_label_lem} (3), there exists nonzero maps 
$f \colon L \to S$ and $g \colon S \to L$.
By Lemma \ref{Schur} (1), $f$ is surjective.
Thus, the map $gf \colon L \to L$ is nonzero, and it is an isomorphism, 
because $L$ is a brick.
Since $f$ is surjective, we have $L \cong S$.

(2)
Let $S'$ be the brick on the arrow $M' \to N'$ in the exchange quiver of $\stitilt A$.
The dual of Lemma \ref{brick_label_lem} (1) yields $S' \in \cC$.
Then (1) implies that $S' \cong S$.

(3)
Clearly, $\Filt S \subset \cC$ holds.

It remains to show that $\cC \subset \Filt S$.
Let $L \in \cC$.
If $L=0$, then $L \in \Filt S$.
Thus, we may assume $L \ne 0$,
and we use induction on $\dim_K L$. 
By Lemma \ref{brick_label_lem} (1) and (3), 
there exists a nonzero map $f \colon L \to S$,
and by Lemma \ref{Schur} (1), $f$ is surjective.

If $\dim_K L=\dim_K S$, then we have $L \cong S \in \Filt S$.

If $\dim_K L>\dim_K S$, there exists 
a short exact sequence $0 \to \Ker f \to L \to S \to 0$.
By Lemma \ref{Schur} (1), $\Ker f \in \cT_1$ holds. 
We also have $\Ker f \in \cF_2$, because $L \in \cF_2$.
Thus, we have $\Ker f \in \cC$, 
and then the induction hypothesis implies that $\Ker f \in \Filt S$,
since $\dim_K \Ker f < \dim_K L$.
Therefore, $L \in \Filt S$ is obtained.

The induction process is now complete.
\end{proof}

\begin{proof}[Proof of Theorem \ref{brick_label_coincide}]
It is an immediate result of Proposition \ref{brick_label_prop} (2). 
\end{proof}

We remark that the subcategory $\cC=\Filt S$ of $\mod A$
in Proposition \ref{brick_label_prop}
is a \textit{wide subcategory} of $\mod A$, see Subsection \ref{wide_mod}.

\begin{Ex}\label{sttilt_ex}
Let $A$ be the path algebra of the quiver $1 \to 2 \to 3$.
Figure \ref{A3} below is the exchange quivers 
of $\sttilt A$ and $\stitilt A$ labeled with bricks.
The bricks in $\ind(M/{\rad_B M}) \in \fLsbrick A$ 
for each $M \in \sttilt A$ 
and the bricks in $\ind(\soc_B M) \in \fRsbrick A$ 
for each $M \in \stitilt A$ are denoted by bold letters.
\begin{figure}[h]
\begin{align*}
\begin{xy}
(-42, 27)*+{
\begin{smallmatrix} \rthr \end{smallmatrix},
\begin{smallmatrix} \rtwo \\ 3 \end{smallmatrix},
\begin{smallmatrix} \rone \\ 2 \\ 3 \end{smallmatrix}
}= "01",
(-42,  0)*+{
\begin{smallmatrix} \rtwo \\ \rthr \end{smallmatrix},
\begin{smallmatrix} \rone \\ 2 \\ 3 \end{smallmatrix},
\begin{smallmatrix} 2 \end{smallmatrix}
}= "02",
(  0, 27)*+{
\begin{smallmatrix} \rthr \end{smallmatrix},
\begin{smallmatrix} \rone \\ \rtwo \\ 3 \end{smallmatrix},
\begin{smallmatrix} 1 \end{smallmatrix}
}= "03",
(-60,  0)*+{
\begin{smallmatrix} \rthr \end{smallmatrix},
\begin{smallmatrix} \rtwo \\ 3 \end{smallmatrix}
}= "04",
(-18,  0)*+{
\begin{smallmatrix} \rone \\ \rtwo \\ \rthr
\end{smallmatrix},
\begin{smallmatrix} \rtwo \end{smallmatrix},
\begin{smallmatrix} 1 \\ 2 \end{smallmatrix}
}= "05",
(-42,-27)*+{
\begin{smallmatrix} \rtwo \\ \rthr \end{smallmatrix},
\begin{smallmatrix} 2 \end{smallmatrix}
}= "06",
(  0, 12)*+{
\begin{smallmatrix} \rone \\ \rtwo \\ \rthr
\end{smallmatrix},
\begin{smallmatrix} 1 \\ 2 \end{smallmatrix},
\begin{smallmatrix} 1 \end{smallmatrix}
}= "07",
( 42, 27)*+{
\begin{smallmatrix} \rthr \end{smallmatrix},
\begin{smallmatrix} \rone \end{smallmatrix}
}= "08",
(  0,-12)*+{
\begin{smallmatrix} \rtwo \end{smallmatrix},
\begin{smallmatrix} \rone \\ 2 \end{smallmatrix}
}= "09",
( 18,  0)*+{
\begin{smallmatrix} \rone \\ \rtwo \end{smallmatrix},
\begin{smallmatrix} 1 \end{smallmatrix}
}= "10",
( 60,  0)*+{
\begin{smallmatrix} \rthr \end{smallmatrix}
}= "11",
(  0,-27)*+{
\begin{smallmatrix} \rtwo \end{smallmatrix}
}= "12",
( 42,  0)*+{
\begin{smallmatrix} \rone \end{smallmatrix}
}= "13",
( 42, -27)*+{0}= "14",
\ar_{\begin{smallmatrix} 3 \end{smallmatrix}} "01";"02" 
\ar^{\begin{smallmatrix} 2 \end{smallmatrix}} "01";"03" 
\ar_{\begin{smallmatrix} 1 \end{smallmatrix}} "01";"04"
\ar^{\begin{smallmatrix} 2 \\ 3 \end{smallmatrix}} "02";"05" 
\ar_{\begin{smallmatrix} 1 \end{smallmatrix}} "02";"06"
\ar_{\begin{smallmatrix} 3 \end{smallmatrix}} "03";"07" 
\ar^{\begin{smallmatrix} 1 \\ 2 \end{smallmatrix}} "03";"08"
\ar_{\begin{smallmatrix} 3 \end{smallmatrix}} "04";"06" 
\ar@{-} "04";(-60,-36) 
\ar@{-}^{\begin{smallmatrix} 2 \end{smallmatrix}} (-60,-36);(60,-36) 
\ar(60,-36);"11"
\ar^{\begin{smallmatrix} 2 \end{smallmatrix}} "05";"07" 
\ar_{\begin{smallmatrix} 1 \\ 2 \\ 3 \end{smallmatrix}} "05";"09"
\ar^{\begin{smallmatrix} 2 \\ 3 \end{smallmatrix}} "06";"12"
\ar^{\begin{smallmatrix} 1 \\ 2 \\ 3 \end{smallmatrix}} "07";"10"
\ar^{\begin{smallmatrix} 1 \end{smallmatrix}} "08";"11" 
\ar_{\begin{smallmatrix} 3 \end{smallmatrix}} "08";"13"
\ar_{\begin{smallmatrix} 2 \end{smallmatrix}} "09";"10" 
\ar_{\begin{smallmatrix} 1 \end{smallmatrix}} "09";"12"
\ar^{\begin{smallmatrix} 1 \\ 2 \end{smallmatrix}} "10";"13"
\ar^{\begin{smallmatrix} 3 \end{smallmatrix}} "11";"14"
\ar^{\begin{smallmatrix} 2 \end{smallmatrix}} "12";"14"
\ar_{\begin{smallmatrix} 1 \end{smallmatrix}} "13";"14"
\end{xy}
\end{align*}
\begin{align*}
\begin{xy}
(-42, 27)*+{0}= "01",
(-42,  0)*+{
\begin{smallmatrix} \rthr \end{smallmatrix}
}= "02",
(  0, 27)*+{
\begin{smallmatrix} \rtwo \end{smallmatrix}
}= "03",
(-60,  0)*+{
\begin{smallmatrix} \rone \end{smallmatrix}
}= "04",
(-18,  0)*+{
\begin{smallmatrix} 3 \end{smallmatrix},
\begin{smallmatrix} \rtwo \\ \rthr \end{smallmatrix}
}= "05",
(-42,-27)*+{
\begin{smallmatrix} \rthr \end{smallmatrix},
\begin{smallmatrix} \rone \end{smallmatrix}
}= "06",
(  0, 12)*+{
\begin{smallmatrix} 2 \\ \rthr \end{smallmatrix},
\begin{smallmatrix} \rtwo \end{smallmatrix}
}= "07",
( 42, 27)*+{
\begin{smallmatrix} 2 \end{smallmatrix},
\begin{smallmatrix} \rone \\ \rtwo \end{smallmatrix}
}= "08",
(  0,-12)*+{
\begin{smallmatrix} 3 \end{smallmatrix},
\begin{smallmatrix} 2 \\ 3 \end{smallmatrix},
\begin{smallmatrix} \rone \\ \rtwo \\ \rthr
\end{smallmatrix}
}= "09",
( 18,  0)*+{
\begin{smallmatrix} 2 \\ 3 \end{smallmatrix},
\begin{smallmatrix} \rone \\ \rtwo \\ \rthr
\end{smallmatrix},
\begin{smallmatrix} \rtwo \end{smallmatrix}
}= "10",
( 60,  0)*+{
\begin{smallmatrix} 1 \\ \rtwo \end{smallmatrix},
\begin{smallmatrix} \rone \end{smallmatrix}
}= "11",
(  0,-27)*+{
\begin{smallmatrix} 3 \end{smallmatrix},
\begin{smallmatrix} 1 \\ \rtwo \\ \rthr \end{smallmatrix},
\begin{smallmatrix} \rone \end{smallmatrix}
}= "12",
( 42,  0)*+{
\begin{smallmatrix} 2 \end{smallmatrix},
\begin{smallmatrix} 1 \\ 2 \\ \rthr \end{smallmatrix},
\begin{smallmatrix} \rone \\ \rtwo \end{smallmatrix}
}= "13",
( 42,-27)*+{
\begin{smallmatrix} 1 \\ 2 \\ \rthr \end{smallmatrix},
\begin{smallmatrix} 1 \\ \rtwo \end{smallmatrix},
\begin{smallmatrix} \rone \end{smallmatrix}
}= "14",
\ar_{\begin{smallmatrix} 3 \end{smallmatrix}} "01";"02" 
\ar^{\begin{smallmatrix} 2 \end{smallmatrix}} "01";"03" 
\ar_{\begin{smallmatrix} 1 \end{smallmatrix}} "01";"04"
\ar^{\begin{smallmatrix} 2 \\ 3 \end{smallmatrix}} "02";"05" 
\ar_{\begin{smallmatrix} 1 \end{smallmatrix}} "02";"06"
\ar_{\begin{smallmatrix} 3 \end{smallmatrix}} "03";"07" 
\ar^{\begin{smallmatrix} 1 \\ 2 \end{smallmatrix}} "03";"08"
\ar_{\begin{smallmatrix} 3 \end{smallmatrix}} "04";"06" 
\ar@{-} "04";(-60,-36) 
\ar@{-}^{\begin{smallmatrix} 2 \end{smallmatrix}} (-60,-36);(60,-36) 
\ar(60,-36);"11"
\ar^{\begin{smallmatrix} 2 \end{smallmatrix}} "05";"07" 
\ar_{\begin{smallmatrix} 1 \\ 2 \\ 3 \end{smallmatrix}} "05";"09"
\ar^{\begin{smallmatrix} 2 \\ 3 \end{smallmatrix}} "06";"12"
\ar^{\begin{smallmatrix} 1 \\ 2 \\ 3 \end{smallmatrix}} "07";"10"
\ar^{\begin{smallmatrix} 1 \end{smallmatrix}} "08";"11" 
\ar_{\begin{smallmatrix} 3 \end{smallmatrix}} "08";"13"
\ar_{\begin{smallmatrix} 2 \end{smallmatrix}} "09";"10" 
\ar_{\begin{smallmatrix} 1 \end{smallmatrix}} "09";"12"
\ar^{\begin{smallmatrix} 1 \\ 2 \end{smallmatrix}} "10";"13"
\ar^{\begin{smallmatrix} 3 \end{smallmatrix}} "11";"14"
\ar^{\begin{smallmatrix} 2 \end{smallmatrix}} "12";"14"
\ar_{\begin{smallmatrix} 1 \end{smallmatrix}} "13";"14"
\end{xy}
\end{align*}
\caption{The exchange quivers of $\sttilt A$ and $\stitilt A$}\label{A3}
\end{figure}
\end{Ex}

Labeling the exchange quiver of $\sttilt A$ with bricks
is originally considered as \textit{layer labeling}
in the case that $A$ is 
a preprojective algebra of Dynkin type $\varDelta$ by 
Iyama--Reading--Reiten--Thomas \cite{IRRT}. 
The definition of their layer labeling uses
the bijection between $\sttilt A$ and the Coxeter group of $\varDelta$
given by Mizuno \cite[Theorem 2.21]{Mizuno},
so it is rather different from the definition of our labeling.
However, 
for the preprojective algebras of Dynkin type, 
layers are precisely bricks,
and the layer on each arrow $M \to N$ 
belongs to $\cC$ in
Proposition \ref{brick_label_prop}, see \cite[Theorem 4.1]{IRRT}.
Therefore, their layer labeling completely coincides with our labeling.

\begin{Ex}\label{preproj_ex}
Let $A$ be the preprojective algebra of type $\mathbb{A}_3$.
This is given by the following quiver and relations:
\begin{align*}
\begin{xy}
( 0,0) *+{1} = "1",
(10,0) *+{2} = "2",
(20,0) *+{3} = "3"
\ar @<1mm>  ^{\alpha} "1"; "2"
\ar @<1mm>  ^{\beta}  "2"; "1"
\ar @<1mm>  ^{\gamma} "2"; "3"
\ar @<1mm>  ^{\delta} "3"; "2"
\end{xy}, \quad
\alpha\beta=0, \quad \beta\alpha=\gamma\delta, \quad \delta\gamma = 0.
\end{align*}
Figure \ref{A3_preproj} below is the exchange quiver
of $\sttilt A$ labeled with bricks. 
The bricks in the semibrick $\ind(M/{\rad_B M}) \in \fLsbrick A$ 
for each $M \in \sttilt A$ are denoted by bold letters.
Compare this labeling with the layer labeling of the exchange quiver of 
$\sttilt A$ \cite[Figure 2]{IRRT}.
\begin{figure}[h]
\begin{align*}
\begin{xy}
(  0,132) *+{
\begin{smallmatrix} \rone&& \\ &2& \\ &&3 \end{smallmatrix},
\begin{smallmatrix} &\rtwo& \\ 1&&3 \\ &2& \end{smallmatrix},
\begin{smallmatrix} &&\rthr \\ &2& \\ 1&& \end{smallmatrix}
}="1234",
( 32,114) *+{
\begin{smallmatrix} 2& \\ &3 \end{smallmatrix},
\begin{smallmatrix} &\rtwo& \\ \rone&&3 \\ &2& \end{smallmatrix},
\begin{smallmatrix} &&\rthr \\ &2& \\ 1&& \end{smallmatrix}
}="2134",
(  0,114) *+{
\begin{smallmatrix} \rone&& \\ &\rtwo& \\ &&3 \end{smallmatrix},
\begin{smallmatrix} 1&&3 \\ &2& \end{smallmatrix},
\begin{smallmatrix} &&\rthr \\ &\rtwo& \\ 1&& \end{smallmatrix}
}="1324",
(-32,114) *+{
\begin{smallmatrix} \rone&& \\ &2& \\ &&3 \end{smallmatrix},
\begin{smallmatrix} &\rtwo& \\ 1&&\rthr \\ &2& \end{smallmatrix},
\begin{smallmatrix} &2 \\ 1& \end{smallmatrix}
}="1243",
( 64, 90) *+{
\begin{smallmatrix} \rtwo& \\ &3 \end{smallmatrix},
\begin{smallmatrix} &3 \\ 2& \end{smallmatrix},
\begin{smallmatrix} &&\rthr \\ &\rtwo& \\ \rone&& \end{smallmatrix}
}="2314",
( 32, 90) *+{
\begin{smallmatrix} 3 \end{smallmatrix},
\begin{smallmatrix} \rone&&3 \\ &2& \end{smallmatrix},
\begin{smallmatrix} &&\rthr \\ &\rtwo& \\ 1&& \end{smallmatrix}
}="3124",
(  0, 90) *+{
\begin{smallmatrix} 2& \\ &3 \end{smallmatrix},
\begin{smallmatrix} &\rtwo& \\ \rone&&\rthr \\ &2& \end{smallmatrix},
\begin{smallmatrix} &2 \\ 1& \end{smallmatrix}
}="2143",
(-32, 90) *+{
\begin{smallmatrix} \rone&& \\ &\rtwo& \\ &&3 \end{smallmatrix},
\begin{smallmatrix} 1&&\rthr \\ &2& \end{smallmatrix},
\begin{smallmatrix} 1 \end{smallmatrix}
}="1342",
(-64, 90) *+{
\begin{smallmatrix} \rone&& \\ &\rtwo& \\ &&\rthr \end{smallmatrix},
\begin{smallmatrix} 1& \\ &2 \end{smallmatrix},
\begin{smallmatrix} &\rtwo \\ 1& \end{smallmatrix}
}="1423",
( 70, 66) *+{
\begin{smallmatrix} \rtwo& \\ &3 \end{smallmatrix},
\begin{smallmatrix} &\rthr \\ 2& \end{smallmatrix}
}="2341",
( 42, 66) *+{
\begin{smallmatrix} 3 \end{smallmatrix},
\begin{smallmatrix} &3 \\ 2& \end{smallmatrix},
\begin{smallmatrix} &&\rthr \\ &\rtwo& \\ \rone&& \end{smallmatrix}
}="3214",
( 14, 66) *+{
\begin{smallmatrix} 3 \end{smallmatrix},
\begin{smallmatrix} \rone&&\rthr \\ &\rtwo& \end{smallmatrix},
\begin{smallmatrix} 1 \end{smallmatrix}
}="3142",
(-14, 66) *+{
\begin{smallmatrix} \rtwo& \\ &\rthr \end{smallmatrix},
\begin{smallmatrix} 2 \end{smallmatrix},
\begin{smallmatrix} &\rtwo \\ \rone& \end{smallmatrix}
}="2413",
(-42, 66) *+{
\begin{smallmatrix} \rone&& \\ &\rtwo& \\ &&\rthr \end{smallmatrix},
\begin{smallmatrix} 1& \\ &2 \end{smallmatrix},
\begin{smallmatrix} 1 \end{smallmatrix}
}="1432",
(-70, 66) *+{
\begin{smallmatrix} \rone& \\ &2 \end{smallmatrix},
\begin{smallmatrix} &\rtwo \\ 1& \end{smallmatrix}
}="4123",
( 64, 42) *+{
\begin{smallmatrix} 3 \end{smallmatrix},
\begin{smallmatrix} &\rthr \\ \rtwo& \end{smallmatrix}
}="3241",
( 32, 42) *+{
\begin{smallmatrix} \rtwo& \\ &\rthr \end{smallmatrix},
\begin{smallmatrix} 2 \end{smallmatrix}
}="2431",
(  0, 42) *+{
\begin{smallmatrix} \rthr \end{smallmatrix},
\begin{smallmatrix} \rone \end{smallmatrix}
}="3412",
(-32, 42) *+{
\begin{smallmatrix} 2 \end{smallmatrix},
\begin{smallmatrix} &\rtwo \\ \rone& \end{smallmatrix}
}="4213",
(-64, 42) *+{
\begin{smallmatrix} \rone& \\ &\rtwo \end{smallmatrix},
\begin{smallmatrix} 1 \end{smallmatrix}
}="4132",
( 32, 18) *+{
\begin{smallmatrix} \rthr \end{smallmatrix}
}="3421",
(  0, 18) *+{
\begin{smallmatrix} \rtwo \end{smallmatrix}
}="4231",
(-32, 18) *+{
\begin{smallmatrix} \rone \end{smallmatrix}
}="4312",
(  0,  0) *+{
0
}="4321",
\ar^{\begin{smallmatrix} 1 \end{smallmatrix}} "1234";"2134"
\ar^{\begin{smallmatrix} 2 \end{smallmatrix}} "1234";"1324"
\ar_{\begin{smallmatrix} 3 \end{smallmatrix}} "1234";"1243"
\ar^{\begin{smallmatrix} &2 \\ 1& \end{smallmatrix}} "2134";"2314"
\ar^(.3){\begin{smallmatrix} 3 \end{smallmatrix}} "2134";"2143"
\ar_(.25){\begin{smallmatrix} 1& \\ &2 \end{smallmatrix}} "1324";"3124"
\ar^(.25){\begin{smallmatrix} &3 \\ 2& \end{smallmatrix}} "1324";"1342"
\ar_(.3){\begin{smallmatrix} 1 \end{smallmatrix}} "1243";"2143"
\ar_{\begin{smallmatrix} 2& \\ &3 \end{smallmatrix}} "1243";"1423"
\ar^{\begin{smallmatrix} &&3 \\ &2& \\ 1&& \end{smallmatrix}} "2314";"2341"
\ar_{\begin{smallmatrix} 2 \end{smallmatrix}} "2314";"3214"
\ar^{\begin{smallmatrix} 1 \end{smallmatrix}} "3124";"3214"
\ar_{\begin{smallmatrix} &3 \\ 2& \end{smallmatrix}} "3124";"3142"
\ar^(.3){\begin{smallmatrix} &2& \\ 1&&3 \end{smallmatrix}} "2143";"2413"
\ar_(.3){\begin{smallmatrix} 1& \\ &2 \end{smallmatrix}} "1342";"3142"
\ar_{\begin{smallmatrix} 3 \end{smallmatrix}} "1342";"1432"
\ar^{\begin{smallmatrix} 2 \end{smallmatrix}} "1423";"1432"
\ar_{\begin{smallmatrix} 1&& \\ &2& \\ &&3 \end{smallmatrix}} "1423";"4123"
\ar^{\begin{smallmatrix} 2 \end{smallmatrix}} "2341";"3241"
\ar^(.3){\begin{smallmatrix} 3 \end{smallmatrix}} "2341";"2431"
\ar_(.2){\begin{smallmatrix} &&3 \\ &2& \\ 1&& \end{smallmatrix}} "3214";"3241"
\ar^(.25){\begin{smallmatrix} 1&&3 \\ &2& \end{smallmatrix}} "3142";"3412"
\ar_(.3){\begin{smallmatrix} &2 \\ 1& \end{smallmatrix}} "2413";"2431"
\ar^(.3){\begin{smallmatrix} 2& \\ &3 \end{smallmatrix}} "2413";"4213"
\ar^(.2){\begin{smallmatrix} 1&& \\ &2& \\ &&3 \end{smallmatrix}} "1432";"4132"
\ar_(.3){\begin{smallmatrix} 1 \end{smallmatrix}} "4123";"4213"
\ar_{\begin{smallmatrix} 2 \end{smallmatrix}} "4123";"4132"
\ar^{\begin{smallmatrix} &3 \\ 2& \end{smallmatrix}} "3241";"3421"
\ar^(.25){\begin{smallmatrix} 2& \\ &3 \end{smallmatrix}} "2431";"4231"
\ar_(.25){\begin{smallmatrix} 1 \end{smallmatrix}} "3412";"3421"
\ar^(.25){\begin{smallmatrix} 3 \end{smallmatrix}} "3412";"4312"
\ar_(.25){\begin{smallmatrix} &2 \\ 1& \end{smallmatrix}} "4213";"4231"
\ar_{\begin{smallmatrix} 1& \\ &2 \end{smallmatrix}} "4132";"4312"
\ar^{\begin{smallmatrix} 3 \end{smallmatrix}} "3421";"4321"
\ar^{\begin{smallmatrix} 2 \end{smallmatrix}} "4231";"4321"
\ar_{\begin{smallmatrix} 1 \end{smallmatrix}} "4312";"4321"
\end{xy}
\end{align*}
\caption{The exchange quiver of $\sttilt A$}\label{A3_preproj}
\end{figure}
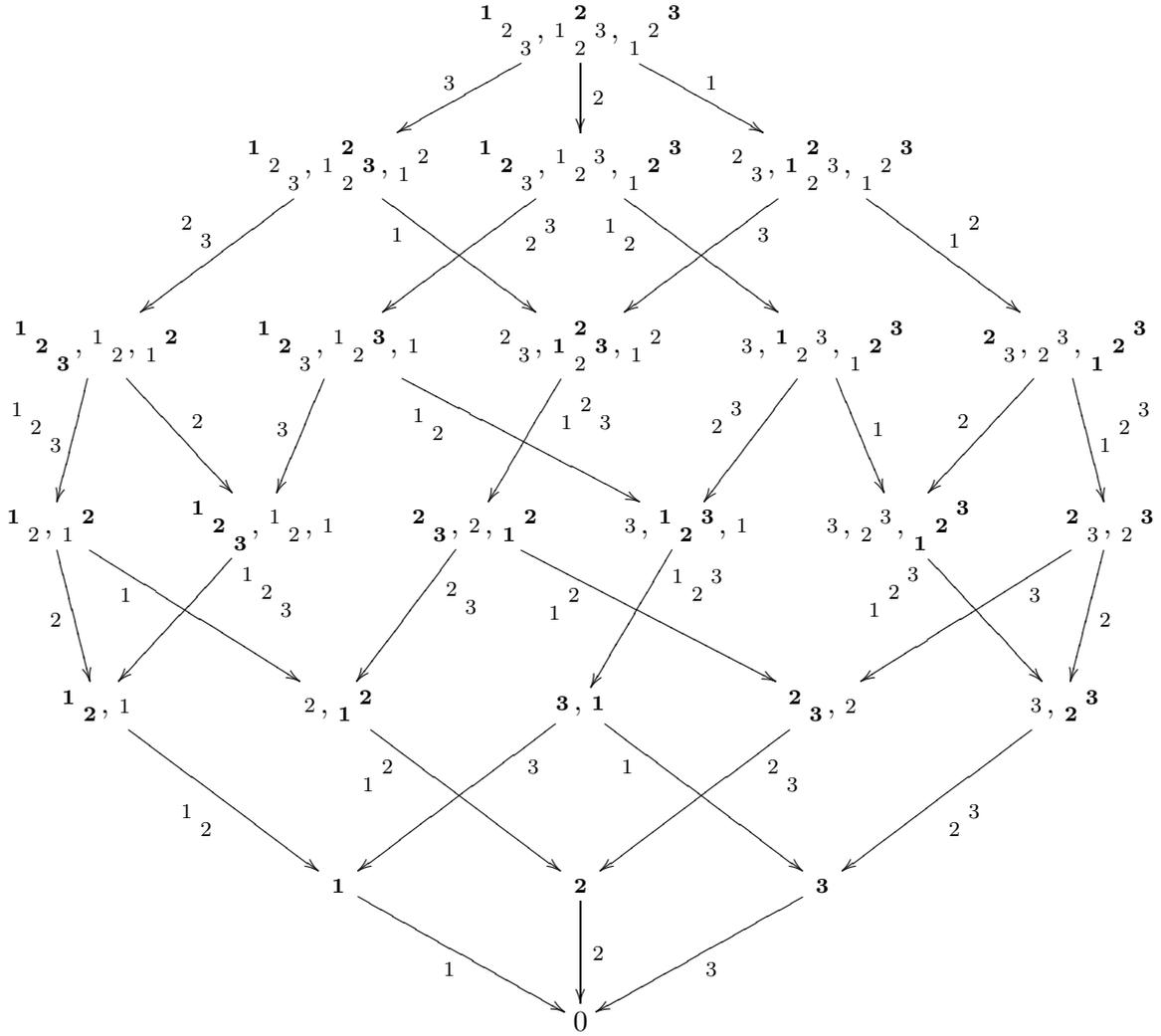
\end{Ex}

Now we recall a result of Jasso \cite{Jasso} on \textit{reductions}
of support $\tau$-tilting modules.

Fix $U \in \trigid A$.
For $M \in \mod A$, we have a canonical exact sequence
$0 \to \mathsf{t}M \to M \to \mathsf{f}M \to 0$ 
with $\mathsf{t}M \in \Fac U$ and $\mathsf{f}M \in U^\perp$.

There uniquely exists the \textit{Bongartz completion} $T$ of the $\tau$-rigid module $U$,
that is, the module $T \in \sttilt A$ satisfying 
$U \in \add T$ and $\Fac T={^\perp(\tau U)}$ \cite[Theorem 2.10]{AIR}.
We define two subsets 
$\sttilt_U A \subset \sttilt A$ and
$\ftors_U A \subset \ftors A$ by
\begin{align*}
\sttilt_U A:=\{ M \in \sttilt A \mid U \in \add M \}, \quad
\ftors_U A:=\{\cT \in \ftors A \mid \Fac U \subset \cT \subset {^\perp(\tau U)}\}.
\end{align*}
Let $M \in \sttilt A$, then
$M \in \sttilt_U A$ holds if and only if $\Fac U \subset \Fac M \subset {^\perp(\tau U)}$
\cite[Proposition 2.9]{AIR}. 
Therefore, the bijection $\Fac \colon \sttilt A \to \ftors A$ is restricted to
a bijection $\Fac \colon \sttilt_U A \to \ftors_U A$.

We consider the algebra $C:=\End_A(T)/[U]$,
where $[U]$ denotes the ideal of $\End_A(T)$ consisting of the morphisms
factoring through some objects in $\add U$.
Then, there is an equivalence 
$\Phi:=\Hom_A(T,?) \colon U^\perp \cap {^\perp (\tau U)} \to \mod C$ \cite[Theorem 3.8]{Jasso}.
We also remark that $M \in \sttilt_U A$ implies
$\Fac M \cap U^\perp \subset U^\perp \cap {^\perp (\tau U)}$.

Under this preparation, the following result holds.

\begin{Prop}\label{reduction}
\cite[Theorems 3.14, 3.16, Corollary 3.17]{Jasso}
There exist bijections
$\ftors_U A \to \ftors C$ and $\sttilt_U A \to \sttilt C$ given by
$\cT \mapsto \Phi(\cT \cap U^\perp)$ and $M \mapsto \Phi(\mathsf{f}M)$.
Moreover, they are compatible with mutations and satisfy the following commutative diagram:
\begin{align*}
\begin{xy}
( 0,  8) *+{\sttilt_U A}  ="01",
(48,  8) *+{\sttilt C}    ="11",
( 0, -8) *+{\ftors_U A}   ="00",
(48, -8) *+{\ftors C}     ="10",
\ar^{\Fac} "11";"10"
\ar^{\Fac} "01";"00"
\ar^{M \mapsto \Phi(\mathsf{f} M)} "00";"10"
\ar^{\cT \mapsto \Phi(\cT \cap U^\perp)} "01";"11"
\end{xy}.
\end{align*}
\end{Prop}

This proposition says that
the exchange quiver of $\sttilt C$ is isomorphic to
the full subquiver of the exchange quiver of $\sttilt A$
consisting of the elements in $\sttilt_U A$.
We have the following assertion on the labels of 
these exchange quivers.

\begin{Th}\label{label_reduc}
Let $M \to N$ be an arrow labeled with a brick $S$ in the exchange quiver of $\sttilt A$.
Assume $M,N \in \sttilt_U A$ and set $Y:=\Phi(\mathsf{f}M)$ and $Z:=\Phi(\mathsf{f}N)$.
Then the corresponding arrow $Y \to Z$ in the exchange quiver of $\sttilt C$
is labeled with a brick $F(S)$.
\end{Th}

\begin{proof}
By Proposition \ref{brick_label_prop} (1),
the assertion holds if the brick $\Phi(S)$ belongs to 
$\Fac Y \cap (\Fac Z)^\perp$.
Proposition \ref{brick_label_prop} (1) also implies that
$S$ is a brick in $\cT_1 \cap (\cT_2)^\perp$,
where $\cT_1:=\Fac M$ and $\cT_2:=\Fac N$.
By Proposition \ref{reduction}, 
we have $\Fac Y=\Phi(\cT_1 \cap U^\perp)$ and $\Fac Z=\Phi(\cT_2 \cap U^\perp)$,
so it suffices to show 
$\Phi(S) \in \Phi(\cT_1 \cap U^\perp) \cap \Phi(\cT_2 \cap U^\perp)^\perp$.

We first prove $\Phi(S) \in \Phi(\cT_1 \cap U^\perp)$.
It is enough to show $S \in \cT_1 \cap U^\perp$.
Since $U \in \cT_2$, we get $(\cT_2)^\perp \subset U^\perp$.
We have seen $S \in \cT_1 \cap (\cT_2)^\perp$, so we get $S \in \cT_1 \cap U^\perp$.
Therefore, $\Phi(S) \in \Phi(\cT_1 \cap U^\perp)$.

Next, we show $\Phi(S) \in \Phi(\cT_2 \cap U^\perp)^\perp$, 
which is equivalent to $\Hom_C(\Phi(\cT_2 \cap U^\perp),\Phi(S))=0$ by definition.
By using the equivalence
$\Phi=\Hom_A(T,?) \colon U^\perp \cap {^\perp (\tau U)} \to \mod C$,
it suffices to get $\Hom_A(\cT_2 \cap U^\perp,S)=0$,
which follows from $S \in (\cT_2)^\perp$.
Therefore, $\Phi(S) \in \Phi(\cT_2 \cap U^\perp)^\perp$.

Now, we have proved
$\Phi(S) \in F(\cT_1 \cap U^\perp) \cap \Phi(\cT_2 \cap U^\perp)^\perp$.
Then, the argument in the beginning gives the assertion.
\end{proof}

We give an example of Theorem \ref{label_reduc}.

\begin{Ex}
We use the setting of Example \ref{sttilt_ex}.
Define $U,T_1,T_2 \in \mod A$ as follows:
\begin{align*}
U:=\left( \begin{smallmatrix} 2 \end{smallmatrix} \right), \quad 
T_1:=\left( \begin{smallmatrix} 1 \\ 2 \\ 3 \end{smallmatrix} \right), \quad
T_2:=\left( \begin{smallmatrix} 2 \\ 3 \end{smallmatrix} \right).
\end{align*}
Then it is easy to see that $U \in \trigid A$
and that the Bongartz completion $T$ of $U$ is $T_2 \oplus T_1 \oplus U \in \sttilt A$.
The algebra $C=\End_A(T)/[U]$ is isomorphic to the path algebra of the quiver $1 \to 2$,
where $T_1,T_2$ correspond to the vertices $1,2$.
In the following two quivers, 
the left one is the full subquiver of the exchange quiver of $\sttilt A$
consisting of the elements in $\sttilt_U A$,
and the right one is the exchange quiver of $\sttilt C$,
and both of them are labeled with bricks:
\begin{align*}
\begin{xy}
(-42, 12)*+{
\begin{smallmatrix} \rtwo \\ \rthr \end{smallmatrix},
\begin{smallmatrix} \rone \\ 2 \\ 3 \end{smallmatrix},
\begin{smallmatrix} 2 \end{smallmatrix}
}= "02",
(-18, 12)*+{
\begin{smallmatrix} \rone \\ \rtwo \\ \rthr
\end{smallmatrix},
\begin{smallmatrix} \rtwo \end{smallmatrix},
\begin{smallmatrix} 1 \\ 2 \end{smallmatrix}
}= "05",
(-42,-15)*+{
\begin{smallmatrix} \rtwo \\ \rthr \end{smallmatrix},
\begin{smallmatrix} 2 \end{smallmatrix}
}= "06",
(  0,  0)*+{
\begin{smallmatrix} \rtwo \end{smallmatrix},
\begin{smallmatrix} \rone \\ 2 \end{smallmatrix}
}= "09",
(  0,-15)*+{
\begin{smallmatrix} \rtwo \end{smallmatrix}
}= "12",
\ar^{\begin{smallmatrix} 2 \\ 3 \end{smallmatrix}} "02";"05" 
\ar_{\begin{smallmatrix} 1 \end{smallmatrix}} "02";"06"
\ar_{\begin{smallmatrix} 1 \\ 2 \\ 3 \end{smallmatrix}} "05";"09"
\ar^{\begin{smallmatrix} 2 \\ 3 \end{smallmatrix}} "06";"12"
\ar_{\begin{smallmatrix} 1 \end{smallmatrix}} "09";"12"
\end{xy}, \quad
\begin{xy}
(-42, 12)*+{
\begin{smallmatrix} \rtwo \end{smallmatrix},
\begin{smallmatrix} \rone \\ 2 \end{smallmatrix}
}= "02",
(-18, 12)*+{
\begin{smallmatrix} \rone \\ \rtwo
\end{smallmatrix},
\begin{smallmatrix} 1 \end{smallmatrix}
}= "05",
(-42,-15)*+{
\begin{smallmatrix} \rtwo \end{smallmatrix}
}= "06",
(  0,  0)*+{
\begin{smallmatrix} \rone \end{smallmatrix}
}= "09",
(  0,-15)*+{
0
}= "12",
\ar^{\begin{smallmatrix} 2 \end{smallmatrix}} "02";"05" 
\ar_{\begin{smallmatrix} 1 \end{smallmatrix}} "02";"06"
\ar_{\begin{smallmatrix} 1 \\ 2 \end{smallmatrix}} "05";"09"
\ar^{\begin{smallmatrix} 2 \end{smallmatrix}} "06";"12"
\ar_{\begin{smallmatrix} 1 \end{smallmatrix}} "09";"12"
\end{xy}.
\end{align*}
We can easily check that the brick on each arrow of the left quiver
is sent to the brick on the corresponding arrow of the right quiver by
the functor $\Phi=\Hom_A(T,?)$.
\end{Ex}

\subsection{Realizing wide subcategories as module categories I}\label{wide_mod}

In this subsection, we study wide subcategories of the module category $\mod A$
by using our results on semibricks.
Our goal here is describing a left finite wide subcategory of $\mod A$
as the module category $\mod E$ of some explicitly given algebra $E$.

First, we give the definition of wide subcategories.

\begin{Def}
A full subcategory $\cW \subset \mod A$ is called a \textit{wide subcategory}
if $\cW$ is a subcategory and closed under kernels, cokernels, and 
extensions of $\mod A$.
We write $\wide A$ for the set of wide subcategories of $\mod A$.
\end{Def}

Clearly, a wide subcategory of $\mod A$ is precisely 
an abelian subcategory of $\mod A$ closed under extensions.
We recall the following important result deduced from \cite[1.2]{Ringel}.

\begin{Prop}\label{sbrick_wide}
The map $\Filt \colon \sbrick A \to \wide A$ is a bijection,
and its inverse $\wide A \to \sbrick A$ sends $\cW \in \wide A$ to
the set of isoclasses of simple objects in $\cW$.
\end{Prop}

It is easy to see that, if $\cS$ is a semibrick and 
$\cW=\Filt \cS$ is the corresponding wide subcategory,
then torsion classes $\sT(\cS)$ and $\sT(\cW)$ coincide.
Thus, we can define left finiteness and right finiteness
for wide subcategories as in the case of semibricks.
We write $\fLwide A$ for the set of left finite wide subcategories,
and $\fRwide A$ for the set of right finite wide subcategories.
Clearly, the bijection $\Filt \colon \sbrick A \to \wide A$ 
in Proposition \ref{sbrick_wide} 
is restricted to bijections $\fLsbrick A \to \fLwide A$ and $\fRsbrick A \to \fRwide A$.

Now, we define  $\sW_\sL(\cT)=\{ M \in \cT \mid 
\mbox{for any $f \colon L \to M$ in $\cT$, $\Ker f \in \cT$} \}$
for $\cT \in \tors A$ and 
and $\sW_\sR(\cF)=\{ M \in \cF \mid 
\mbox{for any $f \colon M \to N$ in $\cF$, $\Coker f \in \cF$} \}$
for $\cF \in \torf A$.
Each of them is a wide subcategory of $\mod A$, 
see \cite[Proposition 2.12]{IT} for the proof.
We have well-defined maps $\sW_\sL \colon \tors A \to \wide A$
and $\sW_\sR \colon \torf A \to \wide A$.
We recall the next result by Marks--\v{S}t\!'ov\'{i}\v{c}ek.

\begin{Prop}\label{tors_wide}
The operations $\sT$ and $\sW_\sL$ satisfy the following properties.
\begin{itemize}
\item[(1)] \cite[Proposition 3.3]{MS}
A composition $\sW_\sL \circ \sT \colon \wide A \to \wide A$ is the identity.
In particular, the map $\sT \colon \wide A \to \tors A$ is injective.
\item[(2)] \cite[Theorem 3.10]{MS}
The map $\sT \colon \fLwide A \to \ftors A$ is bijective.
The inverse is given by $\sW_\sL \colon \ftors A \to \fLwide A$.
\end{itemize}
\end{Prop}

We note that our Lemmas \ref{Schur}, \ref{sbrick_tors} and Proposition \ref{comm_3} are 
an analogue of Proposition \ref{tors_wide}.
We have $\ftors A \subset \sT(\wide A)=\sT(\sbrick A)$.

Summing up the result of Adachi--Iyama--Reiten, Marks--\v{S}t\!'ov\'{i}\v{c}ek and ours,
we have the next assertions.

\begin{Prop}\label{comm_4}
We have the following commutative diagrams of bijections:
\begin{align*}
\begin{xy}
(-40,  0) *+{\sttilt A} = "0",
(  0,  0) *+{\ftors A} = "1",
( 40, 16) *+{\fLwide A} = "2",
( 40,  0) *+{\fLsbrick A} = "3",
(-40, -8) = "4",
( 40, -8) = "5",
\ar@<1mm>^{\Fac} "0"; "1"
\ar@<1mm>^{\textup{Ext-proj's}} "1"; "0"
\ar@<1mm>^{\sW_\sL} "1"; "2"
\ar@<1mm>^{\sT} "2"; "1"
\ar@<1mm>^{\Filt} "3"; "2"
\ar@<1mm>^{\textup{simples}} "2"; "3"
\ar^{\sT} "3"; "1"
\ar@{-} "0"; "4"
\ar@{-}_{M \mapsto \ind(M/{\rad_B M})} "4"; "5" 
\ar "5"; "3"
\end{xy}, \\
\begin{xy}
(-40,  0) *+{\stitilt A} = "0",
(  0,  0) *+{\ftorf A} = "1",
( 40, 16) *+{\fRwide A} = "2",
( 40,  0) *+{\fRsbrick A} = "3",
(-40, -8) = "4",
( 40, -8) = "5",
\ar@<1mm>^{\Sub} "0"; "1"
\ar@<1mm>^{\textup{Ext-inj's}} "1"; "0"
\ar@<1mm>^{\sW_\sR} "1"; "2"
\ar@<1mm>^{\sF} "2"; "1"
\ar@<1mm>^{\Filt} "3"; "2"
\ar@<1mm>^{\textup{simples}} "2"; "3"
\ar^{\sF} "3"; "1"
\ar@{-} "0"; "4"
\ar@{-}_{M \mapsto \ind(\soc_B M)} "4"; "5" 
\ar "5"; "3"
\end{xy}.
\end{align*}
\end{Prop}

\begin{proof}
Propositions \ref{sttilt_ftors}, \ref{comm_3}, and \ref{tors_wide} imply the assertion. 
\end{proof}

Next we recall that a support $\tau$-tilting $A$-module $M \in \sttilt A$ 
is a tilting $A/{\ann M}$-module
\cite[Lemma 2.1, Proposition 2.2]{AIR}.
By Brenner--Butler's theorem, $\Hom_A(M,?) \colon \Fac M \to \Sub_B DM$ is an equivalence,
where $B:=\End_A(M)$.
It is also an equivalence of exact categories \cite[Proposition 3.2]{DIJ}
and allows us to regard the wide subcategory $\cW$ as a Serre subcategory of $\mod B$
as follows.
We use the notation in Definition \ref{summand_symbol} here.
Recall that $I=\{ i \in \{1,2,\ldots,m\} \mid N_i \ne 0 \}$.

\begin{Th}\label{wide_B}
In the setting of Definition \ref{summand_symbol}, let $M \in \sttilt A$.
Define $e_i \in B$ as the idempotent endomorphism 
$M_i \to M \to M_i$ for each $i=1,2,\ldots,m$
and set $e:=1-\sum_{i \in I} e_i$.
Then the equivalence $\Hom_A(M,?) \colon \Fac M \to \Sub_B DM$ 
is restricted to an equivalence 
\begin{align*}
\cW \cong \mod B/\langle e \rangle.
\end{align*}
\end{Th}

\begin{proof}
Since $\Hom_A(M,?) \colon \Fac M \to \Sub_B DM$ is
an equivalence of exact categories \cite[Proposition 3.2]{DIJ},
it is restricted to an equivalence $\cW=\Filt \cS \to \Filt_B \Hom_A(M,\cS)$.
Thus, it suffices to show that $\Hom_A(M,\cS)$ coincides with 
the set of simple $B/\langle e \rangle$-modules.
The cardinalities of these two sets coincide,
because the number of isoclasses of simple  
$\mod B/\langle e \rangle$-modules is $\# I$, 
and it is equal to $\# \cS$ by Lemma \ref{well-def} (4).
Therefore, we show that each element of $\Hom_A(M,\cS)$ 
is a simple $B/\langle e \rangle$-module;
then, we obtain the desired equivalence.
Let $S \in \cS$.

We first show that $\Hom_A(M,S)$ is a simple $B$-module.
Because $\Hom_A(M,S)$ is a nonzero $B$-module in $\Sub_B DM$,
we can take a simple submodule $Z$ of $\Hom_A(M,S)$.
It is enough to obtain $\Hom_A(M,S) \cong Z$.

Since $M \in \Sub_B DM$, the submodule $Z$ is also in $\Sub_B DM$. 
We define $f \colon Z \to \Hom_A(M,S)$ as the canonical inclusion.
A quasi-inverse of $\Hom_A(M,?) \colon \Fac M \to \Sub_B DM$
is given by $? \otimes_B M \colon \Sub_B DM \to \Fac M$, 
which sends $f$ to a nonzero homomorphism 
$f \otimes_B M \colon Z \otimes_B M \to \Hom_A(M,S) \otimes_B M$
in $\Fac M$.
Since $\Hom_A(M,S) \otimes_B M \cong S \in \cS$ and 
$Z \otimes_B M \in \Fac M = \sT(\cS)$,
the nonzero homomorphism $f \otimes_B M$ is surjective
and $\Ker f \in \sT(\cS) = \Fac M$ holds by Lemma \ref{Schur} (1).

Thus, there exists a short exact sequence
$0 \to \Ker f \to Z \otimes_B M \to S \to 0$ in $\Fac M$.
By using the equivalence $\Hom_A(M,?) \colon \Fac M \to \Sub_B DM$ 
of exact categories, 
we have a short exact sequence
$0 \to \Hom_A(M,\Ker f) \to Z \to \Hom_A(M,S) \to 0$ in $\Sub_B DM$.
Since $Z$ is a simple $B$-module and $\Hom_A(M,S) \ne 0$,
we obtain that $\Hom_A(M,S) \cong Z$
and that $\Hom_A(M,S)$ is a simple $B$-module.

Next, we show $\Hom_A(M,S) \in \mod B/\langle e \rangle$.
By Lemma \ref{well-def} (4), we can take $j \in I$ such that $S=N_j$.
Clearly, $\Hom_A(M,N_j) \in \mod B/\langle e \rangle$ is equivalent to 
$\Hom_A(M_i,N_j)=0$ for each $i \in \{1,2,\ldots,m\} \setminus I$.
By Lemma \ref{well-def} (3), the latter condition holds.
We have $\Hom_A(M,S) \in \mod B/\langle e \rangle$.

These two properties yield that 
each element of $\Hom_A(M,\cS)$ is a simple $B/\langle e \rangle$-module.
By the argument in the first paragraph, 
we get $\cW \cong \mod B/\langle e \rangle$.
\end{proof}

We remark that Theorem \ref{wide_B} follows also 
from Jasso's reductions of $\tau$-rigid modules.
In the setting of Theorem \ref{wide_B}, set $U:=\bigoplus_{i \notin I} M_i$.
Then, we can check that $M$ is the Bongartz completion of $U$. 
From \cite[Theorem 3.8]{Jasso}, 
we have an equivalence $\Hom_A(M,?) \colon U^\perp \cap {^\perp (\tau U)} \to \mod C$
with $C:=\End_A(M)/[U]$.
The algebra $C$ is clearly $B/\langle e \rangle$.
It follows from Theorem \ref{label_reduc} that the simple objects of
the category $U^\perp \cap {^\perp (\tau U)}$ are the elements in $\cS$.
Therefore, $U^\perp \cap {^\perp (\tau U)}$ coincides with $\cW$,
and we obtain the equivalence of Theorem \ref{wide_B}.

As an application, we prove the next assertion.
We write $\fwide A$ for the set of functorially finite wide subcategories of $\mod A$.
This is also deduced from \cite[Proposition 3.3, Lemma 3.8]{MS}.

\begin{Prop}\label{left_func}
We have an inclusion $\fLwide A \subset \fwide A$.
\end{Prop}

\begin{proof}
Let $\cW \in \fLwide A$ and take $M \in \sttilt A$ 
corresponding to $\cW$ in Proposition \ref{comm_4}.
Because $\Fac M \in \ftors A$,
it suffices to show that $\cW$ is functorially finite in $\Fac M$.
By Theorem \ref{wide_B},
this condition is equivalent to that $\mod B/\langle e \rangle$ 
is functorially finite in $\Sub_B DM$,
which holds true because $\mod B/\langle e \rangle$ is 
functorially finite in $\mod B \supset \Sub_B DM$.
\end{proof}

\subsection{Semibricks for factor algebras}\label{on_EJR}

Let $I \subset A$ be an ideal of a finite dimensional algebra $A$.
Then we obviously have $\sbrick A/I \subseteq \sbrick A$,
but the equality does not hold in general.

In this subsection, we first prove that $\sbrick A/I = \sbrick A$ holds
in the following condition given by Eisele--Janssens--Raedschelders \cite{EJR}:
\begin{quote}
the ideal $I \subset A$ is
generated by a set $X \subset Z(A) \cap \rad A$.
\end{quote}
Here, $Z(A)$ denotes the center of $A$.
Moreover, by using semibricks, we give another proof of their theorem, 
which gives a canonical bijection $\sttilt A \to \sttilt A/I$.

For $M \in \mod A$ and every $a \in Z(A)$, 
we have an endomorphism $(\cdot a) \colon M \to M$.
We can take a $K$-basis $a_1,a_2,\ldots, a_s$ 
of the vector subspace $KX \subset Z(A)$ generated by $X$,
then the image of the homomorphism 
$\begin{bmatrix} (\cdot a_1) & \cdots & (\cdot a_s) \end{bmatrix}$
coincides with $MI$.
Thus, the submodule $MI \subset M$ satisfies $MI \in \Fac M$.
 
From now on, we define $\overline{A}:=A/I$ and $\overline{M}:=M/MI$ for $M \in \mod A$.
The operation $\overline{?}$ gives a right exact functor 
$\overline{?} \colon \mod A \to \mod \overline{A}$.

First, we observe that $\sbrick \overline{A} = \sbrick A$.

\begin{Prop}
We have $\sbrick A=\sbrick \overline{A}$.
\end{Prop}

\begin{proof}
It is sufficient to show that $S \in \brick A$ implies $S \in \brick \overline{A}$.
By assumption, $I$ is generated by the set $X \subset Z(A) \cap \rad A$.
For each $x \in X$, 
the $A$-endomorphism $(\cdot x) \colon S \to S$ is nilpotent because $x \in \rad A$,
and this map must be zero since $S \in \brick A$,
so we have $Sx=0$.
Thus, $SI=0$ holds; hence, we have $S \in \brick \overline{A}$.
We have the assertion.
\end{proof}

In the rest of this subsection, 
we prove the following.

\begin{Th}\label{EJR_th}
We have the following assertions.
\begin{itemize}
\item[(1)]
We have $\fLsbrick A = \fLsbrick \overline{A}$.
\item[(2)]
\cite[Theorem 4.1]{EJR}
There exists a bijection $\sttilt A \to \sttilt \overline{A}$ given as 
$M \mapsto \overline{M}$, which satisfies the following commutative diagram:
\begin{align*}
\begin{xy}
( 0, 8) *+{\sttilt A} = "1",
(48, 8) *+{\sttilt \overline{A}} = "2",
( 0,-8) *+{\fLsbrick A} = "3",
(48,-8) *+{\fLsbrick \overline{A}} = "4",
\ar^{\overline{?}} "1"; "2"
\ar@{=} "3"; "4"
\ar_{\cong}^{M \mapsto \ind (M/{\rad_B M})} "1"; "3"
\ar_{\cong}^{M' \mapsto \ind (M'/{\rad_{B'} M'})} "2"; "4"
\end{xy} \quad
\begin{pmatrix}
B:=\End_A(M), \\
B':=\End_{\overline{A}}(\overline{M})
\end{pmatrix}.
\end{align*}
Moreover, this bijection preserves the mutations,
and hence the exchange quivers of $\sttilt A$ and $\sttilt \overline{A}$ are isomorphic.
\item[(3)]
The labels of the exchange quivers of $\sttilt A$ and $\sttilt \overline{A}$ coincide.
\end{itemize}
\end{Th}

Eisele--Janssens--Raedschelders obtained the above result in a geometric way,
but we show this result in terms of semibricks.
There is another approach to this problem in \cite[Section 5]{DIRRT}, 
which deals with lattice congruences on torsion classes.

The following lemmas are crucial.

\begin{Lem}\label{EJR_tors}
Let $\cS \in \sbrick \overline{A}$, then we have the following assertions.
\begin{itemize}
\item[(1)]
We have $\sT_A(\cS) \cap \mod \overline{A} = \sT_{\overline{A}}(\cS)$.
\item[(2)]
Take an integer $n>0$ such that $I^n=0$.
Then $\sT_A(\cS)=(\sT_{\overline{A}}(\cS))^{*n} \subset \mod A$ holds,
where $(\sT_{\overline{A}}(\cS))^{*n}$ consists of the $A$-modules $L$
having a sequence $0=L_0 \subset L_1 \subset \cdots \subset L_n =L$
with $L_i/L_{i-1} \in \sT_{\overline{A}}(\cS)$.
\item[(3)]
If $\cS \in \fLsbrick \overline{A}$, then $\cS \in \fLsbrick A$.
\item[(4)]
We have $\fLsbrick \overline{A} \subset \fLsbrick A$.
\end{itemize}
\end{Lem}

\begin{proof}
(1)
This follows as 
$\sT_A(\cS) \cap \mod \overline{A} = \Filt_A(\Fac \cS) \cap \mod \overline{A}
= \Filt_{\overline{A}}(\Fac \cS) = \sT_{\overline{A}}(\cS)$.

(2)
Clearly, we get that $\sT_A(\cS) \supset (\sT_{\overline{A}}(\cS))^{*n}$.
We show that $\sT_A(\cS) \subset (\sT_{\overline{A}}(\cS))^{*n}$.
Let $L \in \sT_A(\cS)$.
In the sequence $0=LI^n \subset \cdots \subset LI \subset L$,
we have $LI^t/LI^{t+1} \in \sT_A(\cS) \cap \mod \overline{A}$ for all $t \ge 0$.
From (1), we get $LI^t/LI^{t+1} \in \sT_{\overline{A}}(\cS)$.
Thus, $L$ belongs to $(\sT_{\overline{A}}(\cS))^{*n}$.
Now, the inclusion $\sT_A(\cS) \subset (\sT_{\overline{A}}(\cS))^{*n}$ is proved;
hence, $\sT_A(\cS) = (\sT_{\overline{A}}(\cS))^{*n}$.

(3)
By assumption, 
$\sT_{\overline{A}}(\cS)$ is functorially finite in $\mod \overline{A}$.
Because $\mod \overline{A}$ is functorially finite in $\mod A$,
the subcategory
$\sT_{\overline{A}}(\cS)$ is functorially finite in $\mod A$.
By (2) and \cite[Theorem 2.6]{SikS}, 
$\sT_A(\cS) = (\sT_{\overline{A}}(\cS))^{*n}$ 
is functorially finite in $\mod A$.

(4)
It immediately follows from (3).
\end{proof}

\begin{Lem}\label{EJR_basic}
Let $M \in \sttilt A$, and set $B:=\End_A(M)$ and $B':=\End_{\overline{A}}(\overline{M})$.
Then we have the following properties.
\begin{itemize}
\item[(1)]
The functor $\overline{?}$ induces a $K$-algebra epimorphism 
$\overline{?}_{M,M} \colon B \to B'$,
and it is restricted to a surjection $\rad B \to \rad B'$.
\item[(2)]
As $A$-modules, $M/{\rad_B M} \cong \overline{M}/{\rad_{B'} \overline{M}}$.
\item[(3)]
The $\overline{A}$-modules $\overline{M}$ belongs to $\sttilt \overline{A}$.
\item[(4)]
We have $\fLsbrick A \subset \fLsbrick \overline{A}$.
\end{itemize}
\end{Lem}

\begin{proof}
(1)
The map $\overline{?}_{M,M} \colon B \to B'$ is clearly a $K$-algebra homomorphism.
We prove its surjectivity.
Let $\alpha \in B'$.
Consider the exact sequence $0 \to MI \to M \xrightarrow{p} \overline{M} \to 0$.
Since $MI \in \Fac M$ and $M \in \sttilt A$ imply $\Ext_A^1(M,MI)=0$,
we get that $\Hom_A(M,M) \to \Hom_A(M,\overline{M})$ is surjective.
Thus, there exists $f \in \Hom_A(M,M)=B$ such that $pf=\alpha p$,
and we obtain that $\overline{f}=\alpha$.
Thus, we get a $K$-algebra epimorphism 
$\overline{?}_{M,M} \colon B \to B'$.

Next, let $f \in B$. 
For the remaining statement, 
it suffices to show that $\overline{f} \in \rad B'$ holds
if and only if $f \in \rad B$.
If $f \in \rad B$, then $f$ is nilpotent, and so is $\overline{f}$.
Thus, we have $\overline{f} \in \rad B'$.
On the other hand, if $f \notin \rad B$, then
there exists an indecomposable direct summand $M_1$ of $M$
such that the component $f_{11} \colon M_1 \to M_1$ of $f$ is isomorphic.
Clearly, $\overline{f_{11}} \colon \overline{M_1} \to \overline{M_1}$ is an isomorphism.
Here, $\overline{M_1} \ne 0$ holds, because $I \subset \rad A$.
Thus, we have $\overline{f_{11}} \notin \rad \End_{\overline{A}}(\overline{M_1})$,
and hence $\overline{f} \notin \rad \End_{\overline{A}}(\overline{M})=\rad B'$.
Thus $\overline{f} \in \rad B'$ holds
if and only if $f \in \rad B$.

(2)
By (1), the canonical epimorphism $M \to \overline{M}$ is restricted to an epimorphism
\begin{align*}
\rad_B M=\sum_{f \in \rad B} \Im f \to \sum_{g \in \rad B'} \Im g = 
\rad_{B'} \overline{M}
\end{align*}
of $A$-modules,
so we get $M/{\rad_B M}=\overline{M}/{\rad_{B'} \overline{M}}$ as $A$-modules.

(3)
\cite[Proposition 5.6 (a)]{DIRRT} implies that 
$\overline{M}$ is a $\tau$-rigid $\overline{A}$-module.
On the other hand, we obtain an isomorphism $B/{\rad B} \cong B'/{\rad B'}$
of semisimple $K$-algebras from (1).
By assumption, $M$ is basic, 
so $B/{\rad B}$ is a direct sum of $|M|$ division $K$-algebras.
Thus, $B'/{\rad B'}$ is also a direct sum of $|M|$ division $K$-algebras;
hence, $\overline{M}$ is basic and $|\overline{M}|=|M|$.
These observations implies that $\overline{M}$ is 
a basic support $\tau$-tilting $\overline{A}$-module.

(4)
Let $\cS \in \fLsbrick A$, then Theorem \ref{sttilt_fsbrick} implies that
there exists $M \in \sttilt A$ such that $\cS=\ind(M/{\rad_B M})$.
By (2), we obtain $\cS=\ind(\overline{M}/{\rad_{B'} \overline{M}})$, 
and by (3), we have $\overline{M} \in \sttilt \overline{A}$.
Applying Theorem \ref{sttilt_fsbrick} again,
$\cS=\ind(\overline{M}/{\rad_{B'} \overline{M}})$ belongs to 
$\fLsbrick \overline{A}$.
This implies the assertion.
\end{proof}

Now, we can finish the proof of the theorem.

\begin{proof}[Proof of Theorem \ref{EJR_th}]
(1) 
It follows from \ref{EJR_tors} (4) and \ref{EJR_basic} (4).

(2)
The map $\overline{?} \colon \sttilt A \to \sttilt \overline{A}$ 
is well-defined by Lemma \ref{EJR_basic} (3), 
and the diagram is commutative by Lemma \ref{EJR_basic} (2).
The vertical two maps in the diagram are 
bijective by Theorem \ref{sttilt_fsbrick},
so the map $\overline{?} \colon \sttilt A \to \sttilt \overline{A}$ is also bijective.
It is easy to see that this bijection preserves the mutations.
Thus, the exchange quivers of $\sttilt A$ and $\sttilt \overline{A}$ 
are isomorphic.

(3)
Let $M \to N$ be an arrow in the exchange quiver of $\sttilt A$ labeled with a brick $S$.
Lemma \ref{EJR_basic} (2) implies that
$S \in \ind (\overline{M}/{\rad_B \overline{M}})$. 
By the definition of labels,
$S$ is on the unique arrow $\overline{M} \to \overline{L}$
with $S \notin \Fac \overline{L}$ in the exchange quiver of $\sttilt \overline{A}$.
Thus, it suffices to show that $S \notin \Fac \overline{N}$.
By the definition of labels again, we have $S \notin \Fac N$.
Then $S \notin \Fac \overline{N}$ holds.
Thus, the arrow $\overline{M} \to \overline{N}$ is labeled with $S$.
\end{proof}

\section{Semibricks in Derived Categories}\label{Section_smc}

In this section, we consider 2-term simple-minded collections in 
$\sD^\rb(\mod A)$
to understand more properties of left finite or right finite semibricks.
We call a triangulated subcategory $\cC'$ of a triangulated category $\cC$ 
a \textit{thick subcategory} if $\cC'$ is closed under direct summands in $\cC$.

\subsection{Bijections II}

First, we give the definition of 2-term simple-minded collections in the derived category 
$\sD^\rb(\mod A)$.

\begin{Def}\label{smc_def}
A set $\cX$ of isoclasses of objects in $\sD^\rb(\mod A)$ is called 
a \textit{simple-minded collection} in $\sD^\rb(\mod A)$ if it satisfies the following conditions:
\begin{itemize}
\item for any $X \in \cX$, the endomorphism ring 
$\End_{\sD^\rb(\mod A)}(X)$ is a division $K$-algebra,
\item for any $X_1 \ne X_2 \in \cX$,
we have $\Hom_{\sD^\rb(\mod A)}(X_1,X_2)=0$,
\item for any $X_1, X_2 \in \cX$ and $n < 0$,
we have $\Hom_{\sD^\rb(\mod A)}(X_1,X_2[n])=0$,
\item the smallest thick subcategory of $\sD^\rb(\mod A)$ containing $\cX$ 
is $\sD^\rb(\mod A)$ itself.
\end{itemize}
A simple-minded collection $\cX$ in $\sD^\rb(\mod A)$ is said to be \textit{2-term}
if the $i$th cohomology $H^i(X)$ is $0$ for any $i \ne -1,0$ and any $X \in \cX$.
We write $\twosmc A$ for the set of 2-term simple-minded collections in $\sD^\rb(\mod A)$.
\end{Def}

\begin{Rem}\label{smc_rem}
We note the following basic properties holding for $\cX \in \twosmc A$.
\begin{itemize}
\item[(1)]
\cite[Corollary 5.5]{KY}
The cardinality $\# \cX$ is equal to $|A|$.
\item[(2)]
\cite[Remark 4.11]{BY}
Every $X \in \cX$ belongs to either $\mod A$ or $(\mod A)[1]$ up to isomorphisms in $\sD^\rb(\mod A)$.
\end{itemize}
\end{Rem}

The following is the main theorem of this section.
It means that every 2-term simple-minded collection is divided into
a left finite semibrick and a right finite semibrick,
and there are bijections between these three notions.

\begin{Th}\label{smc_fsbrick}
We have the following assertions.
\begin{itemize}
\item[(1)]
There are bijections 
\begin{align*}
? \cap \mod A \colon \twosmc A \to \fLsbrick A \quad \mathrm{and} \quad
?[-1] \cap \mod A \colon \twosmc A \to \fRsbrick A
\end{align*}
given by $\cX \mapsto \cX \cap \mod A$ and 
$\cX \mapsto \cX[-1] \cap \mod A$, respectively.
\item[(2)]
The diagram in Figure \ref{big_comm} below is commutative and all the maps are bijective.
In this diagram,
$\cT \in \ftors A$ corresponds to $\cF \in \ftorf A$ 
if and only if $(\cT,\cF)$ is a torsion pair in $\mod A$.
\item[(3)]
If $\cX \in \twosmc A$ corresponds to $\cS \in \fLsbrick A$ and $\cS' \in \fRsbrick A$,
then we have $\cX=\cS \cup \cS'[1]$ and a torsion pair $(\sT(\cS),\sF(\cS'))$ in $\mod A$.
\end{itemize}
\begin{figure}[h]
\begin{align*}
\begin{xy}
( 0, 36) *+{\stitilt A}   ="06",
(50, 36) *+{\ftorf A}     ="16",
(90, 36) *+{\fRsbrick A}  ="36",
( 0, 24) *+{\twocosilt A} ="04",
( 0, 12) *+{\twosilt A}   ="02",
(50, 18) *+{\inttstr A}   ="13",
(90, 18) *+{\twosmc A}    ="33",
( 0,  0) *+{\sttilt A}    ="00",
(50,  0) *+{\ftors A}     ="10",
(90,  0) *+{\fLsbrick A}  ="30",
\ar ^{\Sub}                            "06";"16"
\ar _{\sF}                             "36";"16"
\ar ^{H^{-1}}                          "04";"06"
\ar _{\textup{(heart)}[-1]\cap \mod A}"13";"16"
\ar _{?[-1] \cap \mod A}           "33";"36"
\ar ^{\nu}                             "02";"04"
\ar ^{I \mapsto ({^\perp}I[{<}0],{^\perp}I[{>}0])} "04";"13"
\ar _{P \mapsto (P[{<}0]^\perp,P[{>}0]^\perp)} "02";"13"
\ar ^{\textup{simples of heart}}       "13";"33"
\ar ^{\Fac}                            "00";"10"
\ar _{\sT}                             "30";"10"
\ar _{H^0}                             "02";"00"
\ar ^{\textup{(heart)}\cap \mod A}    "13";"10"
\ar ^{? \cap \mod A}               "33";"30"
\ar@{-} "00"; ( 0,-8)
\ar@{-}_{M \mapsto \ind(M/{\rad_B M})} ( 0,-8);(90,-8)
\ar (90,-8); "30"
\ar@{-} "06"; ( 0,44)
\ar@{-}^{M \mapsto \ind(\soc_B M)} ( 0,44);(90,44)
\ar (90,44); "36"
\end{xy}\\
(\text{$B:=\End_A(M)$ in each case})
\end{align*}
\caption{The commutative diagram}\label{big_comm}
\end{figure}
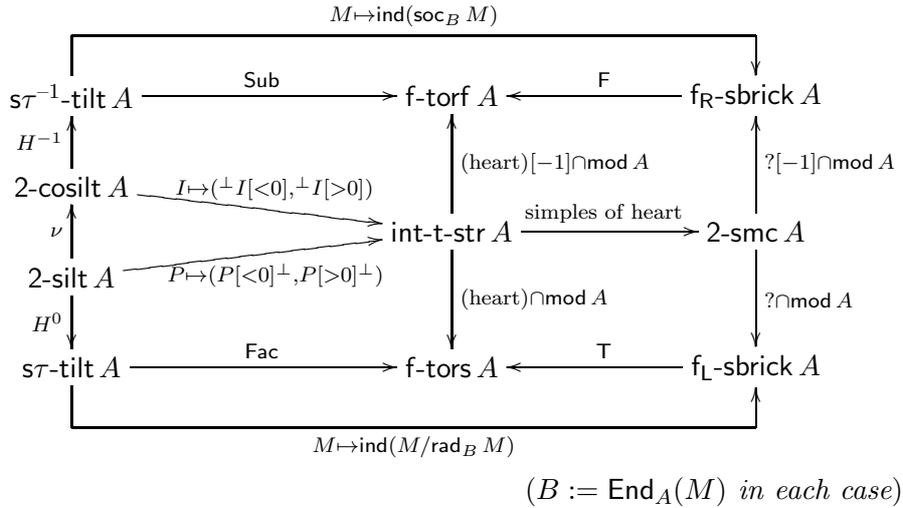
\end{Th}

It follows that the correspondence between $\sttilt A$ and $\stitilt A$
in Theorem \ref{smc_fsbrick} is the same one given before
Proposition \ref{sttilt_stitilt}.

To prove this theorem, we use the notions and fundamental properties
of 2-term silting objects in the homotopy category $\sK^\rb(\proj A)$ and 
intermediate t-structures in $\sD^\rb(\mod A)$.
We refer to \cite{BY,KY} here.

An object $P \in \sK^\rb(\proj A)$ is called a \textit{silting object} 
in $\sK^\rb(\proj A)$
if the following conditions are satisfied:
\begin{itemize}
\item for any $n>0$, we have $\Hom_{\sK^\rb(\proj A)}(P,P[n])=0$,
\item the smallest thick subcategory of $\sK^\rb(\proj A)$ containing $P$
is $\sK^\rb(\proj A)$ itself.
\end{itemize}
A silting object $P$ in $\sK^\rb(\proj A)$ 
is said to be \textit{2-term} if it is isomorphic to some complex
$(P^{-1} \to P^0)$ in $\sK^\rb(\proj A)$ 
with its components zero except for $-1$st and 0th ones.
We write $\twosilt A$ for the set of isoclasses of basic
2-term silting objects in $\sK^\rb(\proj A)$.

Dually, an object $I \in \sK^\rb(\inj A)$ is called a \textit{2-term cosilting object} 
in $\sK^\rb(\inj A)$ if the following conditions are satisfied:
\begin{itemize}
\item for any $n>0$, we have $\Hom_{\sK^\rb(\inj A)}(I,I[n])=0$,
\item the smallest thick subcategory of $\sK^\rb(\inj A)$ containing $I$
is $\sK^\rb(\inj A)$ itself,
\item $I$ is isomorphic to some complex $(I^{-1} \to I^0)$ in $\sK^\rb(\inj A)$.
\end{itemize}
We write $\twocosilt A$ for the set of isoclasses of basic ones.
The Nakayama functor $\nu \colon \sK^\rb(\proj A) \to \sK^\rb(\inj A)$
induces a bijection from $\twosilt A$ to $\twocosilt A$.
There is a formula called the Nakayama duality:
$D \Hom_{\sD^\rb(\mod A)}(P,X) \cong \Hom_{\sD^\rb(\mod A)}(X,\nu P)$
for $P \in \sK^\rb(\proj A)$ and $X \in \sD^\rb(\mod A)$.

We next recall the notion of intermediate t-structures.
The concept of t-structures was introduced 
by Be\u{\i}linson--Bernstein--Deligne \cite{BBD}.
First, let $\cD$ be a triangulated category and 
$(\cD^{\le 0}, \cD^{\ge 0})$ be a pair of additive full subcategories of $\cD$.
For simplicity, we define 
$\cD^{\le n}:=\cD^{\le 0} [-n]$ and $\cD^{\ge n}:=\cD^{\ge 0} [-n]$ for $n \in \Z$.
The pair $(\cD^{\le 0}, \cD^{\ge 0})$ is called a \textit{t-structure} in $\cD$
if it satisfies the following conditions:
\begin{itemize}
\item we have $\cD^{\le -1} \subset \cD^{\le 0}$ and $\cD^{\ge 1} \subset \cD^{\ge 0}$, 
and $\Hom_{\cD}(\cD^{\le 0},\cD^{\ge 1})=0$,
\item for every $X \in \cD$, there exists a triangle
$Y \to X \to Z \to Y[1]$ in $\cD$ with $Y \in \cD^{\le 0}$ and $Z \in \cD^{\ge 1}$. 
\end{itemize}
In this case, 
it is easy to see that each of $\cD^{\le 0}$ and $\cD^{\ge 0}$ determines the other.
For a t-structure $(\cD^{\le 0}, \cD^{\ge 0})$ in $\cD$, 
the intersection $\cD^{\le 0} \cap \cD^{\ge 0}$ is called 
the \textit{heart} of the t-structure.
It is well-known that the heart is an abelian category \cite[Th\'{e}or\`{e}me 1.3.6]{BBD}.

If $\cD=\sD^\rb(\mod A)$, there is a canonical t-structure 
$(\cD_{\mathrm{std}}^{\le 0},\cD_{\mathrm{std}}^{\ge 0})$
in $\sD^\rb(\mod A)$ defined with cohomologies, namely,
\begin{align*}
\cD_{\mathrm{std}}^{\le 0} &:= \{ X \in \sD^\rb(\mod A) \mid H^i(X)=0 \ (i>0) \}, \\
\cD_{\mathrm{std}}^{\ge 0} &:= \{ X \in \sD^\rb(\mod A) \mid H^i(X)=0 \ (i<0) \},
\end{align*}
and it is called the \textit{standard t-structure} in $\sD^\rb(\mod A)$.
We say that a t-structure $(\cD^{\le 0}, \cD^{\ge 0})$ in $\sD^\rb(\mod A)$
is \textit{intermediate} with respect to the standard t-structure 
(or simply \textit{intermediate}) if
$\cD_{\mathrm{std}}^{\le -1} \subset \cD^{\le 0} \subset \cD_{\mathrm{std}}^{\le 0}$,
or equivalently,
$\cD_{\mathrm{std}}^{\ge -1} \supset \cD^{\ge 0} \supset \cD_{\mathrm{std}}^{\ge 0}$
holds.
We can see that every intermediate t-structure 
$(\cD^{\le 0}, \cD^{\ge 0})$ in $\sD^\rb(\mod A)$ 
is \textit{bounded}, that is,
\begin{align*}
\bigcup_{n \in \Z} \cD^{\le n}=\sD^\rb(\mod A)=\bigcup_{n \in \Z} \cD^{\ge n}
\end{align*}
hold.
The next lemma follows from the definition of intermediate t-structures.

\begin{Lem}\label{heart_tors_pair}
Each intermediate t-structure $(\cD^{\le 0},\cD^{\ge 0})$ in $\sD^\rb(\mod A)$ 
with its heart $\cH$
gives a torsion pair $(\cH \cap \mod A, \cH[-1] \cap \mod A)$ in $\mod A$.
\end{Lem}

In this paper, we only consider intermediate t-structures in $\sD^\rb(\mod A)$ 
such that its heart is an abelian category with length.
The set of such intermediate t-structures is denoted by $\inttstr A$.

The following is an important result by Br\"{u}stle--Yang \cite{BY}
on these concepts, and 
it is a ``2-term'' 
restrictions of Koenig--Yang bijections \cite[Theorem 6.1]{KY}.

\begin{Prop}\label{silt_smc}\cite[Corollary 4.3]{BY}
We have the following bijections.
\begin{itemize}
\item[(1)] There is a bijection $\twosilt A \to \inttstr A$ given by
$P \mapsto (P[{<}0]^\perp,P[{>}0]^\perp)$, where 
\begin{align*}
P[{<}0]^\perp&= \{X \in \sD^\rb(\mod A) \mid \Hom_{\sD^\rb(\mod A)}(P[n],X)=0 \ (n<0) \}\\
&= \{X \in \sD^\rb(\mod A) \mid \Hom_{\sD^\rb(\mod A)}(X,\nu P[n])=0 \ (n<0) \}
= {^\perp}\nu P[{<}0], \\
P[{>}0]^\perp &= \{X \in \sD^\rb(\mod A) \mid \Hom_{\sD^\rb(\mod A)}(P[n],X)=0 \ (n>0) \}\\
&= \{X \in \sD^\rb(\mod A) \mid \Hom_{\sD^\rb(\mod A)}(X,\nu P[n])=0 \ (n>0) \}
= {^\perp}\nu P[{>}0].
\end{align*}
\item[(2)] There is a bijection $\inttstr A \to \twosmc A$ defined as follows:
each $(\cD^{\le 0}, \cD^{\ge 0}) \in \inttstr A$ is sent to the set of isoclasses of 
simple objects in the heart $\cH=\cD^{\le 0} \cap \cD^{\ge 0}$.
\end{itemize}
\end{Prop}

Note that the heart of the corresponding intermediate t-structure
$(P[{<}0]^\perp,P[{>}0]^\perp)$ for $P \in \twosilt A$ is
\begin{align*}
P[{\ne}0]^\perp &= \{X \in \sD^\rb(\mod A) 
\mid \Hom_{\sD^\rb(\mod A)}(P[n],X)=0 \ (n \ne 0) \}.
\end{align*}
It will play an important role in the proof of Theorem \ref{smc_fsbrick}.

Proposition \ref{silt_smc} gives bijections on notions in 
$\sK^\rb(\proj A)$ and $\sD^\rb(\mod A)$,
but we also would like to know their relationship with notions in the module category.
Here, we cite \cite{AIR,BY}.

\begin{Prop}\label{silt_sttilt}\cite[Theorem 3.2]{AIR}
The map $\twosilt A \ni P \mapsto H^0(P) \in \sttilt A$ is bijective.
The inverse is given as follows:
each $(M,P) \in \sttilt A$ with a minimal projective resolution 
$P^{-1} \to P^0 \to M \to 0$ of $M$
is sent to $(P^{-1} \oplus P \to P^0) \in \twosilt A$.
\end{Prop}

\begin{Prop}\label{silt_ftors}
\cite[Theorem 4.9]{BY}
We have a bijection 
$(\textup{heart}) \cap \mod A \colon \inttstr A \to \ftors A$ defined as
$(\cD^{\le 0},\cD^{\ge 0}) \mapsto \cH \cap \mod A$,
where $\cH=\cD^{\le 0} \cap \cD^{\ge 0}$ is the heart.
Moreover, this bijection joins
the following commutative diagram of bijections:
\begin{align*}
\begin{xy}
( 0,  8) *+{\twosilt A}   ="01",
(55,  8) *+{\inttstr A}   ="11",
( 0, -8) *+{\sttilt A}    ="00",
(55, -8) *+{\ftors A}     ="10",
\ar^{P \mapsto (P[{<}0]^\perp,P[{>}0]^\perp)} "01";"11"
\ar^{\Fac}                                    "00";"10"
\ar^{H^0}                                     "01";"00"
\ar^{(\textup{heart}) \cap \mod A}            "11";"10"
\end{xy}.
\end{align*}
\end{Prop}

In Proposition \ref{silt_ftors}, we can easily see 
$\cH \cap \mod A=\cD^{\le 0} \cap \mod A$ 
from the definition of intermediate t-structures,
and the original statement in \cite{BY} uses the latter.

Now, we start the proof of Theorem \ref{smc_fsbrick}.

\begin{Lem}\label{heart_lem}
Let $P \in \twosilt A$,
$\cH:=P[{\ne}0]$ be the heart of the corresponding t-structure,
$\cX$ be the set of isoclasses of simple objects in $\cH$,
and $\cS:=\cX \cap \mod A$ as in Theorem \ref{smc_fsbrick}.
Then we have $\cH \cap \mod A = \sT(\cS)$.
Moreover, $\cS$ belongs to $\fLsbrick A$.
\end{Lem}

\begin{proof}
%
We first prove that $\sT(\cS) \subset \cH \cap \mod A$.
Because $\cH \cap \mod A$ is a torsion class in $\mod A$ 
by Lemma \ref{heart_tors_pair},
it is sufficient to see that $\cS \subset \cH \cap \mod A$.
This is clear.

Next, we show the converse $\cH \cap \mod A \subset \sT(\cS)$.
Since $H^0(\cH)=\cH \cap \mod A$ follows from \cite[Remark 4.11]{BY},
it is enough to show $H^0(\cH) \subset \sT(\cS)$,
that is, $H^0(X) \in \sT(\cS)$ for all $X \in \cH$.
We use induction on the length $l$ of $X$ in the abelian category $\cH$ with length.

If $l=0$, then $X \cong 0$, so $H^0(X)=0 \in \sT(\cS)$.

If $l \ge 1$, there exists a short exact sequence $0 \to Y \to X \to Z \to 0$ in $\cH$
with $Y$ simple and $Z$ of length $l-1$.
By Remark \ref{smc_rem} (2), we may assume that $Y$ belongs to $\mod A$ or $(\mod A)[1]$.
By the induction hypothesis, we can see $H^0(Z) \in \sT(\cS)$.
The short exact sequence is lifted to a triangle $Y \to X \to Z \to Y[1]$.

If $Y \in \mod A$, we have $Y \in \cS$ and
an exact sequence $Y \to H^0(X) \to H^0(Z) \to 0$
in $\mod A$.
Set $M:=\Ker(H^0(X) \to H^0(Z))$, then we get $M \in \Fac Y \subset \sT(\cS)$.
Moreover, there is a short exact sequence $0 \to M \to H^0(X) \to H^0(Z) \to 0$
in $\mod A$ with $H^0(Z) \in \sT(\cS)$.
Thus, the module $H^0(X)$ is also in $\sT(\cS)$.

If $Y \in (\mod A)[1]$, we have an exact sequence $0 \to H^0(X) \to H^0(Z) \to 0$
in $\mod A$.
Thus, we have $H^0(X) \cong H^0(Z)$ and $H^0(X) \in \sT(\cS)$.

The induction is complete, and we have $\cH \cap \mod A = H^0(\cH) = \sT(\cS)$.

Obviously, $\cS$ is a semibrick,
and $\sT(\cS)=\cH \cap \mod A$ is functorially finite in $\mod A$ 
by Proposition \ref{silt_ftors}.
These mean $\cS \in \fLsbrick A$.
\end{proof}

\begin{proof}[Proof of Theorem \ref{smc_fsbrick}]
(1)(2)
From Proposition \ref{silt_ftors} and Lemma \ref{heart_lem}, 
it follows that the map $? \cap \mod A \colon \twosmc A \to \fLsbrick A$ is well-defined,
and the following diagram is commutative:
\begin{align*}
\begin{xy}
( 0,  8) *+{\twosilt A}   ="01",
(55,  8) *+{\inttstr A}   ="11",
(99,  8) *+{\twosmc A}    ="21",
( 0, -8) *+{\sttilt A}    ="00",
(55, -8) *+{\ftors A}     ="10",
(99, -8) *+{\fLsbrick A}  ="20",
\ar^{P \mapsto (P[{<}0]^\perp,P[{>}0]^\perp)} "01";"11"
\ar^{\textup{simples of heart}} "11";"21"
\ar^{\Fac}                                    "00";"10"
\ar_{\sT}                                     "20";"10"
\ar^{H^0}                                     "01";"00"
\ar^{(\textup{heart}) \cap \mod A}            "11";"10"
\ar^{? \cap \mod A}                           "21";"20"
\end{xy}.
\end{align*}
Because the other maps in this diagram are bijective 
by Propositions \ref{comm_3}, \ref{silt_smc} and \ref{silt_ftors}, 
so the map $? \cap \mod A \colon \twosmc A \to \fLsbrick A$ is bijective.

We also have the dual commutative diagram 
by considering the opposite algebra $A^\mathrm{op}$ and appropriate shifts of complexes,
and thus, the map $?[-1] \cap \mod A \colon \twosmc A \to \fRsbrick A$ is bijective.
The part (1) is proved.

The remaining part of the commutative diagram in (2) is obtained 
from the Nakayama duality and Theorem \ref{sttilt_fsbrick}. 

Moreover, 
Proposition \ref{silt_ftors} and Lemma \ref{heart_tors_pair} imply that
$\cT \in \ftors A$ corresponds to $\cF \in \ftorf A$
if and only if $(\cT,\cF)$ is a torsion pair in $\mod A$.

(3)
It is immediately deduced from (2) and Remark \ref{smc_rem}.
\end{proof}

As in Lemma \ref{heart_lem}, 
the map $(\text{heart}) \cap \mod A \colon \inttstr A \to \ftors A$
is equal to $H^0(\text{heart})$,
and dually,
the map $(\text{heart})[-1] \cap \mod A \colon \inttstr A \to \ftorf A$
is equal to $H^{-1}(\text{heart})$.
Therefore, in the commutative diagram in Theorem \ref{smc_fsbrick} (2), 
we can regard the bottom row as the 0th cohomologies of 
the middle(-lower) row,
and the top row as the $-1$st cohomologies of the middle(-upper) row.

\subsection{Labeling the exchange quiver with bricks II}\label{der_label}

In this section, we label the exchange quiver of $\twosmc A$ with bricks
in a similar way to Subsection \ref{label_mod}.
First, we recall the fundamental properties of mutations in $\twosmc A$ from \cite{BY}.

\begin{Prop}\label{smc_mut}\cite[Subsection 3.7]{BY}
Let $\cX \in \twosmc A$ and $\cS:=\cX \cap \mod A$.
\begin{itemize}
\item[(1)]
Assume $X \in \cX$.
Then there exists a left mutation of $\cX$ at $X$ in $\twosmc A$ 
if and only if $X \in \cS$.
\item[(2)]
For $S_0 \in \cS$, 
a left mutation of $\cX$ at $S_0$ uniquely exists, and it is given as follows:
\begin{itemize}
\item define a full subcategory $\mathcal{E}:=\Filt S_0$ of $\mod A$,
\item set $Y_{S_0}:=S_0[1]$,
\item for any $S \in \cS \setminus \{S_0\}$, 
take a left minimal $\mathcal{E}$-approximation
$f_S \colon S[-1] \to E_S$ in $\sD^\rb(\mod A)$ and 
define $Y_S$ as the mapping cone of $f_S$
(we can see $Y_S \in \mod A$ and 
have an exact sequence $0 \to E_S \to Y_S \to S \to 0$ in $\mod A$),
\item for any $S'[1] \in \cX \setminus \cS$, 
take a left minimal $\mathcal{E}$-approximation
$f_{S'[1]} \colon S' \to E_{S'[1]}$ in $\mod A$.
If $f_{S'[1]}$ is injective, then set $Y_{S'[1]}:=\Coker f_{S'[1]}$,
and otherwise (in this case $f_{S'[1]}$ is surjective), 
set $Y_{S'[1]}:=(\Ker f_{S'[1]})[1]$,
\item now $\cY:=\{Y_X \mid X \in \cX \}$ is the left mutation of $\cX$ at $S_0$.
\end{itemize}
\end{itemize}
\end{Prop}

We define labels of the exchange quiver of $\twosmc A$ with bricks as follows.

\begin{Def}\label{label_smc_def}
In the setting of Proposition \ref{smc_mut}, 
we label the arrow $\cX \to \cY$ in the exchange quiver of $\twosmc A$
with a brick $S_0$.
\end{Def}

The right mutation in $\twosmc A$ at each object in $\cS':=\cX[-1] \cap \mod A$
is similarly defined.
By definition,
the labels on the arrows from $\cX$ are the elements of $\cS$, and
the labels on the arrows to $\cX$ are the elements of $\cS'$.

We give an example, compare it with Example \ref{sttilt_ex}.

\begin{Ex}\label{smc_ex}
Let $A$ be the path algebra of the quiver $1 \to 2 \to 3$.
Figure \ref{A3_smc} below is 
the exchange quiver of $\twosmc A$ labeled with bricks.
\begin{figure}[h]
\begin{align*}
\begin{xy}
(-45, 27)*+{
\begin{smallmatrix} 3 \end{smallmatrix},
\begin{smallmatrix} 2 \end{smallmatrix},
\begin{smallmatrix} 1 \end{smallmatrix}
}= "01",
(-45,  0)*+{
\begin{smallmatrix} 2 \\ 3 \end{smallmatrix},
\begin{smallmatrix} 1 \end{smallmatrix},
\begin{smallmatrix} 3 \end{smallmatrix}[1]
}= "02",
(  0, 27)*+{
\begin{smallmatrix} 3 \end{smallmatrix},
\begin{smallmatrix} 1 \\ 2 \end{smallmatrix},
\begin{smallmatrix} 2 \end{smallmatrix}[1]
}= "03",
(-69,  0)*+{
\begin{smallmatrix} 3 \end{smallmatrix},
\begin{smallmatrix} 2 \end{smallmatrix},
\begin{smallmatrix} 1 \end{smallmatrix}[1]
}= "04",
(-18,  0)*+{
\begin{smallmatrix} 1 \\ 2 \\ 3 \end{smallmatrix},
\begin{smallmatrix} 2 \end{smallmatrix},
\begin{smallmatrix} 2 \\ 3 \end{smallmatrix}[1]
}= "05",
(-45,-27)*+{
\begin{smallmatrix} 2 \\ 3 \end{smallmatrix},
\begin{smallmatrix} 3 \end{smallmatrix}[1],
\begin{smallmatrix} 1 \end{smallmatrix}[1]
}= "06",
(  0, 15)*+{
\begin{smallmatrix} 1 \\ 2 \\ 3 \end{smallmatrix},
\begin{smallmatrix} 3 \end{smallmatrix}[1],
\begin{smallmatrix} 2 \end{smallmatrix}[1]
}= "07",
( 45, 27)*+{
\begin{smallmatrix} 3 \end{smallmatrix},
\begin{smallmatrix} 1 \end{smallmatrix},
\begin{smallmatrix} 1 \\ 2 \end{smallmatrix}[1]
}= "08",
(  0,-15)*+{
\begin{smallmatrix} 2 \end{smallmatrix},
\begin{smallmatrix} 1 \end{smallmatrix},
\begin{smallmatrix} 1 \\ 2 \\ 3 \end{smallmatrix}[1]
}= "09",
( 18,  0)*+{
\begin{smallmatrix} 1 \\ 2 \end{smallmatrix},
\begin{smallmatrix} 1 \\ 2 \\ 3 \end{smallmatrix}[1],
\begin{smallmatrix} 2 \end{smallmatrix}[1]
}= "10",
( 69,  0)*+{
\begin{smallmatrix} 3 \end{smallmatrix},
\begin{smallmatrix} 2 \end{smallmatrix}[1],
\begin{smallmatrix} 1 \end{smallmatrix}[1]
}= "11",
(  0,-27)*+{
\begin{smallmatrix} 2 \end{smallmatrix},
\begin{smallmatrix} 2 \\ 3 \end{smallmatrix}[1],
\begin{smallmatrix} 1 \end{smallmatrix}[1]
}= "12",
( 45,  0)*+{
\begin{smallmatrix} 1 \end{smallmatrix},
\begin{smallmatrix} 3 \end{smallmatrix}[1],
\begin{smallmatrix} 1 \\ 2 \end{smallmatrix}[1]
}= "13",
( 45,-27)*+{
\begin{smallmatrix} 3 \end{smallmatrix}[1],
\begin{smallmatrix} 2 \end{smallmatrix}[1],
\begin{smallmatrix} 1 \end{smallmatrix}[1]
}= "14",
\ar_{\begin{smallmatrix} 3 \end{smallmatrix}} "01";"02" 
\ar^{\begin{smallmatrix} 2 \end{smallmatrix}} "01";"03" 
\ar_{\begin{smallmatrix} 1 \end{smallmatrix}} "01";"04"
\ar^{\begin{smallmatrix} 2 \\ 3 \end{smallmatrix}} "02";"05" 
\ar_{\begin{smallmatrix} 1 \end{smallmatrix}} "02";"06"
\ar_{\begin{smallmatrix} 3 \end{smallmatrix}} "03";"07" 
\ar^{\begin{smallmatrix} 1 \\ 2 \end{smallmatrix}} "03";"08"
\ar_{\begin{smallmatrix} 3 \end{smallmatrix}} "04";"06" 
\ar@{-} "04";(-69,-36) 
\ar@{-}^{\begin{smallmatrix} 2 \end{smallmatrix}} (-69,-36);(69,-36) 
\ar(69,-36);"11"
\ar^(.45){\begin{smallmatrix} 2 \end{smallmatrix}} "05";"07" 
\ar_(.45){\begin{smallmatrix} 1 \\ 2 \\ 3 \end{smallmatrix}} "05";"09"
\ar^{\begin{smallmatrix} 2 \\ 3 \end{smallmatrix}} "06";"12"
\ar^(.55){\begin{smallmatrix} 1 \\ 2 \\ 3 \end{smallmatrix}} "07";"10"
\ar^{\begin{smallmatrix} 1 \end{smallmatrix}} "08";"11" 
\ar_{\begin{smallmatrix} 3 \end{smallmatrix}} "08";"13"
\ar_(.55){\begin{smallmatrix} 2 \end{smallmatrix}} "09";"10" 
\ar_{\begin{smallmatrix} 1 \end{smallmatrix}} "09";"12"
\ar^{\begin{smallmatrix} 1 \\ 2 \end{smallmatrix}} "10";"13"
\ar^{\begin{smallmatrix} 3 \end{smallmatrix}} "11";"14"
\ar^{\begin{smallmatrix} 2 \end{smallmatrix}} "12";"14"
\ar_{\begin{smallmatrix} 1 \end{smallmatrix}} "13";"14"
\end{xy}
\end{align*}
\caption{The exchange quiver of $\twosmc A$}\label{A3_smc}
\end{figure}
\end{Ex}

The three exchange quivers of $\sttilt A$, $\stitilt A$, and $\twosmc A$ 
in Examples \ref{sttilt_ex} and \ref{smc_ex} are naturally isomorphic
by the maps in Theorem \ref{smc_fsbrick}.
This holds for an arbitrary finite dimensional $K$-algebra 
\cite[Corollary 4.3, Theorem 4.9]{BY}.
Now, we show that the maps in Theorem \ref{smc_fsbrick} also preserve 
the labels of the exchange quivers.

\begin{Th}\label{brick_label_same}
The maps in Theorem \ref{smc_fsbrick} preserve the labels
of the exchange quivers of $\sttilt A$, $\stitilt A$, and $\twosmc A$.
\end{Th}

\begin{proof}
Let $\cX \to \cY$ be an arrow in the exchange quiver of $\twosmc A$
labeled with a brick $S$.
We have $S \in \cX \cap \mod A$ and $S \in \cY[-1] \cap \mod A$
by definition.
By Theorem \ref{smc_fsbrick}, 
we can take $M \in \sttilt A$ corresponding to $\cX \in \twosmc A$
and $N' \in \stitilt A$ corresponding to $\cY \in \twosmc A$,
and then $\sT(\cX \cap \mod A)=\Fac M$ and $\sF(\cY[-1] \cap \mod A)=\Sub N'$ hold.
Therefore, $S$ is a brick belonging to $\Fac M \cap \Sub N'$.
Now, Proposition \ref{brick_label_prop} implies the assertion.
\end{proof}

\subsection{Realizing wide subcategories as module categories II}\label{wide_der}

We have investigated each left finite subcategory $\cW \in \fLwide A$ 
in terms of the endomorphism algebra $B:=\End_A(M)$ of 
the corresponding support $\tau$-tilting module $M \in \sttilt A$
in Subsection \ref{wide_mod}.
Now, instead of $B$, 
we use the endomorphism algebra $C:=\End_{\sD^\rb(\mod A)}(P)$
of the corresponding 2-term silting objects $P \in \twosilt A$
to study $\cW$.

From now on, we fix $P \in \twosilt A$ 
and set $C:=\End_{\sK^\rb(\proj A)}(P)$.
The $K$-algebra $C$ is isomorphic to $C':=\End_{\sK^\rb(\inj A)}(\nu P)$
by the Nakayama functor $\nu \colon \sK^\rb(\proj A) \to \sK^\rb(\inj A)$.
As in \cite[Proposition 4.8]{IY}, the heart $\cH=P[{\ne}0]^\perp$ of 
the corresponding t-structure $(P[{<}0]^\perp,P[{>}0]^\perp)$ in $\sD^\rb(\mod A)$
is equivalent to 
$\mod C$ by the functor
\begin{align*}
\Hom_{\sD^\rb(\mod A)}(P,?) \colon \cH \xrightarrow{\cong} \mod C.
\end{align*}
By the Nakayama duality, we also have an equivalence 
$D\Hom_{\sD^\rb(\mod A)}(?,\nu P) \colon \cH \to \mod C'$.
We also define 
\begin{align*}
M'&:=H^{-1}(\nu P) \in \stitilt A, & M&:=H^0(P) \in \sttilt A, \\
B'&:=\End_A(M'), & B&:=\End_A(M), \\
\cS'&:=\ind(\soc_{B'} M') \in \fRsbrick A, & \cS&:=\ind(M/{\rad_B M}) \in \fLsbrick A,
\end{align*}
and then Theorem \ref{smc_fsbrick} 
implies that $\cX:=\cS \cup \cS'[1] \in \twosmc A$.

Let $P_1,P_2,\ldots,P_n$ be the indecomposable direct summands of $P$, where
$n=|A|$ by \cite[Corollary 5.1]{KY}.
For $i=1,2,\ldots,n$, we set 
\begin{align*}
M'_i&:=H^{-1}(\nu P_i), & M_i&:=H^0(P_i), \\
N'_i&:=\bigcap_{f \in \rad_A(M'_i,M')} \Ker f, & 
N_i&:=M_i/\sum_{f \in \rad_A(M,M_i)} \Im f.
\end{align*}
We obtain that $N_i$ is a brick or zero from Lemma \ref{well-def} (2).
The module $N'_i$ is also a brick or zero.
Clearly, $\cS'=\{ N'_i \mid N'_i \ne 0\}$ and 
$\cS=\{ N_i \mid N_i \ne 0\}$ hold.

\begin{Lem}\label{weak_strong}
In the setting above, the following assertions hold for $i=1,2,\ldots,n$.
\begin{itemize}
\item[(1)] The module $M_i$ is projective if and only if $M'_i=0$.
\item[(2)] We have $N_i \ne 0$ if and only if $N'_i = 0$.
\item[(3)] We have $M_i \ne 0$ if and only if $M'_i$ is not injective.
\item[(4)] If $M_i$ is projective then $N_i \ne 0$, and if $N_i \ne 0$ then $M_i \ne 0$.
\end{itemize}
\end{Lem}

\begin{proof}
The parts (1), (3) and (4) are obvious.

(2)
Because $\#\cS+\#\cS'=n$ by Remark \ref{smc_rem},
it is sufficient to show that the condition $N_i \ne 0$ implies $N'_i = 0$.
We assume $N_i \ne 0$ and $N'_i \ne 0$ and deduce a contradiction.

Because $N_i$ and $N'_i[1]$ are in 
$\cX=\cS \cup \cS'[1] \in \twosmc A$,
they are simple in $\cH$ by Theorem \ref{smc_fsbrick}.
On the other hand,
$N_i \in \Fac M_i$ implies $\Hom_{\sD^\rb(\mod A)}(P_i,N_i) \ne 0$, and
$N'_i \in \Sub M'_i$ implies $\Hom_{\sD^\rb(\mod A)}(N'_i[1],\nu P'_i) \ne 0$.
By the Nakayama duality,
we have $\Hom_{\sD^\rb(\mod A)}(P_i,N'_i[1]) \ne 0$.

Considering the equivalence $\Hom_{\sD^\rb(\mod A)}(P,?) \colon \cH \to \mod C$,
the two simple objects $N_i$ and $N'_i[1]$ in $\cH$ must be isomorphic. 
It is a contradiction.
\end{proof}

We have a one-to-one correspondence
between the indecomposable direct summands of $P$ and the elements of $\cX$
as follows.
See also \cite[Lemma 5.3]{KY}.

\begin{Lem}\label{P_X_dual}
For $i \in \{1,2,\ldots,n\}$, 
let $X_i:=N_i$ if $N_i \ne 0$ and $X_i:=N'_i[1]$ if $N'_i \ne 0$.
Then we have $\cX=\{X_1,X_2,\ldots,X_n\}$.
Moreover, let $j \in \{1,2,\ldots,n\}$ and
define a division ring $D_j$ as $\End_{\sD^\rb(\mod A)}(X_j)$.
Then,
\begin{align*}
\Hom_{\sD^\rb(\mod A)}(P_i,X_j) \cong \begin{cases}
D_j & (i=j) \\
0 & (i \ne j)
\end{cases}
\end{align*}
as left $D_j$-modules.
\end{Lem}

\begin{proof}
Because $\cX=\cS \cup \cS'[1]$, we have $\cX=\{X_1,\ldots,X_n\}$. 

Next, let $j \in \{1,2,\ldots,n\}$.
Then, $X_j$ is a simple object in $\cH$.
Thus, by the equivalence $\Hom_{\sD^\rb(\mod A)}(P,?) \colon \cH \to \mod C$,
the $C$-module $\Hom_{\sD^\rb(\mod A)}(P,X_j)$ is simple.
This module is a simple module over its endomorphism algebra
$\End_C (\Hom_{\sD^\rb(\mod A)}(P,X_j))$.
By the functoriality of $\Hom_{\sD^\rb(\mod A)}(P,?) \colon \cH \to \mod C$,
the module $\Hom_{\sD^\rb(\mod A)}(P,X_j)$ is a simple $D_j$-module.
Thus, it is isomorphic to $D_j$ as left $D_j$-modules.

Now, $D_j \cong \Hom_{\sD^\rb(\mod A)}(P,X_j) \cong 
\bigoplus_{i=1}^n \Hom_{\sD^\rb(\mod A)}(P_i,X_j)$ hold
as left $D_j$-modules.
Therefore, there uniquely exists $i$ such that $\Hom_{\sD^\rb(\mod A)}(P_i,X_j) \ne 0$.
As in the proof of Lemma \ref{weak_strong},
we have $\Hom_{\sD^\rb(\mod A)}(P_j,X_j) \ne 0$.
Therefore, 
$\Hom_{\sD^\rb(\mod A)}(P_j,X_j)$ is 
isomorphic to $D_j$ as left $D_j$-modules,
and $\Hom_{\sD^\rb(\mod A)}(P_i,X_j)$ must be zero if $i \ne j$.
\end{proof}

By Lemma \ref{weak_strong}, 
we may assume that $P_1,P_2,\ldots,P_n$ and $k \le l \le m$ satisfy the following:
\begin{itemize}
\item $M_i$ is projective if and only if $i=1,\ldots,k$,
\item $N_i \ne 0$ holds if and only if $i=1,\ldots,l$,
\item $M_i \ne 0$ holds if and only if $i=1,\ldots,m$. 
\end{itemize}
In this case, we have 
\begin{align*}
\cS=\{N_1,\ldots,N_l\}, \quad \cS'=\{N'_{l+1},\ldots,N'_n\}, \quad
\cX=\{ N_1,\ldots,N_l,N'_{l+1}[1],\ldots,N'_n[1]\}.
\end{align*}
The corresponding left finite wide subcategory is $\cW=\Filt \cS$,
and the corresponding right finite wide subcategory is $\cW'=\Filt \cS'$.

We can see the two categories $\cW, \cW'[1]$ as
Serre subcategories of $\cH$ by the next theorem, 
which is the main result of this subsection.

We define $f_i \in C$ as the idempotent endomorphism
$P \to P_i \to P$ in $\sK^\rb(\proj A)$
and set $f:=f_{l+1}+\cdots+f_n$ for $P \in \twosilt A$.
We also define $f'_i \in C'$ as the idempotent endomorphism
$\nu P \to \nu P_i \to \nu P$ in $\sK^\rb(\inj A)$
and set $f':=f'_1+\cdots+f'_l$ for $\nu P \in \twocosilt A$.

\begin{Th}\label{wide_C}
In the setting above, we have the following equivalences.
\begin{itemize}
\item[(1)]
The equivalence 
$\Hom_{\sD^\rb(\mod A)}(P,?) \colon \cH \to \mod C$
is restricted to an equivalence 
\begin{align*}
\cW \cong \mod C/\langle f \rangle.
\end{align*}
\item[(2)]
The equivalence 
$D\Hom_{\sD^\rb(\mod A)}(?,\nu P) \colon \cH \to \mod C'$
is restricted to an equivalence
\begin{align*}
\cW'[1] \cong \mod C'/\langle f' \rangle.
\end{align*}
\end{itemize}
\end{Th}

\begin{proof}
We only prove (1).
The part (2) is shown similarly.
Let $X \in \cH$.

A simple object $X_i \in \cH$ does not appear in the composition factors of $X$ in $\cH$
if and only if $\Hom_{\sD^\rb(\mod A)}(P_i,X)=0$, which is deduced from
an exact functor $\Hom_{\sD^\rb(\mod A)}(P_i,?) \colon \cH \to \mod K$
and Lemma \ref{P_X_dual}.
This is also equivalent to $\Hom_{\sD^\rb(\mod A)}(P,X) \circ f_i=0$.
 
The condition $X \in \cW$ exactly means that
none of $X_{l+1},\ldots,X_n$ appear in the composition factors of $X$ in $\cH$.
By the observation above, this holds if and only if $\Hom_{\sD^\rb(\mod A)}(P,X) \circ f=0$.
Therefore, the equivalence 
$\Hom_{\sD^\rb(\mod A)}(P,?) \colon \cH \to \mod C$ is restricted to
an equivalence $\cW \cong \mod C/\langle f \rangle$.
\end{proof}

Now, we relate Theorem \ref{wide_C} to Theorem \ref{wide_B}.
The idempotent $e \in B$ defined in Theorem \ref{wide_B} is
$e=\sum_{i=l+1}^m e_i$ in the current setting,
where each $e_i \in B$ is the idempotent endomorphism $M_i \to M \to M_i$.
There is an epimorphism $\phi \colon C \to B$ of $K$-algebras 
defined as $g \mapsto H^0(g)$,
which induces an epimorphism $\phi \colon C/\langle f \rangle \to B/\langle e \rangle$. 
We dually define an idempotent $e' \in B'$
as $e':=\sum_{i=k+1}^l e'_i$,
with each $e'_i \in B$ the idempotent endomorphism $M'_i \to M' \to M'_i$.
Then we have a $K$-algebra epimorphism $\phi' \colon C' \to B'$
defined as $g \mapsto H^{-1}(g)$,
which induces an epimorphism 
$\phi' \colon C'/\langle f' \rangle \to B'/\langle e' \rangle$. 

\begin{Th}\label{wide_C_B}
The following diagrams are commutative up to isomorphisms of functors:
\begin{align*}
&\begin{xy}
(  0,  8)*+{\cW}="01",
(  0, -8)*+{\cW}="00",
( 48,  8)*+{\mod B/\langle e \rangle}="11",
( 48, -8)*+{\mod C/\langle f \rangle}="10",
( 80,  8)*+{\mod B}="21",
( 80, -8)*+{\mod C}="20",
\ar@{=} "01";"00"
\ar_{\cong}^{\phi} "11";"10"
\ar^{\phi} "21";"20"
\ar_{\cong}^{\Hom_A(M,?)} "01";"11"
\ar_{\cong}^{\Hom_{\sD^\rb(\mod A)}(P,?)} "00";"10"
\ar^{\textup{nat}} "11";"21"
\ar^{\textup{nat}} "10";"20"
\end{xy}, \\
&\begin{xy}
(  0,  8)*+{\cW'}="01",
(  0, -8)*+{\cW'[1]}="00",
( 48,  8)*+{\mod B'/\langle e' \rangle}="11",
( 48, -8)*+{\mod C'/\langle f' \rangle}="10",
( 80,  8)*+{\mod B'}="21",
( 80, -8)*+{\mod C'}="20",
\ar_{\cong}^{[1]} "01";"00"
\ar_{\cong}^{\phi'} "11";"10"
\ar^{\phi'} "21";"20"
\ar_{\cong}^{\Hom_A(?,M')} "01";"11"
\ar_{\cong}^{\Hom_{\sD^\rb(\mod A)}(?,\nu P)} "00";"10"
\ar^{\textup{nat}} "11";"21"
\ar^{\textup{nat}} "10";"20"
\end{xy}.
\end{align*}
In particular, $\phi \colon C/\langle f \rangle \to B/\langle e \rangle$
and $\phi' \colon C'/\langle f' \rangle \to B'/\langle e' \rangle$ are isomorphisms.
\end{Th}

\begin{proof}
We show the commutativity of the first diagram. The second one is similarly proved.

It is enough to show that
the functor $\Hom_{\sD^\rb(\mod A)}(P,?) \colon \cW \to \mod C/\langle f \rangle$ is 
isomorphic to the following functor:
each $X \in \cW$ is sent to $\Hom_A(M,X)$ regarded as an object 
in $\mod C$, 
where $C$ acts on $\Hom_A(M,X)$ by $g \cdot h=g \circ H^0(h)$
for $g \in \Hom_A(M,X)$ and $h \in C$.

We consider a natural morphism $\pi \colon P \to M$ in $\sD^\rb(\mod A)$.
By straightforward calculation, a map 
$\Hom_A(\pi,X) \colon \Hom_A(M,X) \to \Hom_{\sD^\rb(\mod A)}(P,X)$ for $X \in \cW$
is an isomorphism of $K$-vector spaces functorial in $X$.
If it is a functorial isomorphism of $C$-modules
with the $C$-action on $\Hom_A(M,X)$ defined as above,
then the diagram is commutative.
This is easily deduced from the formula $g \circ H^0(h) \circ \pi=g \circ \pi \circ h$
for $g \in \Hom_A(M,X)$ and $h \in C$.

Now, we have proved that $\phi$ induces an equivalence
between $\mod B/\langle e \rangle$ and $\mod C/\langle f \rangle$.
Because $C/\langle f \rangle$ and $B/\langle e \rangle$ are basic algebras, 
$\phi \colon C/\langle f \rangle \to B/\langle e \rangle$ is an isomorphism.
\end{proof}

\subsection{Grothendieck groups and semibricks}\label{subsec_Gro}

In this subsection, 
we observe semibricks and 2-term simple-minded collections 
from the point of view of Grothendieck groups.

We first briefly recall the definition of Grothendieck groups of triangulated categories.
Let $\cD$ be an essentially small triangulated category.
Then the Grothendieck group $K_0(\cD)$ of $\cD$ is the abelian group
generated by all isomorphic classes $[X]$ in $\cD$ and bounded by the relation
$[X]-[Y]+[Z]=0$ for each triangle $X \to Y \to Z \to X[1]$.
In this paper, we consider the case $\cD=\sD^\rb(\mod A)$ and the case 
$\cD=\sK^\rb(\proj A)$.
It is clear that $[X[1]]=-[X]$.

For simplicity, we assume that $A$ is a basic finite-dimensional $K$-algebra.
Let $n:=|A|$ and 
$e_1,e_2,\ldots,e_n$ be the primitive idempotents in $A$.

It is well-known that the Grothendieck group $K_0(\sD^\rb(\mod A))$
has the family $([e_i (A/{\rad A})])_{i=1}^n$ of 
isoclasses of simple modules as a $\Z$-basis \cite[III.1.2]{Happel}.
If $e_i (A/{\rad A})$ appears $c_i$ times 
in the composition factors of a module $M \in \mod A$,
then in $K_0(\sD^\rb(\mod A))$, 
the element $[M]$ is equal to $\sum_{i=1}^n c_i[e_i (A/{\rad A})]$,
and for each complex $X=(X^a \to X^{a+1} \to \cdots \to X^b) \in \sD^\rb(\mod A)$,
the element $[X]$ is equal to $\sum_{j=a}^b (-1)^j [X^j]$.

Similarly, the Grothendieck group $K_0(\sK^\rb(\proj A))$ has
the family $([e_i A])_{i=1}^n$ of isoclasses of indecomposable projective modules
as a $\Z$-basis.
If a projective module 
$P \in \proj A$ is decomposed as $P \cong \bigoplus_{i=1}^n(e_i A)^{\oplus g_i}$,
then $[P]=\sum_{i=1}^n g_i [e_i A]$ holds,
and for each complex $P=(P^a \to P^{a+1} \to \cdots \to P^b) \in \sK^\rb(\proj A)$,
the element $[P]$ is equal to $\sum_{j=a}^b (-1)^j [P^j]$.

For these two Grothendieck groups,
there is a natural bilinear form
\begin{align*}
\langle ?, ? \rangle \colon 
K_0(\sK^\rb(\proj A)) \times K_0(\sD^\rb(\mod A)) \to \Z
\end{align*}
defined as 
\begin{align*}
\langle P, X \rangle := \sum_{i \in \Z} (-1)^j \dim_K \Hom_{\sD^\rb(\mod A)}(P,X[j])
\end{align*}
for any $P \in \sK^\rb(\proj A)$ and $X \in \sD^\rb(\mod A)$.
It is easy to see that
\begin{align*}
\langle e_i A, e_j(A/{\rad A}) \rangle = \begin{cases}
\dim_K \End_A (e_i(A/{\rad A})) & (i=j) \\
0 & (i \ne j)
\end{cases}.
\end{align*}
The ring $\End_A (e_i(A/{\rad A}))$ is a division $K$-algebra,
since $e_i(A/{\rad A})$ is a simple $A$-module.
We define an $n \times n$ diagonal matrix $\bD$ 
so that its $(i,i)$ entry is $\dim_K \End_A (e_i(A/{\rad A}))$.

Under this preparation, we fix $P \in \twosilt A$.
Decompose $P \cong \bigoplus_{k=1}^n P_k$, 
and define a vertical vector
$\bg_k=\begin{bmatrix} (g_k)_1 & (g_k)_2 & \cdots & (g_k)_n \end{bmatrix}^\rT \in \Z^n$ 
so that $[P_k]=\sum_{i=1}^n (g_k)_i [e_i A]$ holds in $K_0(\sK^\rb(\proj A))$
for each $k \in \{1,2,\ldots,n\}$.
Here, $?^\rT$ means the transpose.
The vector $\bg_k$ is called the \textit{g-vector} of $[P_k]$.
Since $P \in \twosilt A$,
the g-vectors $\bg_1,\bg_2,\ldots,\bg_n$ generate $\Z^n$
\cite[Theorem 2.27]{AI}.

We take $\cX \in \twosmc A$ corresponding to $P \in \twosilt A$
in Proposition \ref{silt_smc}.
By Lemma \ref{P_X_dual}, we may assume that $\cX=\{X_1,X_2,\ldots,X_n\}$
and that $\Hom_{\sD^\rb(\mod A)}(P_k,S_l)=0$ holds for $k \ne l$.
We define a vertical vector
$\bc_k=\begin{bmatrix} (c_k)_1 & (c_k)_2 & \cdots & (c_k)_n \end{bmatrix}^\rT \in \Z^n$ 
so that $[X_k]=\sum_{i=1}^n (c_k)_i [e_i (A/{\rad A})]$ holds in $K_0(\sD^\rb(\mod A))$
for each $k \in \{1,2,\ldots,n\}$.
We here call $\bc_k$ the \textit{c-vector} of $[X_k]$.
Since $X \in \twosmc A$,
the c-vectors $\bc_1,\bc_2,\ldots,\bc_n$ generate $\Z^n$
\cite[Lemma 3.3]{KY}.

It is clear that $\langle P_k,X_l \rangle=(\bg_k)^\rT \cdot \bD \cdot \bc_l$.
Actually, this value is given as follows.

\begin{Th}\label{P_X_dual_Gro}
Let $k,l \in \{1,2,\ldots,n\}$.
Then,
\begin{align*}
\langle P_k,X_l \rangle=\begin{cases}
\dim_K \End_{\sD^\rb(\mod A)}(X_l) & (k=l) \\
0 & (k \ne l)
\end{cases}
\end{align*}
\end{Th}

\begin{proof}
Let $k,l \in \{1,2,\ldots,n\}$.
Proposition \ref{silt_smc} implies $\Hom_{\sD^\rb(\mod A)}(P_k,X_l[j])=0$ 
for any $j \ne 0$.
Thus, $\langle P_k,X_l \rangle$ coincides with
$\dim_K \Hom_{\sD^\rb(\mod A)}(P_k,X_l)$.
Now, the assertion follows from Lemma \ref{P_X_dual}.
\end{proof}

We define $n \times n$ matrices $\bG$ and $\bC$ as 
$\bG:=\begin{bmatrix} \bg_1 & \bg_2 & \cdots & \bg_n\end{bmatrix}$ and
$\bC:=\begin{bmatrix} \bc_1 & \bc_2 & \cdots & \bc_n\end{bmatrix}$.
Then, the above theorem means $\bG^\rT \bD \bC=\bD'$,
where $\bD'$ is the $n \times n$ diagonal matrix
whose $(i,i)$ entry is $\dim_K \End_{\sD^\rb(\mod A)}(X_i)$.
Since the g-vectors generate $\Z^n$ and so do the c-vectors,
$\bG$ and $\bC$ are invertible matrices on $\Z$.
Therefore, the Smith normal forms of $\bD$ and $\bD'$ are the same.
In particular, if $\bD$ is the identity matrix, then so is $\bD'$.
We do not know whether the diagonal matrices $\bD$ and $\bD'$ always coincide
up to reordering of their entries,
but if the 2-term silting object 
$P \in \twosilt A$ can be obtained from $A \in \twosilt A$ or $A[1] \in \twosilt A$
by repeating mutations, then we have $\bD=\bD'$ up to reordering of their entries, 
see \cite[Remark 7.7]{KY}.

\section{Examples}\label{sec_ex}

\subsection{Semibricks for Nakayama algebras}\label{Nakayama_number}

In this subsection, we calculate the numbers of semibricks for 
the following Nakayama algebras.
For integers $n,l \ge 1$,
let $A_{n,l}$ be the algebra given by the following quiver with relations:
\begin{align*}
\begin{xy}
(  0,  0)*+{1}="1",
( 10,  0)*+{2}="2",
( 20,  0)*+{\cdots}="cdots",
( 30,  0)*+{n}="n",
\ar "1";"2"
\ar "2";"cdots"
\ar "cdots";"n"
\end{xy},
\quad \text{all the paths of length $l$ are 0},
\end{align*}
and $B_{n,l}$ be the algebra given by the following quiver with relations:
\begin{align*}
\begin{xy}
(  0,  0)*+{1}="1",
( 10,  0)*+{2}="2",
( 20,  0)*+{\cdots}="cdots",
( 30,  0)*+{n}="n",
\ar "1";"2"
\ar "2";"cdots"
\ar "cdots";"n"
\ar@{-} "n";(30,-6)
\ar@{-} (30,-6);(0,-6)
\ar (0,-6);"1"
\end{xy}
, \quad \text{all the paths of length $l$ are 0}.
\end{align*}
In this subsection, we calculate 
$a_{n,l}:=\#\sbrick A_{n,l}$ and $b_{n,l}:=\#\sbrick B_{n,l}$.
For $n=0$, let $A_{0,l}:=0$; hence $a_{0,l}=1$, since $\#\sbrick A_{0,l}=\{\emptyset\}$.
For convinience, we also set $a_{n,l}=0$ for $n<0$.

To state our result, we recall the $n$th Catalan numbers:
\begin{align*}
c_n=\frac{1}{n+1}
\begin{pmatrix} 2n \\ n \end{pmatrix} = \frac{(2n)!}{(n+1)! \cdot n!} \quad 
(n \ge 0).
\end{align*}
They satisfy the following equation: 
\begin{align*}
c_{n+1}=\sum_{i=0}^{n} c_i c_{n-i} \quad 
(n \ge 0). \tag{*}
\end{align*}
The next equations are the main result of this subsection.

\begin{Th}\label{Nakayama_main}
Let $n,l \ge 1$ be integers.
\begin{itemize}
\item[(1)]
The following equations hold:
\begin{align*}
a_{n,l}=c_{n+1} \quad (n=1,2,\ldots,l), \quad
a_{n,l}=2a_{n-1,l}+\displaystyle \sum_{i=2}^{l}c_{i-1}a_{n-i,l} \quad (n \ge 1).
\end{align*}
\item[(2)]
The following equations hold:
\begin{align*}
b_{n,l}=(n+1)c_n \quad (n=1,2,\ldots,l), \quad
b_{n,l}=2b_{n-1,l}+{\displaystyle \sum_{i=2}^{l}c_{i-1}b_{n-i,l}} \quad (n \ge l+1).
\end{align*}
\item[(3)]
Let $\xi_1,\xi_2,\ldots,\xi_l \in \C$ be the roots of the polynomial 
$F_l(X):=X^l-2X^{l-1}-\sum_{i=2}^{l}c_{i-1}X^{l-i}$ with multiplicities.
Then we have 
\begin{align*}
a_{n,l}=\sum_{\begin{smallmatrix} 
t_1,t_2\ldots,t_l \in \Z_{\ge 0}, \\ t_1+t_2+\cdots+t_l=n 
\end{smallmatrix}} \xi_1^{t_1} \xi_2^{t_2} \cdots \xi_l^{t_l} \quad \mathrm{and} \quad 
b_{n,l}=\xi_1^n+\xi_2^n+\cdots+\xi_l^n \quad 
(n \ge 1).
\end{align*}
\end{itemize}
\end{Th}

We remark that $F_l(X)$ is the characteristic polynomial of 
the common recurrence relation of the sequences 
$(a_{n,l})_{n=1}^\infty$ and $(b_{n,l})_{n=1}^\infty$
stated in the parts (1) and (2) of the above theorem.

From Theorem \ref{Nakayama_main}, we obtain Tables \ref{table_A} and \ref{table_B} below.

\begin{table}[h]
\begin{minipage}{0.45\hsize}
\centering
\begin{tabular}{cc|ccccccc}
& & \multicolumn{7}{c}{$n$} \\ 
& & 1 & 2 & 3 & 4 & 5 & 6 & 7 \\
\hline
\multirow{7}*{$l$} & 1 & 2 & 4 & 8 & 16 & 32 & 64 & 128 \\
& 2 & 2 & 5 & 12 & 29 & 70 & 169 & 408 \\
& 3 & 2 & 5 & 14 & 37 & 98 & 261 & 694 \\
& 4 & 2 & 5 & 14 & 42 & 118 & 331 & 934 \\
& 5 & 2 & 5 & 14 & 42 & 132 & 387 & 1130 \\
& 6 & 2 & 5 & 14 & 42 & 132 & 429 & 1298 \\
& 7 & 2 & 5 & 14 & 42 & 132 & 429 & 1430 \\
\end{tabular}
\caption{The table of $a_{n,l}$}\label{table_A}
\vspace{-\baselineskip}
\end{minipage}
\begin{minipage}{0.45\hsize}
\centering
\begin{tabular}{cc|ccccccc}
& & \multicolumn{7}{c}{$n$} \\ 
& & 1 & 2 & 3 & 4 & 5 & 6 & 7\\
\hline
\multirow{7}*{$l$} & 1 & 2 & 4 & 8 & 16 & 32 & 64 & 128 \\
& 2 & 2 & 6 & 14 & 34 & 82 & 198 & 478 \\
& 3 & 2 & 6 & 20 & 50 & 132 & 354 & 940 \\
& 4 & 2 & 6 & 20 & 70 & 182 & 504 & 1430 \\
& 5 & 2 & 6 & 20 & 70 & 252 & 672 & 1920 \\
& 6 & 2 & 6 & 20 & 70 & 252 & 924 & 2508 \\
& 7 & 2 & 6 & 20 & 70 & 252 & 924 & 3432 \\
\end{tabular}
\caption{The table of $b_{n,l}$}\label{table_B}
\vspace{-\baselineskip}
\end{minipage}
\end{table}

The algebras $A_{n,l}$ and $B_{n,l}$ are representation-finite,
so every semibrick is left finite.
For each of these algebras, Theorem \ref{sttilt_fsbrick} implies that
the number of semibricks coincides with
the number of support $\tau$-tilting modules.
Support $\tau$-tilting modules over Nakayama algebras were studied in the paper
\cite{Adachi},
which contains the tables of the numbers of support $\tau$-tilting modules
for the algebras $A_{n,l}$ and $B_{n,l}$ \cite[Tables 3 and 5]{Adachi}.
Compare them with our Tables \ref{table_A} and \ref{table_B} above.

For $u \in \{1,2,\ldots,n\}$, we write $e_u$ for the idempotent of $A_{n,l}$ and $B_{n,l}$
for the vertex $u$, and set $e_U:=\sum_{u \in U} e_u$
for each subset $U \subset \{1,2,\ldots,n\}$.
The notation $\supp M:=\{ u \in \{1,2,\ldots,n\} \mid Me_i \ne 0\}$
denotes the support of a module $M$ over $A_{n,l}$ or $B_{n,l}$.
We define (cyclic) intervals $[u,v],[[u,v]] \subset \{1,2,\ldots,n\}$ 
for $u,v \in \{1,2,\ldots,n\}$ as 
\begin{align*}
[u,v] = \begin{cases}
\{ u, u+1, \ldots, v\} & (u \le v) \\
\emptyset & (u > v) 
\end{cases}, \quad
[[u,v]] = \begin{cases}
\{ u, u+1, \ldots, v\} & (u \le v) \\
\{ u, u+1, \ldots, n, 1, 2, \ldots, v\} & (u > v) 
\end{cases}.
\end{align*}
Since $A_{n,l}$ and $B_{n,l}$ are Nakayama algebras,
for $u,v \in \{1,2,\ldots,n\}$,
there exists at most one brick $X$
satisfying $\operatorname{\mathsf{top}} X=S_u$ and $\soc X=S_v$,
where $S_u$ and $S_v$ are the simple modules corresponding to the vertices $u$ and $v$,
respectively.
If such an $X$ exists, then we set $S_{u,v}:=X$.

First, we remark the next property, which follows 
from $A_{n,l}=A_{n,n}$ for $n \le l$.

\begin{Lem}\label{A_down}
Let $n,l \ge 1$ be integers with $n \le l$.
Then we have $a_{n,l}=a_{n,n}$.
\end{Lem}

The next proposition is a keypoint to determine the values of $a_{n,l}$.

\begin{Prop}\label{zenka_A}
For any integers $n,l \ge 1$, the following equation holds:
\begin{align*}
a_{n,l}=2a_{n-1,l}+\sum_{i=2}^{l}a_{i-2,l}a_{n-i,l}.
\end{align*}
\end{Prop}

\begin{proof}
Let $m:=\min\{ n,l \}$.
A brick $S_{n-i+1,n}$ is well-defined for $i=1,\ldots,m$.

Let $\cS \in \sbrick A_{n,l}$.
Because $\cS$ is a semibrick, 
$\cS$ satisfies exactly one of the following conditions $(0), (1), \ldots, (m)$:
\begin{itemize}
\item[(0)] no brick $S \in \cS$ satisfies $n \in \supp S$,
\item[($i$)] 
($i \in \{1,2,\ldots,m\}$)
the brick $S_{n-i+1,n}$ belongs to $\cS$.
\end{itemize}
We define $\sbr(i)$ as the subset of $\sbrick A_{n,l}$
consisting of the semibricks
satisfying the condition $(i)$ for each $i \in \{0,1,\ldots,m\}$.
Then $\#\sbrick A_{n,l}=\sum_{i=0}^m \# \sbr(i)$ holds.

First, we clearly have $\sbr(0)=\sbrick A_{n,l}/\langle e_n \rangle$
and $A_{n,l}/\langle e_n \rangle \cong A_{n-1,l}$.
Thus, we get that $\# \sbr(0)=a_{n-1,l}$.

Second, we have a bijection
$\sbr(1) \to \sbrick A_{n,l}/\langle e_n \rangle$
defined as $\cS \mapsto \cS \setminus \{S_{n,n}\}$.
Because $A_{n,l}/\langle e_n \rangle \cong A_{n-1,l}$,
we get that $\# \sbr(1)=a_{n-1,l}$.

Next, for each $i \in \{2,3,\ldots,m\}$,
there exists a bijection
\begin{align*}
\sbr(i) &\to
\sbrick A_{n,l}/\langle e_{[n-i+1,n]} \rangle \times
\sbrick A_{n,l}/\langle 1-e_{[n-i+2,n-1]} \rangle \\
\cS &\mapsto 
(\{ S \in \cS \mid \supp S \cap [n-i+1,n]=\emptyset \}, 
\{ S \in \cS \mid \supp S \subset [n-i+2,n-1] \}).
\end{align*}
The inverse map 
is given by $(\cS_1,\cS_2) \mapsto \cS_1 \cup \cS_2 \cup \{ S_{n-i+1,n} \}$.
Since there are isomorphisms of algebras 
$A_{n,l}/\langle e_{[n-i+1,n]} \rangle \cong A_{n-i,l}$
and $A_{n,l}/\langle 1-e_{[n-i+2,n-1]} \rangle \cong A_{i-2,l}$ (including $i=2$),
we have $\# \sbr(i)=a_{i-2,l}a_{n-i,l}$.

Now, we obtain the equations
\begin{align*}
a_{n,l}=\#\sbrick A_{n,l}=\sum_{i=0}^m \# \sbr(i)
=2a_{n-1,l}+\sum_{i=2}^m a_{i-2,l}a_{n-i,l}.
\end{align*}
This equation implies the assertion,
because $a_{n,l}=0$ for $n<0$.
\end{proof}

We can regard the equation in Proposition \ref{zenka_A}
as a recurrence relation on $n$ with the coefficients $a_{n,l}$.
We determine the values of these coefficients.

\begin{Lem}\label{A_catalan}
We have $a_{n,l}=c_{n+1}$ if $0 \le n \le l$ and $1 \le l$.
\end{Lem}

\begin{proof}
We use induction on $n$.
Note that we have defined $a_{n,l}=0$ for $n<0$.

For $n=0,1$, we have $a_{0,l}=1=c_1$ and $a_{1,l}=2=c_2$ as desired.

If $n \ge 2$, by Proposition \ref{zenka_A} and the induction hypothesis, we have
\begin{align*}
a_{n,l}=2a_{n-1,l}+\sum_{i=2}^n a_{i-2,l}a_{n-i,l}
=2c_n+\sum_{i=2}^{n}c_{i-1}c_{n-i+1}
=\sum_{i=0}^n c_i c_{n-i}\overset{(*)}{=}c_{n+1},
\end{align*}
and the assertion holds.
The induction process is now complete.
\end{proof}

Now, we have a recurrence relation of 
the sequence $(a_{n,l})_{n=1}^\infty$ for each $l \ge 1$.

\begin{proof}[Proof of Theorem \ref{Nakayama_main} (1)]
Proposition \ref{zenka_A} and Lemma \ref{A_catalan} imply the assertion.
\end{proof}

Next, we consider the algebra $B_{n,l}$.

\begin{Lem}\label{B_down}
We have $\sbrick B_{n,n} = \sbrick B_{n,l}$ and $b_{n,l}=b_{n,n}$ if $1 \le n \le l$.
\end{Lem}

\begin{proof}
It is easy to see that $B_{n,n}$ is a quotient algebra of $B_{n,l}$,
and that any brick in $\mod B_{n,l}$ belongs to $\mod B_{n,n}$.
Thus the assertions follow.
\end{proof}

The value $b_{n,l}$ can be calculated from the sequence $(a_{n,l})_{n=1}^\infty$.

\begin{Prop}\label{zenka_B_A}
For any integers $n,l \ge 1$, the following equation holds:
\begin{align*}
b_{n,l}=2a_{n-1,l}+\sum_{i=2}^{l}ic_{i-1}a_{n-i,l}.
\end{align*}
\end{Prop}

\begin{proof}
We set $m:=\min\{ n,l \}$.

Let $\cS \in \sbrick B_{n,l}$.
Because $\cS$ is a semibrick, 
$\cS$ satisfies exactly one of the following conditions $(0), (1), \ldots, (m)$:
\begin{itemize}
\item[(0)] no brick $S \in \cS$ satisfies $n \in \supp S$,
\item[($i$)] 
($i \in \{1,2,\ldots,m\}$)
there exists some brick $S \in \cS$ satisfying $n \in \supp S$,
and $\max\{\dim_K S \mid S \in \cS, \ n \in \supp S\}=i$ holds.
\end{itemize}
We define $\sbr(i)$ as the subset of $\sbrick B_{n,l}$
consisting of the semibricks
satisfying the condition $(i)$ for each $i \in \{0,1,\ldots,m\}$.
Then $\#\sbrick B_{n,l}=\sum_{i=0}^m \# \sbr(i)$ holds.

First, we clearly have $\sbr(0)=\sbrick B_{n,l}/\langle e_n \rangle$
and $B_{n,l}/\langle e_n \rangle \cong A_{n-1,l}$.
Thus, we get that $\#\sbr(0)=a_{n-1,l}$.

Second, we have a bijection
$\sbr(1) \to \sbrick B_{n,l}/\langle e_n \rangle$
defined as $\cS \mapsto \cS \setminus \{S_{n,n}\}$.
Because $B_{n,l}/\langle e_n \rangle \cong A_{n-1,l}$,
we get that $\#\sbr(1)=a_{n-1,l}$.

Next, let $i \in \{2,3,\ldots,m\}$.
For each pair $(u,v)$ satisfying $n \in [[u,v]]$ and $\#[[u,v]]=i$,
a brick $S_{u,v}$ is well-defined,
so let $\sbr(u,v)$ be the subset of $\sbr(i)$
consisting of $\cS$ with $S_{u,v} \in \cS$.

If $\cS \in \sbr(u,v)$,
then any brick $S \in \cS$ different from $S_{u,v}$
satisfies $\supp S \cap [[u,v]]=\emptyset$ or 
$\supp S \subset [[u,v]] \setminus \{u,v\}$.
In particular, 
$S_{u,v}$ is the unique brick $S \in \cS$ satisfying 
$n \in \supp S$ and $\#\supp S=i$.
 
Thus, we have the following decomposition into a disjoint union:
\begin{align*}
\sbr(i) = \coprod_{n \in [[u,v]],\: \#[[u,v]]=i} \sbr(u,v).
\end{align*}
For a pair $(u,v)$ satisfying $n \in [[u,v]]$ and $\#[[u,v]]=i$, 
there exists a bijection
\begin{align*}
\sbr(u,v) &\to
\sbrick B_{n,l}/\langle e_{[[u,v]]} \rangle \times
\sbrick B_{n,l}/\langle 1-e_{[[u,v]]\setminus\{u,v\}} \rangle \\
\cS &\mapsto 
(\{ S \in \cS \mid \supp S \cap [[u,v]]=\emptyset \}, 
\{ S \in \cS \mid \supp S \subset [[u,v]] \setminus \{u,v\} \}).
\end{align*}
The inverse is given by $(\cS_1,\cS_2) \mapsto \cS_1 \cup \cS_2 \cup \{ S_{u,v} \}$.
Since there are isomorphisms of algebras, $B_{n,l}/\langle e_{[[u,v]]}\rangle \cong A_{n-i,l}$
and $B_{n,l}/\langle 1-e_{[[u,v]]\setminus\{u,v\}} \rangle \cong A_{i-2,l}$
(including $i=2$),
we get that $\#\sbr(u,v)=a_{i-2,l}a_{n-i,l}=c_{i-1}a_{n-i,l}$.
There exist exactly $i$ pairs
$(u,v)$ satisfying $n \in [[u,v]]$ and $\#[[u,v]]=i$,
so we have $\#\sbr(i)=ic_{i-1}a_{n-i,l}$.

Now, we have equations
\begin{align*}
b_{n,l}=\#\sbrick B_{n,l}=\sum_{i=0}^m \#\sbr(i)
=2a_{n-1,l}+\sum_{i=2}^m ic_{i-1}a_{n-i,l}.
\end{align*}
This equation implies the assertion, 
because $a_{n,l}=0$ for $n<0$.
\end{proof}

We can also determine $b_{n,l}$ explicitly if $n \le l$.

\begin{Lem}\label{B_catalan}
We have $b_{n,l}=(n+1)c_n$ if $1 \le n \le l$.
\end{Lem}

\begin{proof}
By Proposition \ref{zenka_B_A} and Lemma \ref{A_catalan}, we have
\begin{align*}
b_{n,l}=2a_{n-1,l}+\sum_{i=2}^n ia_{i-2,l}a_{n-i,l}
=2c_n+\sum_{i=2}^{n}ic_{i-1}c_{n-i+1}.
\end{align*}
On the other hand, the equalities
\begin{align*}
\sum_{i=2}^{n}ic_{i-1}c_{n-i+1}=\sum_{i=1}^{n-1} (i+1) c_i c_{n-i}
=\sum_{i=1}^{n-1} (n-i+1) c_i c_{n-i},
\end{align*}
hold, so we get that
\begin{align*}
2\sum_{i=2}^{n}ic_{i-1}c_{n-i+1}
&=\sum_{i=1}^{n-1} (i+1) c_i c_{n-i}+\sum_{i=1}^{n-1} (n-i+1) c_i c_{n-i} \\
&=(n+2) \sum_{i=1}^{n-1} c_i c_{n-i}
=(n+2) \left( \sum_{i=0}^n c_i c_{n-i} -2c_n \right)
\overset{(*)}{=}(n+2)(c_{n+1}-2c_n).
\end{align*}
Since $(n+2)c_{n+1}=2(2n+1)c_n$ holds by definition, we have
$\sum_{i=2}^{n}ic_{i-1}c_{n-i+1}=(n-1)c_n$;
hence, $b_{n,l}=(n+1)c_n$.
\end{proof}

Now, we prove Theorem \ref{Nakayama_main} (2).

\begin{proof}[Proof of Theorem \ref{Nakayama_main} (2)]
We have obtained the first equation in Lemma \ref{B_catalan}. 
For the second equation, we use the assumption $n \ge l+1$; 
then Propositions \ref{zenka_A} and \ref{zenka_B_A} imply the assertion.
\end{proof}

Let $l \ge 1$ be an integer.
It remains to prove Theorem \ref{Nakayama_main} (3).
For each $i \in \{0,1,\ldots,l\}$, 
we define an integer $d_i$ so that $F_l(X)=\sum_{i=0}^{l}d_iX^{l-i}$.
Then we have $d_0=1$, $d_1=-2$, and $d_i=-c_{i-1}$.
Recall the following symmetric polynomials 
$\be_n, \bh_n, \bp_n \in \Z[X_1,X_2,\ldots,X_l]$ (see \cite{MacDonald}):
\begin{align*}
\be_n(X_1,\ldots,X_l)&:=\sum_{J \subset \{1,2,\ldots,l\}, \: \#J=n} \;
\prod_{j \in J} X_j \quad (n=0,1,\ldots,l), \\
\bh_n(X_1,\ldots,X_l)&:=\sum_{\begin{smallmatrix} 
t_1,t_2\ldots,t_l \in \Z_{\ge 0}, \\ t_1+t_2+\cdots+t_l=n 
\end{smallmatrix}} X_1^{t_1} X_2^{t_2} \cdots X_l^{t_l} \quad (n \in \Z), \\
\bp_n(X_1,\ldots,X_l)&:=X_1^n+X_2^n+\cdots+X_l^n \quad (n \ge 1).
\end{align*}
In particular, we have $\be_0=1$, $\bh_0=1$, and $\bh_n=0$ for $n<0$.
We need a technical lemma on these polynomials.

\begin{Lem}\label{poly_sum}\cite{MacDonald}
For any integer $n \ge 1$, we have polynomial equations 
$\sum_{i=0}^l (-1)^i \be_i \bh_{n-i}=0$ and
$\sum_{i=1}^l (-1)^{i-1} i \be_i \bh_{n-i}=\bp_n$.
\end{Lem}

\begin{proof}
The first equation is shown in \cite[I.2, (2.6$'$)]{MacDonald},
and the second one is deduced from the second and the third equations in 
\cite[I.2, Example 8]{MacDonald}.
\end{proof}

Now, we complete the proof of Theorem \ref{Nakayama_main}.

\begin{proof}[Proof of Theorem \ref{Nakayama_main} (3)]
We first show the assertion for $a_{n,l}$.
Fix $l \ge 1$.
For any $n \in \Z$, we set $a'_{n,l}:=\bh_n(\xi_1,\xi_2,\ldots,\xi_l)$.
It is enough to show that $a'_{n,l}=a_{n,l}$ for all $n \ge 1$.

We claim that $\sum_{i=0}^{l}d_ia'_{n-i,l}=0$ for $n \ge 1$. 
Indeed, 
$\be_i(\xi_1,\xi_2,\ldots,\xi_l)=(-1)^id_i$ holds for each $i=0,1,\ldots,l$.
Substituting $\xi_j$ for $X_j$ in the first equation in Lemma \ref{poly_sum}, 
we have $\sum_{i=0}^{l}d_ia'_{n-i,l}=0$ for $n \ge 1$.

By Theorem \ref{Nakayama_main} (1), we also have 
$\sum_{i=0}^{l}d_ia_{n-i,l}=0$ for $n \ge 1$.
Because $a_{0,l}=1=a'_{0,l}$ and $a_{n,l}=0=a'_{n,l}$ for $n<0$,
it can be shown inductively that $a_{n,l}=a'_{n,l}$ for all $n \ge 1$.
The assertion for $a_{n,l}$ is now proved.

We next show the assertion for $b_{n,l}$.
In the second equation in Lemma \ref{poly_sum},
by setting $X_j=\xi_j$ and using $a'_{n,l}=a_{n,l}$,
we have $\sum_{i=1}^l (-id_i) a_{n-i,l}=\xi_1^n+\xi_2^n+\cdots+\xi_l^n$.
By Proposition \ref{zenka_B_A}, 
it is easily seen that $b_{n,l}=\sum_{i=1}^l (-id_i)a_{n-i,l}$, 
which yields $b_{n,l}=\xi_1^n+\xi_2^n+\cdots+\xi_l^n$.
\end{proof}

\subsection{Wide subcategories for tilted algebras}\label{Section_tilt}

We investigate functorially finiteness
of wide subcategories and semibricks mainly for tilted algebras here.
Recall from Subsection \ref{wide_mod} 
that $\fwide A$ is the set of functorially finite wide subcategories of $\mod A$.
For a given algebra $A$,
we consider the following two conditions in this subsection.

\begin{Cond}\label{wide_fwide}
The equalities $\fwide A = \fLwide A$ and $\fwide A = \fRwide A$ hold.
\end{Cond}

Note that we always have the inclusions 
$\fLwide A \subset \fwide A$ and $\fRwide A \subset \fwide A$,
see Proposition \ref{left_func}.

\begin{Cond}\label{subset}
For any $\cS \in \fLsbrick A$, any subset of $\cS$ belongs to $\fLsbrick A$.
\end{Cond}

Ingalls--Thomas studied functorially finite wide subcategories for hereditary algebras
\cite{IT}.
From their results, it follows that 
Condition \ref{wide_fwide} holds if $A$ is hereditary.

\begin{Prop}\label{fwide_MS}
If $A$ is hereditary, $\fwide A=\fLwide A=\fRwide A$ hold.
\end{Prop}

\begin{proof}
We only show that $\fwide A=\fLwide A$,
because the opposite algebra $A^\mathrm{op}$ is also hereditary.

By Proposition \ref{left_func}, it suffices to show that $\fwide A \subset \fLwide A$.
Let $\cW \in \fwide A$, then $\Fac \cW=\sT(\cW)$ holds by \cite[Proposition 2.13]{IT}.  
Then \cite[Corollary 2.17]{IT} implies $\sT(\cW)=\Fac \cW \in \ftors A$, which yields
$\cW \in \fLwide A$.
Now $\fwide A \subset \fLwide A$ has been proved,
and we obtain the assertion.
\end{proof}

We also remark relationship between Conditions \ref{wide_fwide} and \ref{subset}.

\begin{Prop}
If $\fwide A = \fLwide A$, then 
the algebra $A$ satisfies Condition \ref{subset}.
\end{Prop}

\begin{proof}
Assume $\fwide A = \fLwide A$.
Let $\cS \in \fLsbrick A$ and $\cS_1 \subset \cS$.

Set $\cW:=\Filt \cS$ and $\cW_1:=\Filt \cS_1$.
Since $\cS \in \fLsbrick A$, we have $\cW \in \fLwide A$,
and Proposition \ref{left_func} implies that $\cW \in \fwide A$.

We claim that $\cW_1$ is a functorially finite subcategory of $\cW$.
There exists $M \in \sttilt A$ such that 
$\ind (M/{\rad_B M})=\cS$ with $B:=\End_A(M)$ by Theorem \ref{sttilt_fsbrick}.
By Theorem \ref{wide_B}, 
we obtain an equivalence $\Hom_A(M,?) \colon \cW \to \mod C$ 
for some algebra $C$
sending the elements of $\cS$ to the simple $C$-modules.
Thus, $\Hom_A(M,\cW_1)$ is equivalent to a Serre subcategory of $\mod C$,
and it is functorially finite in $\mod C$.
This implies that $\cW_1$ is functorially finite in $\cW$.

Because $\cW$ belongs to $\fwide A$, we get that $\cW_1 \in \fwide A$.
By assumption, we have $\cW_1 \in \fLwide A$. 
Therefore, $A$ satisfies Condition \ref{subset}. 
\end{proof}

Thus, every hereditary algebra $A$ satisfies 
both Conditions \ref{wide_fwide} and \ref{subset}.
On the other hand, neither of these conditions holds for 
the following algebra $A$ in Example \ref{MS_ex}.

\begin{Ex}\label{MS_ex}
Let $A$ be the finite-dimensional $K$-algebra
given by the following quiver with relation:
\begin{align*}
\begin{xy}
( 0,0) *+{1} = "1",
(10,0) *+{2} = "2",
(20,0) *+{3} = "3"
\ar         ^{\gamma} "2"; "3"
\ar @<1mm>  ^{\alpha}  "1"; "2"
\ar @<-1mm> _{\beta} "1"; "2"
\end{xy}, \quad
\alpha \gamma = 0.
\end{align*}
We write $e_i$ for the corresponding idempotent for the vertex $i=1,2,3$.  

We consider the indecomposable injective module $I_3$:
\begin{align*}
I_3 = \left( \begin{smallmatrix} 1 \\ 2 \\ 3 \end{smallmatrix} \right), \quad
I_3 \alpha=0, \quad I_3 \beta=I_3e_2.
\end{align*}

First, $I_3$ is a brick, so $\Filt I_3 \in \wide A$ follows 
from Proposition \ref{sbrick_wide}.
Because $I_3$ is injective, we have $\Filt I_3 = \add I_3$, so it belongs to $\fwide A$.

We claim that $\Filt I_3 \notin \fLwide A$,
that is, $\sT(\Filt I_3)=\sT(I_3) \notin \ftors A$.
We assume that $\sT(I_3) \in \ftors A$ and deduce a contradiction.
We consider the quotient algebra $A':=A/\langle e_3 \rangle$.
The algebra $A'$ is isomorphic to the path algebra of 
the Kronecker quiver $1 \rightrightarrows 2$.
Let $M:=I_3 \otimes_A A'$.
Because $\sT(I_3)=\Filt(\Fac I_3) \in \ftors A$,
the torsion class $\sT(M)=\Filt(\Fac M)$ must belong to $\ftors A'$.
We can see $M \in \brick A'$,
because
\begin{align*}
M= \left( \begin{smallmatrix} 1 \\ 2 \end{smallmatrix} \right), \quad 
M \alpha=0, \quad M \beta=M e_2,
\end{align*}
so we have $\Filt M \in \wide A'$.
Because $\sT(M)=\sT(\Filt M) \in \ftors A'$, the wide subcategory
$\Filt M$ must be in $\fwide A'$ by Proposition \ref{left_func}.
However, $M$ is in a homogeneous tube in the Auslander--Reiten quiver of $\mod A'$,
so we obtain that $\Filt M \notin \fwide A$, and it is a contradiction.
Thus, the claim $\sT(\Filt I_3)=\sT(I_3) \notin \ftors A$ is now shown.

Therefore, Condition \ref{wide_fwide} does not hold.

Moreover, we claim $\fLwide A \ne \fRwide A$.
To prove this, we show that $\Filt I_3 \in \fRwide A$, that is, $\sF(I_3) \in \ftorf A$.
We use the map $\tirigid A \to \fRsbrick A$ in Theorem \ref{sttilt_fsbrick}.
Clearly, $I_3 \in \tirigid A$ and $I_3 \in \brick A$.
Thus, the corresponding right finite semibrick is $\{ I_3 \}$,
and we have $\sF(I_3) \in \ftorf A$.

For Condition \ref{subset}, we consider a simple module $S_2$:
\begin{align*}
S_2 = \left( \begin{smallmatrix} 2 \end{smallmatrix} \right).
\end{align*}
We can see $\{I_3,S_2\} \in \sbrick A$,
and we have shown that $\{ I_3 \} \notin \fLsbrick A$.
To show that Condition \ref{subset} does not hold, 
it is sufficient to prove that $\{ I_3, S_2 \} \in \fLsbrick A$,
that is, $\sT(I_3,S_2) \in \ftors A$.
We use the map $\trigid A \to \fLsbrick A$ in Theorem \ref{sttilt_fsbrick}.
Let $P_1$ be the following indecomposable projective module:
\begin{align*}
P_1 = \left( \begin{smallmatrix} &1&& \\ 2&&2& \\ &&&3 \end{smallmatrix} \right).
\end{align*}
By straightforward calculation, 
we obtain that $S_2 \oplus P_1$ is a $\tau$-rigid $A$-module,
and then the corresponding left finite semibrick is $\{ I_3, S_2 \}$.
This implies that $\sT(I_3,S_2) \in \ftors A$.
Thus, Condition \ref{subset} does not hold.
\end{Ex}

The algebra in Example \ref{MS_ex} is a tilted algebra of type 
$\tilde{\mathbb{A}}_2$.
The exchange quivers of $\sttilt A$ and $\stitilt A$ are 
written in Figure \ref{A2} below.
The corresponding left finite semibricks and 
the corresponding right finite semibricks are denoted by bold letters. 
We can see that $\fLsbrick A \subset \fRsbrick A$ and 
that $\{ I_3 \}$ is the unique element in $\fRsbrick A \setminus \fLsbrick A$.
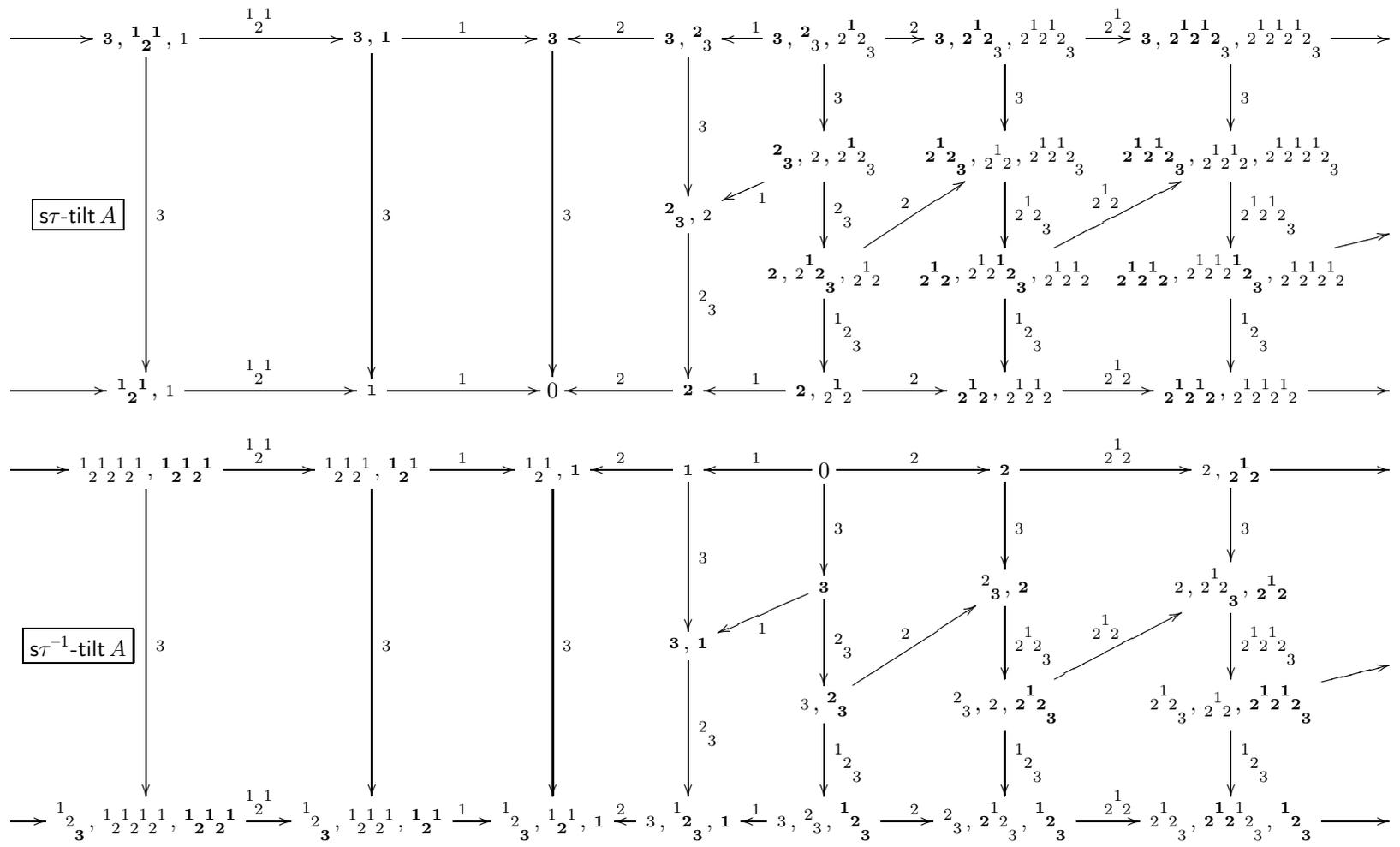
\begin{figure}[pt]
{\small 
\begin{align*}
\rotatebox{90}{
\begin{xy}
(-108.5,0)*+{\boxed{\sttilt A}},
(  7, 27)*+{
\begin{smallmatrix} \rthr \end{smallmatrix},
\begin{smallmatrix} \rtwo\!\!& \\ &3 \end{smallmatrix},
\begin{smallmatrix} &\rone\!\!&& \\ 2\!\!&&2\!\!& \\ &&&3 \end{smallmatrix}
}="55",
( 35, 27)*+{
\begin{smallmatrix} \rthr \end{smallmatrix},
\begin{smallmatrix} &\rone\!\!&& \\ \rtwo\!\!&&\rtwo\!\!& \\ &&&3 \end{smallmatrix},
\begin{smallmatrix} &1\!\!&&1\!\!&& \\ 2\!\!&&2\!\!&&2\!\!& \\ &&&&&3 \end{smallmatrix}
}="65",
( 70, 27)*+{
\begin{smallmatrix} \rthr \end{smallmatrix},
\begin{smallmatrix} &\rone\!\!&&\rone\!\!&& \\ \rtwo\!\!&&\rtwo\!\!&&\rtwo\!\!& \\ &&&&&3 \end{smallmatrix},
\begin{smallmatrix} &1\!\!&&1\!\!&&1\!\!&& \\ 2\!\!&&2\!\!&&2\!\!&&2\!\!& \\ &&&&&&&3 \end{smallmatrix}
}="75",
(-14, 27)*+{
\begin{smallmatrix} \rthr \end{smallmatrix},
\begin{smallmatrix} \rtwo\!\!& \\ &3 \end{smallmatrix}
}="44",
(  7,  9)*+{
\begin{smallmatrix} \rtwo\!\!& \\ &\rthr \end{smallmatrix},
\begin{smallmatrix} 2 \end{smallmatrix},
\begin{smallmatrix} &\rone\!\!&& \\ 2\!\!&&2\!\!& \\ &&&3 \end{smallmatrix}
}="54",
( 35,  9)*+{
\begin{smallmatrix} &\rone\!\!&& \\ \rtwo\!\!&&\rtwo\!\!& \\ &&&\rthr \end{smallmatrix},
\begin{smallmatrix} &1\!\!& \\ 2\!\!&&2 \end{smallmatrix},
\begin{smallmatrix} &1\!\!&&1\!\!&& \\ 2\!\!&&2\!\!&&2\!\!& \\ &&&&&3 \end{smallmatrix}
}="64",
( 70,  9)*+{
\begin{smallmatrix} &\rone\!\!&&\rone\!\!&& \\ \rtwo\!\!&&\rtwo\!\!&&\rtwo\!\!& \\ &&&&&\rthr \end{smallmatrix},
\begin{smallmatrix} &1\!\!&&1\!\!& \\ 2\!\!&&2\!\!&&2 \end{smallmatrix},
\begin{smallmatrix} &1\!\!&&1\!\!&&1\!\!&& \\ 2\!\!&&2\!\!&&2\!\!&&2\!\!& \\ &&&&&&&3 \end{smallmatrix} 
}="74",
(-14,  0)*+{
\begin{smallmatrix} \rtwo\!\!& \\ &\rthr \end{smallmatrix},
\begin{smallmatrix} 2 \end{smallmatrix}
}="43",
(  7, -9)*+{
\begin{smallmatrix} \rtwo \end{smallmatrix},
\begin{smallmatrix} &\rone\!\!&& \\ 2\!\!&&\rtwo\!\!& \\ &&&\rthr \end{smallmatrix},
\begin{smallmatrix} &1\!\!& \\ 2\!\!&&2 \end{smallmatrix}
}="53",
( 35, -9)*+{
\begin{smallmatrix} &\rone\!\!& \\ \rtwo\!\!&&\rtwo \end{smallmatrix},
\begin{smallmatrix} &1\!\!&&\rone\!\!&& \\ 2\!\!&&2\!\!&&\rtwo\!\!& \\ &&&&&\rthr \end{smallmatrix},
\begin{smallmatrix} &1\!\!&&1\!\!& \\ 2\!\!&&2\!\!&&2 \end{smallmatrix}
}="63",
( 70, -9)*+{
\begin{smallmatrix} &\rone\!\!&&\rone\!\!& \\ \rtwo\!\!&&\rtwo\!\!&&\rtwo \end{smallmatrix},
\begin{smallmatrix} &1\!\!&&1\!\!&&\rone\!\!&& \\ 2\!\!&&2\!\!&&2\!\!&&\rtwo\!\!& \\ &&&&&&&\rthr \end{smallmatrix},
\begin{smallmatrix} &1\!\!&&1\!\!&&1\!\!& \\ 2\!\!&&2\!\!&&2\!\!&&2 \end{smallmatrix}
}="73",
(-98, 27)*+{
\begin{smallmatrix} \rthr \end{smallmatrix},
\begin{smallmatrix} \rone\!\!&&\rone \\ &\rtwo\!\!& \end{smallmatrix},
\begin{smallmatrix} 1 \end{smallmatrix}
}="12",
(-63, 27)*+{
\begin{smallmatrix} \rthr \end{smallmatrix},
\begin{smallmatrix} \rone \end{smallmatrix}
}="22",
(-35, 27)*+{
\begin{smallmatrix} \rthr \end{smallmatrix}
}="32",
(-14,-27)*+{
\begin{smallmatrix} \rtwo \end{smallmatrix}
}="42",
(  7,-27)*+{
\begin{smallmatrix} \rtwo \end{smallmatrix},
\begin{smallmatrix} &\rone\!\!& \\ 2\!\!&&2 \end{smallmatrix}
}="52",
( 35,-27)*+{
\begin{smallmatrix} &\rone\!\!& \\ \rtwo\!\!&&\rtwo \end{smallmatrix},
\begin{smallmatrix} &1\!\!&&1\!\!& \\ 2\!\!&&2\!\!&&2 \end{smallmatrix}
}="62",
( 70,-27)*+{
\begin{smallmatrix} &\rone\!\!&&\rone\!\!& \\ \rtwo\!\!&&\rtwo\!\!&&\rtwo \end{smallmatrix},
\begin{smallmatrix} &1\!\!&&1\!\!&&1\!\!& \\ 2\!\!&&2\!\!&&2\!\!&&2 \end{smallmatrix}
}="72",
(-98,-27)*+{
\begin{smallmatrix} \rone\!\!&&\rone \\ &\rtwo\!\!& \end{smallmatrix},
\begin{smallmatrix} 1 \end{smallmatrix}
}="11",
(-63,-27)*+{
\begin{smallmatrix} \rone \end{smallmatrix}
}="21",
(-35,-27)*+{
0
}="31",
\ar^{\begin{smallmatrix} 2 \end{smallmatrix}} "55";"65" 
\ar_{\begin{smallmatrix} 1 \end{smallmatrix}} "55";"44" 
\ar^{\begin{smallmatrix} 3 \end{smallmatrix}} "55";"54"
\ar^{\begin{smallmatrix} &1\!\!& \\ 2\!\!&&2 \end{smallmatrix}} "65";"75" 
\ar^{\begin{smallmatrix} 3 \end{smallmatrix}} "65";"64"
\ar^{\begin{smallmatrix} 3 \end{smallmatrix}} "75";"74"
\ar^{\begin{smallmatrix} 3 \end{smallmatrix}} "44";"43" 
\ar_{\begin{smallmatrix} 2 \end{smallmatrix}} "44";"32"
\ar^{\begin{smallmatrix} 1 \end{smallmatrix}} "54";"43"
\ar^{\begin{smallmatrix} 2\!\!& \\ &3 \end{smallmatrix}} "54";"53" 
\ar^{\begin{smallmatrix} &1\!\!&& \\ 2\!\!&&2\!\!& \\ &&&3 \end{smallmatrix}} "64";"63"
\ar^{\begin{smallmatrix} &1\!\!&&1\!\!&& \\ 2\!\!&&2\!\!&&2\!\!& \\ &&&&&3 \end{smallmatrix}} "74";"73"  
\ar^{\begin{smallmatrix} 2\!\!& \\ &3 \end{smallmatrix}} "43";"42" 
\ar^{\begin{smallmatrix} 2 \end{smallmatrix}} "53";"64" 
\ar^{\begin{smallmatrix} 1\!\!&& \\ &2\!\!& \\ &&3 \end{smallmatrix}} "53";"52" 
\ar^{\begin{smallmatrix} 1\!\!&& \\ &2\!\!& \\ &&3 \end{smallmatrix}} "63";"62"
\ar^{\begin{smallmatrix} &1\!\!& \\ 2\!\!&&2 \end{smallmatrix}} "63";"74"
\ar^{\begin{smallmatrix} 1\!\!&& \\ &2\!\!& \\ &&3 \end{smallmatrix}} "73";"72"  
\ar^{\begin{smallmatrix} 1\!\!&&1 \\ &2\!\!& \end{smallmatrix}} "12";"22"
\ar^{\begin{smallmatrix} 3 \end{smallmatrix}} "12";"11"
\ar^{\begin{smallmatrix} 1 \end{smallmatrix}} "22";"32"
\ar^{\begin{smallmatrix} 3 \end{smallmatrix}} "22";"21"
\ar^{\begin{smallmatrix} 3 \end{smallmatrix}} "32";"31"
\ar_{\begin{smallmatrix} 2 \end{smallmatrix}} "42";"31"
\ar_{\begin{smallmatrix} 1 \end{smallmatrix}} "52";"42"
\ar^{\begin{smallmatrix} 2 \end{smallmatrix}} "52";"62" 
\ar^{\begin{smallmatrix} &1\!\!& \\ 2\!\!&&2 \end{smallmatrix}} "62";"72"
\ar^{\begin{smallmatrix} 1\!\!&&1 \\ &2\!\!& \end{smallmatrix}} "11";"21"
\ar^{\begin{smallmatrix} 1 \end{smallmatrix}} "21";"31"
\ar (-119, 27);"12"
\ar (-119,-27);"11"
\ar "75";(94.5, 27)
\ar "73";(94.5, -3)
\ar "72";(94.5,-27)
\end{xy}
}
\quad
\rotatebox{90}{
\begin{xy}
(-108.5,0)*+{\boxed{\stitilt A}},
(  7, 27)*+{
0
}="55",
( 35, 27)*+{
\begin{smallmatrix} \rtwo \end{smallmatrix}
}="65",
( 70, 27)*+{
\begin{smallmatrix} 2 \end{smallmatrix},
\begin{smallmatrix} &\rone\!\!& \\ \rtwo\!\!&&\rtwo \end{smallmatrix}
}="75",
(-14, 27)*+{
\begin{smallmatrix} \rone \end{smallmatrix}
}="44",
(  7,  9)*+{
\begin{smallmatrix} \rthr \end{smallmatrix}
}="54",
( 35,  9)*+{
\begin{smallmatrix} 2\!\!& \\ &\rthr \end{smallmatrix},
\begin{smallmatrix} \rtwo \end{smallmatrix}
}="64",
( 70,  9)*+{
\begin{smallmatrix} 2 \end{smallmatrix},
\begin{smallmatrix} &1\!\!&& \\ 2\!\!&&2\!\!& \\ &&&\rthr \end{smallmatrix},
\begin{smallmatrix} &\rone\!\!& \\ \rtwo\!\!&&\rtwo \end{smallmatrix}
}="74",
(-14,  0)*+{
\begin{smallmatrix} \rthr \end{smallmatrix},
\begin{smallmatrix} \rone \end{smallmatrix}
}="43",
(  7, -9)*+{
\begin{smallmatrix} 3 \end{smallmatrix},
\begin{smallmatrix} \rtwo\!\!& \\ &\rthr \end{smallmatrix}
}="53",
( 35, -9)*+{
\begin{smallmatrix} 2\!\!& \\ &3 \end{smallmatrix},
\begin{smallmatrix} 2 \end{smallmatrix},
\begin{smallmatrix} &\rone\!\!&& \\ \rtwo\!\!&&\rtwo\!\!& \\ &&&\rthr \end{smallmatrix}
}="63",
( 70, -9)*+{
\begin{smallmatrix} &1\!\!&& \\ 2\!\!&&2\!\!& \\ &&&3 \end{smallmatrix},
\begin{smallmatrix} &1\!\!& \\ 2\!\!&&2 \end{smallmatrix},
\begin{smallmatrix} &\rone\!\!&&\rone\!\!&& \\ \rtwo\!\!&&\rtwo\!\!&&\rtwo\!\!& \\ &&&&&\rthr \end{smallmatrix}
}="73",
(-98, 27)*+{
\begin{smallmatrix} 1\!\!&&1\!\!&&1\!\!&&1 \\ &2\!\!&&2\!\!&&2\!\!& \end{smallmatrix},
\begin{smallmatrix} \rone\!\!&&\rone\!\!&&\rone \\ &\rtwo\!\!&&\rtwo\!\!& \end{smallmatrix}
}="12",
(-63, 27)*+{
\begin{smallmatrix} 1\!\!&&1\!\!&&1 \\ &2\!\!&&2\!\!& \end{smallmatrix},
\begin{smallmatrix} \rone\!\!&&\rone \\ &\rtwo\!\!& \end{smallmatrix}
}="22",
(-35, 27)*+{
\begin{smallmatrix} 1\!\!&&1 \\ &2\!\!& \end{smallmatrix},
\begin{smallmatrix} \rone \end{smallmatrix}
}="32",
(-14,-27)*+{
\begin{smallmatrix} 3 \end{smallmatrix},
\begin{smallmatrix} 1\!\!&& \\ &\rtwo\!\!& \\ &&\rthr \end{smallmatrix},
\begin{smallmatrix} \rone \end{smallmatrix}
}="42",
(  7,-27)*+{
\begin{smallmatrix} 3 \end{smallmatrix},
\begin{smallmatrix} 2\!\!& \\ &3 \end{smallmatrix},
\begin{smallmatrix} \rone\!\!&& \\ &\rtwo\!\!& \\ &&\rthr \end{smallmatrix}
}="52",
( 35,-27)*+{
\begin{smallmatrix} 2\!\!& \\ &3 \end{smallmatrix},
\begin{smallmatrix} &1\!\!&& \\ \rtwo\!\!&&2\!\!& \\ &&&3 \end{smallmatrix},
\begin{smallmatrix} \rone\!\!&& \\ &\rtwo\!\!& \\ &&\rthr \end{smallmatrix}
}="62",
( 70,-27)*+{
\begin{smallmatrix} &1\!\!&& \\ 2\!\!&&2\!\!& \\ &&&3 \end{smallmatrix},
\begin{smallmatrix} &\rone\!\!&&1\!\!&& \\ \rtwo\!\!&&\rtwo\!\!&&2\!\!& \\ &&&&&3 \end{smallmatrix},
\begin{smallmatrix} \rone\!\!&& \\ &\rtwo\!\!& \\ &&\rthr \end{smallmatrix}
}="72",
(-98,-27)*+{
\begin{smallmatrix} 1\!\!&& \\ &2\!\!& \\ &&\rthr \end{smallmatrix},
\begin{smallmatrix} 1\!\!&&1\!\!&&1\!\!&&1 \\ &2\!\!&&2\!\!&&2\!\!& \end{smallmatrix},
\begin{smallmatrix} \rone\!\!&&\rone\!\!&&\rone \\ &\rtwo\!\!&&\rtwo\!\!& \end{smallmatrix}
}="11",
(-63,-27)*+{
\begin{smallmatrix} 1\!\!&& \\ &2\!\!& \\ &&\rthr \end{smallmatrix},
\begin{smallmatrix} 1\!\!&&1\!\!&&1 \\ &2\!\!&&2\!\!& \end{smallmatrix},
\begin{smallmatrix} \rone\!\!&&\rone \\ &\rtwo\!\!& \end{smallmatrix}
}="21",
(-35,-27)*+{
\begin{smallmatrix} 1\!\!&& \\ &2\!\!& \\ &&\rthr \end{smallmatrix},
\begin{smallmatrix} 1\!\!&&1 \\ &\rtwo\!\!& \end{smallmatrix},
\begin{smallmatrix} \rone \end{smallmatrix}
}="31",
\ar^{\begin{smallmatrix} 2 \end{smallmatrix}} "55";"65" 
\ar_{\begin{smallmatrix} 1 \end{smallmatrix}} "55";"44" 
\ar^{\begin{smallmatrix} 3 \end{smallmatrix}} "55";"54"
\ar^{\begin{smallmatrix} &1\!\!& \\ 2\!\!&&2 \end{smallmatrix}} "65";"75" 
\ar^{\begin{smallmatrix} 3 \end{smallmatrix}} "65";"64"
\ar^{\begin{smallmatrix} 3 \end{smallmatrix}} "75";"74"
\ar^{\begin{smallmatrix} 3 \end{smallmatrix}} "44";"43" 
\ar_{\begin{smallmatrix} 2 \end{smallmatrix}} "44";"32"
\ar^{\begin{smallmatrix} 1 \end{smallmatrix}} "54";"43"
\ar^{\begin{smallmatrix} 2\!\!& \\ &3 \end{smallmatrix}} "54";"53" 
\ar^{\begin{smallmatrix} &1\!\!&& \\ 2\!\!&&2\!\!& \\ &&&3 \end{smallmatrix}} "64";"63"
\ar^{\begin{smallmatrix} &1\!\!&&1\!\!&& \\ 2\!\!&&2\!\!&&2\!\!& \\ &&&&&3 \end{smallmatrix}} "74";"73"  
\ar^{\begin{smallmatrix} 2\!\!& \\ &3 \end{smallmatrix}} "43";"42" 
\ar^{\begin{smallmatrix} 2 \end{smallmatrix}} "53";"64" 
\ar^{\begin{smallmatrix} 1\!\!&& \\ &2\!\!& \\ &&3 \end{smallmatrix}} "53";"52" 
\ar^{\begin{smallmatrix} 1\!\!&& \\ &2\!\!& \\ &&3 \end{smallmatrix}} "63";"62"
\ar^{\begin{smallmatrix} &1\!\!& \\ 2\!\!&&2 \end{smallmatrix}} "63";"74"
\ar^{\begin{smallmatrix} 1\!\!&& \\ &2\!\!& \\ &&3 \end{smallmatrix}} "73";"72"  
\ar^{\begin{smallmatrix} 1\!\!&&1 \\ &2\!\!& \end{smallmatrix}} "12";"22"
\ar^{\begin{smallmatrix} 3 \end{smallmatrix}} "12";"11"
\ar^{\begin{smallmatrix} 1 \end{smallmatrix}} "22";"32"
\ar^{\begin{smallmatrix} 3 \end{smallmatrix}} "22";"21"
\ar^{\begin{smallmatrix} 3 \end{smallmatrix}} "32";"31"
\ar_{\begin{smallmatrix} 2 \end{smallmatrix}} "42";"31"
\ar_{\begin{smallmatrix} 1 \end{smallmatrix}} "52";"42"
\ar^{\begin{smallmatrix} 2 \end{smallmatrix}} "52";"62" 
\ar^{\begin{smallmatrix} &1\!\!& \\ 2\!\!&&2 \end{smallmatrix}} "62";"72"
\ar^{\begin{smallmatrix} 1\!\!&&1 \\ &2\!\!& \end{smallmatrix}} "11";"21"
\ar^{\begin{smallmatrix} 1 \end{smallmatrix}} "21";"31"
\ar (-119, 27);"12"
\ar (-119,-27);"11"
\ar "75";(94.5, 27)
\ar "73";(94.5, -3)
\ar "72";(94.5,-27)
\end{xy}
}
\end{align*}
\caption{The exchange quivers of $\sttilt A$ and $\stitilt A$}
\label{A2}
}
\end{figure}

On Condition \ref{wide_fwide},
we prove more general properties for tilted algebras.
In the rest, we assume that $K$ is an algebraically closed field.

\begin{Th}\label{tilt_th}
Let $H$ be a hereditary algebra,
$T \in \mod H$ be a tilting module,
and $A:=\End_H(T)$.
Then the following assertions hold.
\begin{itemize}
\item[(1)]
If $\Sub_H \tau T$ has only finitely many indecomposable $H$-modules, 
then $\fwide A = \fLwide A$.
If $\Fac_H T$ has only finitely many indecomposable $H$-modules, 
then $\fwide A = \fRwide A$.
\item[(2)]
If $T$ is either preprojective or preinjective, 
then $\fwide A=\fLwide A=\fRwide A$.
\item[(3)]
Assume that $H$ is a hereditary algebra of extended Dynkin type.
We decompose $T$ as $T_{\ppr} \oplus T_{\reg} \oplus T_{\pin}$ with  
a preprojective module $T_{\ppr}$, 
a regular module $T_{\reg}$, 
and a preinjective module $T_{\pin}$.
If $T_{\reg} \ne 0$ and $T_{\pin} = 0$, 
then we have $\fwide A \supsetneq \fRwide A$, and
if $T_{\reg} \ne 0$ and $T_{\ppr} = 0$, 
then we have $\fwide A \supsetneq \fLwide A$.
\end{itemize}
\end{Th}

Now, we begin the proof of the theorem.
For additive full subcategories $\cC_1,\cC_2 \subset \mod A$, 
the notation $\cC_1*\cC_2$ denotes
the full subcategory of $\mod A$ consisting of all $M$
such that there exists a short exact sequence $0 \to M_1 \to M \to M_2 \to 0$ 
with $M_1 \in \cC_1$ and $M_2 \in \cC_2$.

\begin{Lem}\label{fwide_ok}
Let $(\cX,\cY)$ be a torsion pair in $\mod A$ satisfying the following conditions:
\begin{itemize}
\item the torsion pair $(\cX,\cY)$ is splitting,
\item for any $N_1,N_2 \in \cY$, we have $\Ext_A^2(N_1,N_2)=0$,
\item the torsion class $\cX$ has only finitely many indecomposable $A$-modules.
\end{itemize}
Then we have $\fwide A = \fLwide A$.
\end{Lem}

\begin{proof}
We note that $\Ext_A^1(N,M)=0$ holds for any $M \in \cX$ and any $N \in \cY$,
because the torsion pair $(\cX,\cY)$ is splitting.

Let $\cW \in \fwide A$ and set $\cW_\cX:=\cW \cap \cX$ and $\cW_\cY:=\cW \cap \cY$.
It is sufficient to prove that $\sT(\cW) \in \ftors A$.

We first claim that $\Fac \cW_\cX * \Fac \cW_\cY
\subset \Fac \cW_\cY * \Fac \cW_\cX$ holds.
Let $L$ belong to $\Fac \cW_\cX * \Fac \cW_\cY$.
There exists a short exact sequence 
$0 \to L_1 \xrightarrow{f} L \xrightarrow{g} L_2 \to 0$ 
such that
$L_1 \in \Fac \cW_\cX$ and $L_2 \in \Fac \cW_\cY$.
Because $L_2 \in \Fac \cW_\cY$, 
there exists a surjection $h \colon N \to L_2$ with $N \in \cW_\cY$.
We have $N \in \cY$ and 
$L_1 \in \Fac \cW_\cX \subset \cX$, because $\cX$ is a torsion class.
Thus, we have $\Ext_A^1(N,L_1)=0$, 
so $\Hom_A(N,g) \colon \Hom_A(N,L) \to \Hom_A(N,L_2)$ is surjective.
There exists $h' \colon N \to L$ such that $h=gh'$. 
We have the following commutative diagram
with the rows exact:
\begin{align*}
\begin{xy}
(  0, 16) *+{0} = "01",
( 16, 16) *+{0} = "11",
( 32, 16) *+{N} = "21",
( 48, 16) *+{N} = "31",
( 64, 16) *+{0} = "41", 
(  0,  0) *+{0} = "00",
( 16,  0) *+{L_1} = "10",
( 32,  0) *+{L} = "20",
( 48,  0) *+{L_2} = "30",
( 64,  0) *+{0} = "40", 
\ar "01"; "11" \ar "11"; "21" \ar@{=} "21"; "31" \ar "31"; "41"
\ar "00"; "10" \ar^{f} "10"; "20" \ar^{g} "20"; "30" \ar "30"; "40"
\ar "11"; "10" \ar^{h'} "21"; "20" \ar^{h} "31"; "30"
\end{xy}.
\end{align*}
We have exact sequences
$L_1 \to \Coker h' \to \Coker h = 0$ by Snake Lemma and
$0 \to \Im h' \to L \to \Coker h' \to 0$ by definition.
We can easily see $\Im h' \in \Fac N \subset \Fac \cW_\cY$ and
$\Coker h' \in \Fac L_1 \subset \Fac \cW_\cX$.
Thus, we have $L \in \Fac \cW_\cY * \Fac \cW_\cX$.
Therefore, the claim $\Fac \cW_\cX * \Fac \cW_\cY
\subset \Fac \cW_\cY * \Fac \cW_\cX$ is proved.

By assumption, every indecomposable module in $\cW$ belongs to $\cX$ or $\cY$. 
Thus, we have $\sT(\cW)=\Filt(\Fac \cW_\cX \cup \Fac \cW_\cY)$.
Because $\Fac \cW_\cX * \Fac \cW_\cY
\subset \Fac \cW_\cY * \Fac \cW_\cX$,
we have $\sT(\cW) = \sT(\cW_\cY) * \sT(\cW_\cX)$.
If $\sT(\cW_\cY)$ and $\sT(\cW_\cX)$ are functorially finite,
then $\sT(\cW)$ is functorially finite,
see \cite[Theorem 2.6]{SikS}.

We would like to show the functorially finiteness of $\sT(\cW_\cY)$.

We prove that 
$\sT(\cW_\cY)=\Fac \cW_\cY$ by a similar argument to \cite[Proposition 2.13]{IT}.
Let $0 \to L_1 \to L_2 \to L_3 \to 0$ be 
a short exact sequence with $L_1,L_3 \in \Fac \cW_\cY$.
It is sufficient to show that $L_2 \in \Fac \cW_\cY$.
By assumption, there exists a surjection $f_3 \colon N_3 \to L_3$
with $N_3 \in \cW_\cY$.
Taking the pull back, 
we have the following commutative diagram with the rows exact
and $f_2$ surjective:
\begin{align*}
\begin{xy}
(  0, 16) *+{0} = "01",
( 16, 16) *+{L_1} = "11",
( 32, 16) *+{L'_2} = "21",
( 48, 16) *+{N_3} = "31",
( 64, 16) *+{0} = "41", 
(  0,  0) *+{0} = "00",
( 16,  0) *+{L_1} = "10",
( 32,  0) *+{L_2} = "20",
( 48,  0) *+{L_3} = "30",
( 64,  0) *+{0} = "40", 
\ar "01"; "11" \ar "11"; "21" \ar "21"; "31" \ar "31"; "41"
\ar "00"; "10" \ar "10"; "20" \ar "20"; "30" \ar "30"; "40"
\ar@{=} "11"; "10" \ar^{f_2} "21"; "20" \ar^{f_3} "31"; "30"
\end{xy}.
\end{align*}
By assumption again, there exists a surjection $g_1 \colon N_1 \to L_1$
with $N_1 \in \cW_\cY$.
Because $\cY$ is a torsion-free class, we have $\Ker g_1 \in \cY$.
By assumption, we get $\Ext_A^2(N_3,\Ker g_1)=0$.
Thus, we have an exact sequence $\Ext_A^1(N_3,N_1) \to \Ext_A^1(N_3,L_1) \to 0$,
so there exists the following commutative diagram with the rows exact and $g_2$ surjective:
\begin{align*}
\begin{xy}
(  0, 16) *+{0} = "01",
( 16, 16) *+{N_1} = "11",
( 32, 16) *+{L''_2} = "21",
( 48, 16) *+{N_3} = "31",
( 64, 16) *+{0} = "41", 
(  0,  0) *+{0} = "00",
( 16,  0) *+{L_1} = "10",
( 32,  0) *+{L'_2} = "20",
( 48,  0) *+{N_3} = "30",
( 64,  0) *+{0} = "40", 
\ar "01"; "11" \ar "11"; "21" \ar "21"; "31" \ar "31"; "41"
\ar "00"; "10" \ar "10"; "20" \ar "20"; "30" \ar "30"; "40"
\ar^{g_1} "11"; "10" \ar^{g_2} "21"; "20" \ar@{=} "31"; "30"
\end{xy}.
\end{align*}
Here, we have $L''_2 \in \cW_\cY$, because $\cW_\cY$ is extension closed.
Since $f_2$ and $g_2$ are surjective, $f_2g_2 \colon L''_2 \to L_2$ is surjective,
so we have $L_2 \in \Fac \cW_\cY$.
It is now proved that $\sT(\cW_\cY)=\Fac \cW_\cY$.

We claim that $\cW_\cY$ is covariantly finite in $\mod A$.
Let $L \in \mod A$.
There exists a left $\cW$-approximation $f \colon L \to M \oplus N$
with $M \in \cW_\cX$ and $N \in \cW_\cY$,
because $\cW$ is functorially finite and $(\cX,\cY)$ is splitting.
Compose the projection $p \colon M \oplus N \to N$,
then we have a left $\cW_\cY$-approximation $pf \colon L \to N$,
because $\Hom_A(\cX,\cY)=0$.
Therefore, $\cW_\cY$ is covariantly finite.

Now we can take a left $\cW_\cY$-approximation $A \to N$ of $A$.
Then we get $N \in \cW_\cY \subset \Fac N$ and $\Fac \cW_\cY=\Fac N$.
By \cite[Proposition 4.6]{AS1},
$\sT(\cW_\cY)=\Fac \cW_\cY=\Fac N$ is functorially finite.

On the other hand, it is clear that $\sT(\cW_\cX)$ is contained in $\cX$,
so $\sT(\cW_\cX)$ has only finitely many indecomposable modules.
Thus, $\sT(\cW_\cX)$ is functorially finite in $\mod A$.

We finally get that $\sT(\cW)=\sT(\cW_\cY)*\sT(\cW_\cX)$ is functorially finite.
\end{proof}

We can show Theorem \ref{tilt_th}.

\begin{proof}[Proof of Theorem \ref{tilt_th}]
(1)
We consider the first statement.
The other one is shown by taking $K$-duals.

The tilting $H$-module $T$
induces a torsion pair $(\Fac_H T,\Sub_H \tau T)$ in $\mod A$ and 
a torsion pair $(\Fac_A D(\tau T),\Sub_A DT)$ in $\mod A$.
We check that the torsion pair $(\Fac_A D(\tau T),\Sub_A DT)$ satisfies 
the conditions in Lemma \ref{fwide_ok}.

First, since $H$ is hereditary, 
$(\Fac_A D(\tau T),\Sub_A DT)$ is a splitting torsion pair in $\mod A$
by \cite[VI.5.7.\ Corollary]{ASS}.

Second, by \cite[VI.4.1.\ Lemma]{ASS},
the projective dimension of every module in $\Sub_A DT$ is at most one.
Thus, we have $\Ext_A(N_1,N_2)=0$ for any $N_1,N_2 \in \Sub_A DT$.

Third, 
the torsion-free class $\Sub_H \tau T$ in $\mod H$ 
has only finitely many indecomposable $H$-modules by assumption.
Thus, the torsion class $\Fac_A D(\tau T)$ in $\mod A$ 
has only finitely many indecomposable $A$-modules.

Thus, the conditions in Lemma \ref{fwide_ok} hold for $(\Fac_A D(\tau T),\Sub_A DT)$.
Thus, $\fwide A=\fLwide A$.

(2)
We consider the case that $T$ is preprojective.
The other case is shown similarly.

Since $T$ is preprojective, the torsion-free class $\Sub_H \tau T$ has
only finitely many indecomposable $H$-modules \cite[VIII.2.5.\ Lemma]{ASS}.
Thus, we get $\fwide A=\fLwide A$ by (1).

On the other hand, 
there exists some preinjective tilting $H$-module $T'$ with $A \cong \End_H(T')$. 
Since $T$ is preinjective, the torsion class $\Fac_H T'$ has
only finitely many indecomposable $H$-modules.
Thus, we also have $\fwide A=\fRwide A$ by (1).

(3)
We consider the first statement.
The other one is shown by taking $K$-duals.
We fully refer to \cite{SimS}.

By assumption, there exists a regular stable tube $\cC$ 
in the Auslander--Reiten quiver of $\mod H$
containing an indecomposable direct summand of $T_{\reg}$.
Let $V$ be the direct sum of the indecomposable direct summands of $T$ in $\cC$,
and decompose $T$ as $U \oplus V$.
We have $\Hom_H(V,U)=0$, because $H$ is a hereditary algebra of extended Dynkin type.

Now we define the \textit{cone} 
for each indecomposable $H$-module in the stable tube $\cC$ as in \cite{SimS}.
Let $r$ be the rank of the stable tube $\cC$.
Then $\cC$ is isomorphic to $\Z \mathbb{A}_\infty/\langle \tau^r \rangle$
as translation quivers.

We may assume that 
the set of vertices of $\Z \mathbb{A}_\infty/\langle \tau^r \rangle$
is $\{ 1,2,3,\ldots \} \times (\Z/r\Z)$,
and that the set of arrows is
\begin{align*}
&\{(a,b+r\Z) \to (a+1,b+r\Z) \mid a \in \{ 1,2,3,\ldots \}, \ b \in \{0,1,\ldots,r-1\} \}\\
&\cup \{(a+1,b+r\Z) \to (a,b+1+r\Z) \mid a \in \{ 1,2,3,\ldots \}, \ b \in \{0,1,\ldots,r-1\} \}.
\end{align*}
Fix an isomorphism $\cC \to \Z \mathbb{A}_\infty/\langle \tau^r \rangle$
and identify $\cC$ with $\Z \mathbb{A}_\infty/\langle \tau^r \rangle$ by this isomorphism.

Let $W$ be an indecomposable module in $\cC$ and 
$(a,b+r\Z)$ be its position,
then the cone $\Cone W$ is defined as the set
of indecomposable modules in $\cC$ located in 
\begin{align*}
\{ (c,b+d+r\Z) \mid c \in \{ 1,2,\ldots,a \}, \ d \in \{0,1,\ldots,a-c \}\}.
\end{align*}
For example, if $a=3$, this set is pictured as follows:
\begin{align*}
\begin{xy}
(  0, 10)*+{(3,b  +r\Z)}="30",
(-16,  0)*+{(2,b  +r\Z)}="20",
( 16,  0)*+{(2,b+1+r\Z)}="21",
(-32,-10)*+{(1,b  +r\Z)}="10",
(  0,-10)*+{(1,b+1+r\Z)}="11",
( 32,-10)*+{(1,b+2+r\Z)}="12",
\ar "10"; "20"
\ar "20"; "30"
\ar "30"; "21"
\ar "21"; "12"
\ar "20"; "11"
\ar "11"; "21"
\end{xy}.
\end{align*}

Now, let $V_1,\ldots,V_p$ be the distinct elements of $\ind_H V$.
They are indecomposable $H$-modules in $\cC$.
By \cite[XVII.1.7.\ Lemma]{SimS}, 
if $i \ne j$, then $\Cone V_i \subsetneq \Cone V_j$ or $\Cone V_i \supsetneq \Cone V_j$
or $\Cone V_i \cap \Cone V_j=\emptyset$ holds.
Thus, we may assume that 
$\Cone V_1 \supsetneq \Cone V_j$ or $\Cone V_1 \cap \Cone V_j=\emptyset$ holds
for $j=2,\ldots,p$.

We define an idempotent $e \in A$ as $U \oplus V \to V \to U \oplus V$.
Since $\Hom_H(V,U)=0$, 
we have $A/\langle e \rangle \cong \End_H(U)$
as $K$-algebras.
By \cite[XVII.2.3.\ Theorem]{SimS} (c), 
the indecomposable $A$-modules in $\Hom_H(U,\cC \cap \Fac_H T \cap V^\perp)$,
where $V^\perp$ is considered in $\mod A/ \langle e \rangle$, forms
a standard stable tube $\cC'$ in the Auslander--Reiten quiver of 
$\mod A/\langle e \rangle$.
Let $M_1,\ldots,M_m$ be all the distinct $A/\langle e \rangle$-modules 
in the mouth of the stable tube $\cC'$.

In this setting, $P_1:=\Hom_H(T,V_1)$ is a projective $A$-module.
If $V_1$ is located on $(a,b+r\Z)$ in $\cC$, 
we consider the $H$-module $W_1$ corresponding to the vertex $(a+1,b-1+r\Z)$.
From the proof of \cite[XVII.2.3.\ Theorem]{SimS},
we obtain the following properties: 
\begin{itemize}
\item 
$\Hom_H(T,W_1)$ is an $A/\langle e \rangle$-module 
lying in the mouth of the stable tube $\cC'$,
\item
there exists an $A/\langle 1-e \rangle$-module $N$ such that 
$\rad_A P_1 = \Hom_H(T,W_1) \oplus N$.
\end{itemize}
Therefore, we may assume $M_1=\Hom_H(T,W_1)$,
and we obtain that 
$M_1=\Hom_H(T,W_1)$ coincides with the maximum $A$-submodule $X$ of $P_1$ satisfying
$X \in \mod A/\langle e \rangle$, or equivalently, $M_1=\Hom_A(A/\langle e \rangle,P_1)$.

Now, we claim that the set $\cS:=\{P_1,M_2,\ldots,M_m \}$ 
is a semibrick in $\mod A$.
First, by a property of standard stable tubes,
we have $\{ M_2,\ldots,M_m \} \in \sbrick A$.
We also have $\Hom_A(P_1,M_i)=0$, since $M_i$ is an $A/\langle e \rangle$-module.
Next, $\Hom_A(M_i,P_1) \cong \Hom_A(M_i,M_1)=0$ holds for $i=2,\ldots,m$, 
because $M_i \in \mod A/\langle e \rangle$ and $M_1=\Hom_A(A/\langle e \rangle,P_1)$.
It remains to show $P_1 \in \brick A$.
Since $V_1$ is a direct summand of $V$, 
it satisfies $\Ext_H^1(V_1,V_1)=0$;
hence, $V_1 \in \brick H$ by \cite[XVII.1.6.\ Lemma]{SimS}.
Thus, $P_1 \in \brick A$.
Therefore, we have $\cS \in \sbrick A$.
We can consider a wide subcategory $\cW=\Filt_A \cS$ of $\mod A$.

We prove that $\cW \in \fwide A$.
We have $\cW=\add_A P_1*\Filt_A(M_2,\ldots,M_m)$,
since $P_1$ is projective.
By construction, $M_2,\ldots,M_m$ are $A/\langle e \rangle$-modules 
in the mouth of $\cC'$
and there also exists $M_1$ in the mouth.
Thus, a wide subcategory 
$\Filt_A(M_2,\ldots,M_m)$ belongs to $\fwide A/\langle e \rangle$;
hence, $\Filt_A(M_2,\ldots,M_m) \in \fwide A$.
On the other hand, $\add_A P_1$ is obviously functorially finite in $\mod A$.
Thus, we get that $\cW \in \fwide A$ by \cite[Theorem 2.6]{SikS}.

We finally prove that $\cW \notin \fRwide A$, or equivalently, 
$\sF_A(\cW) \notin \ftorf A$.
We assume that $\sF_A(\cW) \in \ftorf A$,
and deduce a contradiction.
This assumption implies that $\sF_A(\cW) \cap \mod A/\langle e \rangle$ 
is functorially finite in $\mod A/\langle e \rangle$.
The bricks $M_2,\ldots,M_m$ already belong to $\mod A/\langle e \rangle$,
and $M_1=\Hom_A(A/\langle e \rangle,P_1)$ holds.
Therefore, we have 
$\sF_A(M_1,\ldots,M_m)=\sF_A(\cW) \cap \mod A/\langle e \rangle 
\in \ftorf A/\langle e \rangle$;
hence, a wide subcategory $\Filt_A(M_1,\ldots,M_m)$ must belong to 
$\fwide A/\langle e \rangle$ by Proposition \ref{left_func}.
However, this is a contradiction, because
$M_1,\ldots,M_m$ are all the $A/\langle e \rangle$-modules 
in the mouth of the stable tube $\cC'$.
Thus, we obtain that $\sF_A(\cW) \notin \ftorf A$.

The proof is now complete.
\end{proof}

Finally, we obtain the following result for tilted algebras $A$ of extended Dynkin type.

\begin{Cor}\label{ex_Dynkin_table}
We use the setting of Theorem \ref{tilt_th} (3).
Then, the tilting $H$-module $T$ satisfies one of the conditions (1)--(5) 
in the following table, 
which shows whether $\fwide A = \fLwide A$ and $\fwide A = \fRwide A$
hold in each case.
\textup{
\begin{center}
\begin{tabular}{c|ccc|cc}
No. & $T_{\ppr}$ & $T_{\reg}$ & $T_{\pin}$ &
$\fwide A = \fLwide A$ & $\fwide A = \fRwide A$ \\
\hline
(1) & $\ne 0$ & $=   0$ & $=   0$ & Yes & Yes \\
(2) & $=   0$ & $=   0$ & $\ne 0$ & Yes & Yes \\
(3) & $\ne 0$ & $\ne 0$ & $=   0$ & Yes & No  \\
(4) & $=   0$ & $\ne 0$ & $\ne 0$ & No  & Yes \\
(5) & $\ne 0$ &         & $\ne 0$ & Yes & Yes 
\end{tabular}
\end{center}
}
\end{Cor}

\begin{proof}
Since there never exists a regular tilting $H$-module \cite[XVII.3.4.\ Lemma]{SimS},
$T$ satisfies one of the conditions (1)--(5).

In the cases (1) and (2),
both of the conditions $\fwide A = \fLwide A$ and $\fwide A = \fRwide A$ hold 
by Theorem \ref{tilt_th} (2).

Next, we consider the case (3).
By Step $3^\circ$ of the proof of \cite[XVII.3.5.\ Theorem]{SimS}, 
$\Sub_H \tau T$ has only finitely many indecomposable $H$-modules.
Thus, we can apply Theorem \ref{tilt_th} (1) and obtain $\fwide = \fLwide A$.
On the other hand, Theorem \ref{tilt_th} (3) implies $\fwide \ne \fRwide A$.

Similarly, in the case (4), we get $\fwide = \fRwide A$ and $\fwide \ne \fLwide A$.

It remains to deal with the case (5). 
Then, $A$ is representation-finite \cite[XVII.3.3.\ Lemma]{SimS}.
Thus, we have both $\fwide A = \fLwide A$ and $\fwide A = \fRwide A$.
\end{proof}

\section*{Funding}

This work was supported by Japan Society for the Promotion of Science KAKENHI [JP16J02249].

\section*{Acknowledgement}

The author would like to thank his supervisor Osamu Iyama 
for giving him many interesting topics and thorough instructions.
He also thanks Laurent Demonet, 
Jan \v{S}t\!'ov\'{i}\v{c}ek, Takahide Adachi, and Aaron Chan for valuable discussions.
Moreover, he is grateful to the anonymous referees 
for helping him correct his mistakes and make arguments easier to understand.

\end{document}